\documentclass{ip-journal}

\usepackage{amsmath, amsthm, amssymb}
\input xy
\xyoption{all}
\usepackage{mathrsfs}
\usepackage{enumerate}
\usepackage{caption}
\usepackage{subcaption}
\usepackage{graphicx}
\usepackage{color}
\usepackage{tikz}
\usepackage{tkz-euclide}
\usepackage[driverfallback=hypertex]{hyperref}
\usepackage{nameref,zref-xr}                    %to include reference to other papers

\usepackage{cancel}
\usepackage[normalem]{ulem}

\newtheorem{dummy}{}[section]
\newtheorem{thm}[dummy]{Theorem}
\newtheorem{prop}[dummy]{Proposition}
\newtheorem{pr}[dummy]{Proposition}
\newtheorem{lem}[dummy]{Lemma}

\newtheorem{lemma}[dummy]{Lemma}

\newtheorem{conj}{Conjecture}
\theoremstyle{definition}
\newtheorem{definition}[dummy]{Definition}
\newtheorem{nn}[dummy]{Notation}

\theoremstyle{remark}
\newtheorem{rmk}[dummy]{Remark}
\newtheorem{ex}[dummy]{Example}
\newtheorem{obs}[dummy]{Observation}

\newcommand{\stronglypositive}{{strongly positive }}
\newcommand{\CM}{{\mathcal{M}}}

\newcommand{\oCM}{{\overline{\mathcal{M}}}}

\newcommand{\oCMr}{{{\overline{\mathcal{M}}}^{\frac{1}{r}}_{0,k,\vec{a}}}}
\newcommand{\CMr}{{{{\mathcal{M}}}^{\frac{1}{r}}_{0,k,\vec{a}}}}
\newcommand{\Gammar}{{\Gamma_{0,k,\vec{a}}}}

\newcommand{\CC}{{\mathcal{C}}}

\newcommand{\CL}{{\mathbb{L}}}
\newcommand{\Pos}{{H^+}}
\newcommand{\alt}{{\text{alt}}}
\newcommand{\TAU}{{\wp}}
\newcommand{\detach}{{\text{detach} }}
\newcommand{\LLL}{{\mathfrak{L}}}
\newcommand{\CS}{{\mathcal{S}}}
\newcommand{\CB}{{\mathcal{B}}}

\newcommand{\s}{\mathbf{s}}
\newcommand{\srest}{\mathbf{s}_1}
\newcommand{\tsrest}{{\mathbf{s}}_1}
\newcommand{\Aut}{\text{Aut}}

\newcommand{\ttt}{t}
\newcommand{\tildet}{\ttt}
\newcommand{\ttts}{\mathbf{t}}
\newcommand{\cS}{\mathcal{S}}

\newcommand{\cW}{\mathcal{W}}

\newcommand{\tw}{\text{tw}}

\newcommand{\NNN}{{n}}

\newcommand{\N}{\mathbb{N}}

\newcommand{\TTT}{{\mathcal{T}}}

\newcommand{\TRAM}{{\overline{\mathcal{R}}}}

\newcommand{\Ass}{{{\maltese}}}
\newcommand{\Detach}{{{\text{Detach}}}}
\newcommand{\Conn}{{{\text{Conn}}}}

\newcommand{\oPM}{{\overline{\mathcal{PM}}}^{1/r}}
\newcommand{\PM}{{{\mathcal{PM}}}^{1/r}}

\newcommand{\oPMr}{{{\overline{\mathcal{PM}}}^{\frac{1}{r}}_{0,k,\vec{a}}}}

\newcommand{\Z}{\ensuremath{\mathbb{Z}}}

\newcommand{\C}{\ensuremath{\mathbb{C}}}
\newcommand{\R}{\ensuremath{\mathbb{R}}}

\newcommand{\M}{\ensuremath{\overline{\mathcal{M}}}}
\renewcommand{\O}{\ensuremath{\mathcal{O}}}

\newcommand{\ev}{\ensuremath{\textrm{ev}}}

\renewcommand{\d}{\ensuremath{\partial}}
\renewcommand{\subset}{\ensuremath{\subseteq}}
\newcommand{\eps}{\varepsilon}
\newcommand{\<}{\left<}
\renewcommand{\>}{\right>}
\newcommand{\mbZ}{\mathbb{Z}}
\newcommand{\mbC}{\mathbb{C}}
\DeclareMathOperator{\res}{res}
\newcommand{\Coef}{\mathrm{Coef}}

\newcommand{\tF}{\widetilde{F}}

\DeclareMathOperator{\rk}{rank}

\DeclareMathOperator{\tdeg}{deg}

\numberwithin{equation}{section}

\begin{document}

\title[Open $r$-spin theory II]{Open $r$-spin theory II: The analogue of Witten's conjecture for $r$-spin disks}

\author{Alexandr Buryak}
\address[A. Buryak]{Faculty of Mathematics, National Research University Higher School of Economics, 6 Usacheva str., Moscow, 119048, Russian Federation; \smallskip\newline Center for Advanced Studies, Skolkovo Institute of Science and Technology, 1 Nobel str., Moscow, 143026, Russian Federation}
\email{aburyak@hse.ru}

\author{Emily Clader}
\address[E.~Clader]{San Francisco State University, San Francisco, CA 94132-1722, USA}
\email{eclader@sfsu.edu}

\author{Ran J. Tessler}
\address[R.~J.~Tessler]{Incumbent of the Lilian and George Lyttle Career Development Chair, Department of Mathematics, Weizmann Institute of Science, POB 26, Rehovot 7610001, Israel}
\email{ran.tessler@weizmann.ac.il}

\begin{abstract}
We conclude the construction of $r$-spin theory in genus zero for Riemann surfaces with boundary.  In particular, we define open $r$-spin intersection numbers, and we prove that their generating function is closely related to the wave function of the $r$th Gelfand--Dickey integrable hierarchy.  This provides an analogue of Witten's $r$-spin conjecture in the open setting and a first step toward the construction of an open version of Fan--Jarvis--Ruan--Witten theory.  As an unexpected consequence, we establish a mysterious relationship between open $r$-spin theory and an extension of Witten's closed theory.
\end{abstract}

\maketitle

\setcounter{tocdepth}{1}
%\tableofcontents

\section{Introduction}
In the study of the intersection theory of the moduli space of stable curves, one of the most important modern results is Witten's conjecture \cite{Witten2DGravity}, which was proven by Kontsevich~\cite{Kontsevich}. Let $\psi_1, \ldots, \psi_n \in H^2(\M_{g,n})$ be the first Chern classes of the cotangent line bundles at the $n$ marked points, and let
\[
F^c(t_0,t_1,\ldots,\eps) := \sum_{\substack{g \geq 0, n\geq 1\\2g-2+n>0}} \sum_{d_1,\ldots,d_n\geq 0} \frac{\eps^{2g-2}}{n!}\left(\int_{\M_{g,n}} \psi_1^{d_1} \cdots \psi_n^{d_n} \right)t_{d_1} \cdots t_{d_n}
\]
be the generating function of their intersection numbers. Here,~$\{t_i\}_{i\ge 0}$ and~$\eps$ are formal variables and the superscript ``$c$," which stands for ``closed," is to contrast with the open theory discussed below.  Witten's conjecture states that $\exp(F^c)$ is a tau-function of the Korteweg--de Vries (KdV) hierarchy, or equivalently, that $\exp(F^c)$ satisfies a certain collection of linear differential equations known as the Virasoro equations. This uniquely determines all $\psi$-integrals on $\M_{g,n}$.

Witten also proposed a generalization of his conjecture \cite{Witten93}, in which the moduli space of curves is enhanced to the moduli space of $r$-spin structures.  On a smooth marked curve $(C;z_1, \ldots, z_n)$, an $r$-spin structure is a line bundle~$S$ together with an isomorphism
\[S^{\otimes r} \cong \omega_{C}\left(-\sum_{i=1}^n a_i[z_i]\right),\]
where $a_i \in \{0,1,\ldots, r-1\}$ and $\omega_C$ denotes the canonical bundle.  There is a natural compactification $\M_{g,\{a_1, \ldots, a_n\}}^{1/r}$ of the moduli space of $r$-spin structures on smooth curves, and this space admits a virtual fundamental class $c_W$ known as Witten's class. In genus zero, Witten's class is defined by
\begin{gather}\label{eq:Witten's class}
c_W:= e((R^1\pi_*\mathcal{S})^{\vee}),
\end{gather}
where $\pi: \mathcal{C} \rightarrow \M_{0,\{a_1, \ldots, a_n\}}^{1/r}$ is the universal curve, $\mathcal{S}$ is the universal $r$-spin structure, and $e(\cdot)$ denotes the top Chern class.  In higher genus, on the other hand, $R^1\pi_*\mathcal{S}$ may not be a vector bundle, and the definition of Witten's class is much more intricate; see \cite{PV,ChiodoWitten,Moc06,FJR,CLL} for various constructions.

The {\it closed $r$-spin intersection numbers} are defined by
\begin{gather}\label{eq:closed r-spin}
\<\tau^{a_1}_{d_1}\cdots\tau^{a_n}_{d_n}\>^{\frac{1}{r},c}_g:=r^{1-g}\int_{\M^{1/r}_{g,\{a_1, \ldots, a_n\}}} \hspace{-1cm} c_W \cap \psi_1^{d_1} \cdots \psi_n^{d_n}.
\end{gather}
Witten's $r$-spin conjecture then states that, if $t^a_d$ are formal variables indexed by $0\le a\le r-1$ and $d\ge 0$, and if
\[
F^{\frac{1}{r},c}(t^*_*,\eps):=\sum_{\substack{g \geq 0, n \geq 1\\2g-2+n>0}} \sum_{\substack{0 \leq a_1, \ldots, a_n \leq r-1\\ d_1, \ldots, d_n \geq 0}} \frac{\eps^{2g-2}}{n!}\<\tau^{a_1}_{d_1}\cdots\tau^{a_n}_{d_n}\>^{\frac{1}{r},c}_g t^{a_1}_{d_1} \cdots t^{a_n}_{d_n}
\]
is the generating function of the $r$-spin intersection numbers, then $\exp(F^{\frac{1}{r},c})$ becomes, after a simple change of variables, a tau-function of the $r$th Gelfand--Dickey hierarchy.  This result was proven by Faber--Shadrin--Zvonkine \cite{FSZ10}.

A new direction in the intersection theory of the moduli spaces of curves was initiated by Pandharipande, Solomon, and the third author in \cite{PST14}, studying the moduli space of Riemann surfaces with boundary.  In that work, a moduli space $\M_{0,k,l}$ was considered, which parameterizes tuples consisting of a stable disk $\Sigma$, boundary marked points $x_i \in \d\Sigma$, and internal marked points $z_j \in \Sigma \setminus \d\Sigma$.  Furthermore, intersection numbers on $\M_{0,k,l}$, which can be viewed as integrals of $\psi$-classes at the internal marked points,  were constructed and hence a generating function $F^o_0(t_0,t_1,\ldots,s)$ was defined as a direct generalization of the genus-zero part of $F^c$; the new formal variable $s$ tracks the number of boundary marked points. This construction has been extended in \cite{ST1} to all genera, and the complete open potential $F^o(t_0,t_1,\ldots,s,\eps)$ has been constructed. Verifying a conjecture of~\cite{PST14}, the first and third authors proved that $F^o$ satisfies certain ``open Virasoro equations"~\cite{BT17}, and moreover, the first author proved that these equations imply that $\exp(F^o)$ is explicitly related to the wave function of the KdV hierarchy~\cite{Bur16}.  These results provide an open analogue of Witten's conjecture on $\M_{g,n}$. An introduction to these results in more physical language can be found in \cite{DijkWit}.

A natural question, then, is what the open analogue of Witten's $r$-spin conjecture should be.  In order to make sense of the conjecture, one first must define an appropriate open $r$-spin moduli space~$\M_{g,k,\{a_1, \ldots, a_l\}}^{1/r}$ and an open analogue of Witten's bundle $(R^1\pi_*\mathcal{S})^{\vee}$.  We carried out this construction in genus zero in \cite{BCT1}, leading to a definition of the moduli space of {\it graded $r$-spin disks} $\oCMr$ as well as the open Witten bundle $\mathcal{W}$ and cotangent line bundles $\mathbb{L}_1, \ldots, \mathbb{L}_{l}$ at the internal marked points.

In this work, we build on the foundations of \cite{BCT1} to define {\it open $r$-spin intersection numbers},\footnote{The construction from~\cite{PST14} is recovered as a special case, when $r=2$ and all $a_i$ are zero.} which is a subtle task because the moduli space has boundary and hence the integration of top Chern classes is not well-defined until one finds a canonical way to prescribe the boundary behavior of sections.  The resulting intersection numbers are defined by

% \subsection{Genus-zero open $r$-spin intersection theory}

% In this work, we build on the foundations of \cite{BCT1} to define {\it open $r$-spin intersection numbers}.\footnote{The construction from~\cite{PST14} is recovered as a special case, when $r=2$ and all $a_i$ are zero.} As in~\cite{PST14}, this is a subtle task because the open moduli space has boundary, so the integration of the top Chern classes of the open Witten bundle and the cotangent line bundles over the open moduli space is not well-defined; to define intersection numbers, one must find a canonical way to prescribe the boundary behavior of sections of the bundles.  Additionally, one must resolve various orientation issues, constructing orientations for both the moduli space and the Witten bundle such that the Witten bundle becomes canonically relatively oriented.

% After addressing these issues, the resulting intersection numbers are defined~by

\begin{equation}
\label{eq:intnums}
\<\tau_{d_1}^{a_1}\cdots\tau^{a_l}_{d_l}\sigma^k\>^{\frac{1}{r},o}_0:=\int_{\oPM_{0,k,\{a_1, \ldots, a_l\}}} e\left( \mathcal{W} \oplus \bigoplus_{i=1}^l \mathbb{L}_i^{\oplus d_i}, s_{\text{canonical}}\right),
\end{equation}
where, on a manifold with boundary $M$, the notation $e(E,s)$ denotes the Euler class of $E$ relative to the section $s$ of $E|_{\d M}$.  The section $s_{\text{canonical}}$ is any choice of nowhere-vanishing boundary section of $\mathcal{W} \oplus \bigoplus_{i=1}^l\mathbb{L}_i^{\oplus d_i}$ that satisfies a ``canonicity" condition defined Section~\ref{sec:bc}; Theorem~\ref{thm:int_numbers_well_defined} shows that such $s_{\text{canonical}}$ indeed exists and that any two choices give rise to the same intersection numbers.  Finally, $\oPMr$ is a subspace of $\oCMr$ obtained by removing certain boundary strata; the restriction to this subspace makes defining canonicity more convenient without affecting the intersection numbers that result.

Equipped with these numbers, we define an open $r$-spin potential by
\[
F^{\frac{1}{r},o}_0(t^*_*,s):=\sum_{\substack{k,l\geq 0\\k+2l>2}} \sum_{\substack{0 \leq a_1, \ldots, a_l \leq r-1\\ d_1, \ldots, d_l \geq 0}} \frac{1}{k!l!}\<\tau^{a_1}_{d_1}\cdots\tau^{a_l}_{d_l}\sigma^k\>^{\frac{1}{r},o}_0 t^{a_1}_{d_1} \cdots t^{a_l}_{d_l}s^k.
\]
Before stating the main theorem regarding $F^{\frac{1}{r},o}$, we must recall some definitions from the theory of the Gelfand--Dickey hierarchy.

\subsection{Main results}

% Consider formal variables $T_i$ for $i\ge 1$. A {\it pseudo-differential operator}~$A$ is a Laurent series
% $$
% A=\sum_{n=-\infty}^m a_n\d_x^n,\quad a_n\in\mbC[\eps,\eps^{-1}][[T_1,T_2,\ldots]],
% $$
% where $m$ is an integer and $\d_x$ is a formal variable. The space of such operators is endowed with the structure of a non-commutative associative algebra, in which the multiplication is defined by the formula $\d_x^k\circ f:=\sum_{l=0}^\infty\frac{k(k-1)\ldots(k-l+1)}{l!}\frac{\d^lf}{\d x^l}\d_x^{k-l}$, where $k$ is an integer, $f\in\mbC[\eps,\eps^{-1}][[T_*]]$, and the variable $x$ is identified with $T_1$. For any $r\ge 2$ and any pseudo-differential operator~$A$ of the form $A=\d_x^r+\sum_{n=1}^\infty a_n\d_x^{r-n}$, there exists a unique pseudo-differential operator $A^{\frac{1}{r}}$ of the form $A^{\frac{1}{r}}=\d_x+\sum_{n=0}^\infty \widetilde{a}_n\d_x^{-n}$ such that $\left(A^{\frac{1}{r}}\right)^r=A$.

Consider formal variables $T_i$ for $i\ge 1$. A {\it pseudo-differential operator}~$A$ is a Laurent series
$$
A=\sum_{n=-\infty}^m a_n\d_x^n,\quad a_n\in\mbC[\eps,\eps^{-1}][[T_1,T_2,\ldots]],
$$
where $m$ is an integer and $\d_x$ is a formal variable. The space of such operators is endowed with the structure of a non-commutative associative algebra.  Moreover, for any $r\ge 2$ and any pseudo-differential operator~$A$ of the form $A=\d_x^r+\sum_{n=1}^\infty a_n\d_x^{r-n}$, there exists a unique pseudo-differential operator $A^{\frac{1}{r}}$ of the form $A^{\frac{1}{r}}=\d_x+\sum_{n=0}^\infty \widetilde{a}_n\d_x^{-n}$ such that $\left(A^{\frac{1}{r}}\right)^r=A$.

Let $r\ge 2$, and consider the operator
$$
L:=\d_x^r+\sum_{i=0}^{r-2}f_i\d_x^i,\quad f_i\in\mbC[\eps,\eps^{-1}][[T_*]].
$$
It is not hard to check that for any $n\ge 1$, the commutator $[(L^{n/r})_+,L]$ has the form $\sum_{i=0}^{r-2}h_i\d_x^i$ with $h_i\in\mbC[\eps,\eps^{-1}][[T_*]]$, where $(\cdot)_+$ denotes the part of a pseudo-differential operator with non-negative powers of $\d_x$. The {\it $r$th Gelfand--Dickey hierarchy} is the following system of partial differential equations for the functions $f_0,f_1,\ldots,f_{r-2}$:
\begin{gather}\label{eq:Gelfand-Dickey hierarchy}
\frac{\d L}{\d T_n}=\eps^{n-1}[(L^{n/r})_+,L],\quad n\ge 1.
\end{gather}

Consider the solution $L$ of this hierarchy specified by the initial condition
\begin{gather}\label{eq:initial condition for L}
L|_{T_{\ge 2}=0}=\d_x^r+\eps^{-r}rx.
\end{gather}
The $r$-spin Witten conjecture states that $\frac{\d F^{\frac{1}{r},c}}{\d t^{r-1}_d}=0$ for $d\ge 0$, and under the change of variables
\begin{gather}\label{eq:closed r-spin change of variables}
T_k=\frac{1}{(-r)^{\frac{3k}{2(r+1)}-\frac{1}{2}-d}k!_r}t^a_d,\quad 0\le a\le r-2,\quad d\ge 0,
\end{gather}
where $k=a+1+rd$ and $k!_r:=\prod_{i=0}^d(a+1+ri)$, we have
$$
\res L^{n/r}=\eps^{1-n}\frac{\d^2 F^{\frac{1}{r},c}}{\d T_1\d T_n}
$$
whenever $n\ge 1$ and $r\nmid n$, where $\res L^{n/r}$ denotes the coefficient of $\d_x^{-1}$ in~$L^{n/r}$.

With $L$ as above, let~$\Phi(T_*,\eps)$ be the solution of the system of equations
\begin{gather}\label{eq:equations for phi}
\frac{\d\Phi}{\d T_n}=\eps^{n-1}(L^{n/r})_+\Phi,\quad n\ge 1,
\end{gather}
that satisfies the initial condition $\left.\Phi\right|_{T_{\ge 2}=0}=1$. Denote $\phi:=\log\Phi$ and consider the expansion $\phi=\sum_{g\in\mbZ}\eps^{g-1}\phi_g$, $\phi_g\in\mbC[[T_*]]$. Comparing to the function~$F^{\frac{1}{r},c}_0$, which depends only on the variables $t^0_d,\ldots,t^{r-2}_d$, the function~$F^{\frac{1}{r},o}_0$ depends also on $t^{r-1}_d$ and $s$. So we relate the variables $T_{mr}$ and $t^{r-1}_{m-1}$ as follows:
\begin{gather}\label{eq:r-1 change}
T_{mr}=\frac{1}{(-r)^{\frac{m(r-2)}{2(r+1)}}m!r^m}t^{r-1}_{m-1},\quad m\ge 1.
\end{gather}

The main result of the current paper is the following:

\begin{thm}\label{thm:main}
We have
\begin{gather}\label{eq:main result}
F^{\frac{1}{r},o}_0=\frac{1}{\sqrt{-r}}\phi_0\big|_{t^{r-1}_d\mapsto \frac{1}{\sqrt{-r}}(t^{r-1}_d-r\delta_{d,0}s)}-\frac{1}{\sqrt{-r}}\phi_0\big|_{t^{r-1}_d\mapsto\frac{1}{\sqrt{-r}}t^{r-1}_d}.
\end{gather}
\end{thm}
\noindent This provides the open $r$-spin version of Witten's conjecture in genus zero.

% In the case $r=2$, the function $\Phi$ was first considered in~\cite{Bur16}, and the first author found there an explicit formula for the function $F^o$ in terms of the function $\phi=\log\Phi$. Properties of the function~$\Phi$ for general $r$ were first studied in~\cite{BY15}. In particular, note that the system of differential equations~\eqref{eq:equations for phi} for the function~$\Phi$ coincides with the system of differential equations for the wave function of the KP hierarchy; in~\cite{BY15}, the authors found an explicit formula for $\Phi$ in terms of the wave function.

The idea of the proof of Theorem~\ref{thm:main} is to verify that the genus-zero open $r$-spin intersection numbers satisfy certain geometric recursions (known as topological recursion relations) that allow one to determine all of the open theory using just a handful of numbers that can be explicitly calculated.  We then verify that the right-hand side of equation~\eqref{eq:main result} satisfies the same recursions with the same initial conditions.

In addition, our calculations yield a simple explicit formula for all genus-zero primary open intersection numbers (which, in particular, shows that their structure is much simpler than that of their closed analogues):

\begin{thm}\label{theorem:primary numbers}
Suppose $k,l\ge 0$ and $0\le a_1,\ldots a_l\le r-1$. Then
\begin{equation}
   \label{eq:primary open numbers} 
\<\prod_{i=1}^l\tau^{a_i}_0\sigma^k\>^{\frac{1}{r},o}_0= \displaystyle\frac{(k+l-2)!}{(-r)^{l-1}}
\end{equation}
if $k \geq 1$ and $\frac{1}{r}\big((r-2)(k-1)+2\sum a_i\big)=2l+k-3$, and the intersection number is zero if these conditions are not satisfied.
\end{thm}

% \noindent It is interesting to compare this result with the explicit formula for genus-zero primary closed $r$-spin intersection numbers~\cite{PPZ16}. One can immediately see that the structure of the primary open intersection numbers is much simpler than the structure of the primary closed intersection numbers.

A mirror-theoretic interpretation of formula \eqref{eq:primary open numbers} appears in \cite{GKT}.

Regarding higher genus, we conjecture that, for any genus $g\ge 1$, there is a geometric construction of open $r$-spin intersection numbers $\<\tau^{\alpha_1}_{d_1}\cdots\tau^{\alpha_l}_{d_l}\sigma^k\>^{\frac{1}{r},o}_g$ generalizing our construction in genus zero.  Given such intersection numbers, define a generating series~$F_g^{\frac{1}{r},o}(t^*_*,s)$ by
$$
F_g^{\frac{1}{r},o}(t^*_*,s):=\sum_{l,k\ge 0}\frac{1}{l!k!}\sum_{\substack{0\le\alpha_1,\ldots,\alpha_l\le r-1\\d_1,\ldots,d_l\ge 0}}\<\tau^{\alpha_1}_{d_1}\cdots\tau^{\alpha_l}_{d_l}\sigma^k\>^{\frac{1}{r},o}_g t^{\alpha_1}_{d_1}\cdots t^{\alpha_l}_{d_l}s^k.
$$
\begin{conj}\label{main conjecture}
For any $g\ge 1$, we have
$$
F^{\frac{1}{r},o}_g=\left.(-r)^{\frac{g-1}{2}}\phi_g\right|_{t^{r-1}_d\mapsto\frac{1}{\sqrt{-r}}(t^{r-1}_d-\delta_{d,0}rs)}.
$$
\end{conj}

In the sequel to this paper~\cite{BCT3}, we provide geometric and algebraic evidence for the correctness of this conjecture.

\begin{rmk}
In open Gromov--Witten theory, unlike the closed theory, not much is known or even conjectured about higher-genus invariants. The conjecture presented here is one of the few conjectures that describes the full, all-genus open Gromov--Witten theory in particular cases, and also one of the few that relates the potential to an integrable hierarchy (see also \cite{PST14,ST1,BT17,BPTZ}).
\end{rmk}

\subsection{Open-closed correspondence}

An interesting and unexpected consequence of this work is that it reveals a connection between open $r$-spin theory and what we refer to as {\it closed extended $r$-spin theory}.  More specifically, it is possible to enhance the closed theory by allowing exactly one of the twists $a_i$ to be equal to $-1$; in this case, $R^1\pi_*\mathcal{S}$ is still a bundle in genus zero, so the closed theory can be defined exactly as usual.  We define by
\begin{gather*}
F^{\frac{1}{r},\text{ext}}_0(t^*_*):=\sum_{n\ge 2}\frac{1}{n!}\sum_{\substack{0\le a_1,\ldots,a_n\le r-1\\d_1,\ldots,d_n\ge 0}}\<\tau^{-1}_0\tau^{a_1}_{d_1}\cdots\tau^{a_n}_{d_n}\>^{\frac{1}{r},\text{ext}}_0 t^{a_1}_{d_1}\cdots t^{a_n}_{d_n}
\end{gather*}
the generating series for closed extended $r$-spin intersection numbers.

In the companion paper \cite{BCT_Closed_Extended} to this work, we prove topological recursion relations for the closed extended theory, and from these one finds that the generating function of genus-zero intersection numbers in that setting is {\it also} closely related to the wave function of the $r$th Gelfand--Dickey hierarchy. Using this, we deduce the following theorem:

\begin{thm}\label{thm:open-closed}
The generating series $F^{\frac{1}{r},o}_0$ and $F^{\frac{1}{r},\text{ext}}_0$ are related by
\begin{gather*}
F^{\frac{1}{r},o}_0=-\frac{1}{r}\left.F^{\frac{1}{r},\text{ext}}_0\right|_{t^{r-1}_d\mapsto t^{r-1}_d-r\delta_{d,0}s}+\frac{1}{r}F^{\frac{1}{r},\text{ext}}_0.
\end{gather*}
\end{thm}

\subsection{Plan of the paper}

The structure of the paper is as follows.  Section~\ref{sec:mod_and_bundle} reviews from \cite{BCT1} the definition and properties of the moduli space of open graded $r$-spin disks and the Witten bundle~$\mathcal{W}$.  In Section \ref{sec:bc}, we define the canonical boundary conditions and give the definition of the open $r$-spin intersection numbers.  We prove topological recursion relations satisfied by these intersection numbers in Section \ref{sec:recursions}. Using this result, in Section~\ref{section:open numbers and GD hierarchy} we prove Theorems~\ref{thm:main},~\ref{theorem:primary numbers}, and~\ref{thm:open-closed}, and we derive open string and dilaton equations for the open $r$-spin intersection numbers. In the last section, we detail the construction of boundary conditions for the definition of the intersection numbers.

\subsection{Acknowledgements}

The authors thank M.~Gross, J.~Gu\'er\'e, T.~Kelly, R.~Pandharipande, D.~Ross, J.~Solomon, E.~Witten, A.~Netser Zernik and Yizhen Zhao for discussions and suggestions related to this work.  
The authors also thank the anonymous referee for important comments which improved this paper. 
A.~B. was supported by the grant RFBR-20-01-00579. E.~C. was supported by NSF DMS grant 1810969.  R.T. (incumbent of the Lillian and George Lyttle Career Development Chair) was supported by a research grant from the Center for New Scientists of Weizmann Institute, by Dr. Max R\"ossler, the Walter Haefner Foundation, and the ETH Z\"urich Foundation, by the ISF (grant No. 335/19) and partially by ERC-2012-AdG-320368-MCSK.

\section{Review of graded $r$-spin disks}
\label{sec:mod_and_bundle}

We begin by reviewing the definition of graded $r$-spin disks, their moduli space, and the relevant bundles.  For details, we direct the reader to \cite{BCT1}.

\subsection{The moduli space of graded $r$-spin disks}

The underlying objects parameterized by the moduli space are genus-zero marked Riemann surfaces with boundary, which we view as arising from closed genus-zero curves with an involution.  Specifically, a {\it nodal marked disk} is defined as a tuple
\[(C, \phi, \Sigma, \{z_i\}_{i \in I}, \{x_j\}_{j \in B}, m^I, m^B),\]
in which
\begin{itemize}
\item $C$ is a nodal, possibly disconnected, orbifold Riemann surface with isotropy only at special points and with each component having genus zero;
\item $\phi: C \rightarrow C$ is an anti-holomorphic involution that realizes the coarse underlying Riemann surface $|C|$ topologically as the union of two Riemann surfaces, $\Sigma$ and $\overline{\Sigma}=\phi(\Sigma),$ glued along the common subset $\text{Fix}(|\phi|)$;
\item $z_i \in C$ are a collection of distinct {\it internal marked points} whose images in $|C|$ lie in $\Sigma\setminus\text{Fix}(|\phi|)$, with {\it conjugate marked points} $\overline{z_i}:= \phi(z_i)$;
\item $x_j \in \text{Fix}(\phi)$ are a collection of distinct {\it boundary marked points} whose images in $|C|$ lie in $\d \Sigma$;
\item $m^I$ and $m^B$ are markings of $I$ and $B$, respectively (see the definition below), for which no internal marked point on a connected component $C' \subseteq C$ with $C' \cap \phi(C') \neq \emptyset$ is marked $0$.
\end{itemize}
We say that a nodal marked disk is {\it stable} if each irreducible component has at least three special points.  In the above, a {\it marking} of a set $A$ is a function
\[m: A \rightarrow \{0\} \cup 2^{\mathbb{N}}_{*}\]
such that, for all $a \neq a' \in m^{-1}(0)$, the intersection $m(a) \cap m(a')$ is empty, where $2^{\mathbb{N}}_*$ denotes the set of all nonempty subsets of $\mathbb{N}$.  
Such functions are used in what
follows to label the marked points on a curve; the possibility of marking some points by 0 or with a set is desired to handle marked points that arise via normalization of a nodal curve, which do not carry a natural label but can be canonically labeled by this more general type of marking.

Nodal marked disks can have three types of nodes, illustrated in Figure~\ref{fig:nodes} by shading $\Sigma \subseteq |C|$ in each case.  Note that $\d\Sigma\subset\text{Fix}(|\phi|)$, and $\text{Fix}(|\phi|)\setminus\d\Sigma$ is exactly the union of the contracted boundaries.

\begin{figure}[h]
\centering
\begin{subfigure}{.3\textwidth}
  \centering

\begin{tikzpicture}[scale=0.3]
  \shade[ball color = gray, opacity = 0.5] (0,0) circle (2cm);
  \draw (0,0) circle (2cm);
  \draw (-2,0) arc (180:360:2 and 0.6);
  \draw[dashed] (2,0) arc (0:180:2 and 0.6);

  \shade[ball color = gray, opacity = 0.1] (0,-4) circle (2cm);
  \draw (0,-4) circle (2cm);
  \draw (-2,-4) arc (180:360:2 and 0.6);
  \draw[dashed] (2,-4) arc (0:180:2 and 0.6);
   \shade[ball color = gray, opacity = 0.6] (-2,-4) arc (180:360:2 and 0.6) arc (0:180:2);

    \shade[ball color = gray, opacity = 0.1] (0,-8) circle (2cm);
  \draw (0,-8) circle (2cm);
  \draw (-2,-8) arc (180:360:2 and 0.6);
  \draw[dashed] (2,-8) arc (0:180:2 and 0.6);
\end{tikzpicture}

  \caption{Internal node}
\end{subfigure}
\begin{subfigure}{.3\textwidth}
  \centering
\vspace{0.9cm}
\begin{tikzpicture}[scale=0.4]
  \shade[ball color = gray, opacity = 0.1] (0,0) circle (2cm);
  \draw (0,0) circle (2cm);
  \draw (-2,0) arc (180:360:2 and 0.6);
  \draw[dashed] (2,0) arc (0:180:2 and 0.6);
  \shade[ball color = gray, opacity = 0.6] (-2,0) arc (180:360:2 and 0.6) arc (0:180:2);

  \shade[ball color = gray, opacity = 0.1] (4,0) circle (2cm);
  \draw (4,0) circle (2cm);
  \draw (2,0) arc (180:360:2 and 0.6);
  \draw[dashed] (6,0) arc (0:180:2 and 0.6);
  \shade[ball color = gray, opacity = 0.6] (2,0) arc (180:360:2 and 0.6) arc (0:180:2);
\end{tikzpicture}
\vspace{0.9cm}

  \caption{Boundary node}
\end{subfigure}
\begin{subfigure}{.3\textwidth}
  \centering

\begin{tikzpicture}[scale=0.4]
\vspace{0.15cm}
  \shade[ball color = gray, opacity = 0.6] (0,0) circle (2cm);
  \draw (0,0) circle (2cm);
  \draw (-2,0) arc (180:360:2 and 0.6);
  \draw[dashed] (2,0) arc (0:180:2 and 0.6);

  \shade[ball color = gray, opacity = 0.1] (0,-4) circle (2cm);
  \draw (0,-4) circle (2cm);
  \draw (-2,-4) arc (180:360:2 and 0.6);
  \draw[dashed] (2,-4) arc (0:180:2 and 0.6);
\end{tikzpicture}
\vspace{0.15cm}

  \caption{Contracted boundary}
\end{subfigure}
\caption{The three types of nodes in a nodal marked disk.}
\label{fig:nodes}
\end{figure}

The boundary $\d\Sigma$ of $\Sigma$ is equipped with a natural orientation, dictated by the choice of the preferred half $\Sigma \subseteq |C|$.  This, in turn, induces a notion of positivity for $\phi$-invariant sections of $\omega_{|C|}$ over $\d \Sigma$: a section $s$ is said to be {\it positive} if, for any point $p$ and any tangent vector $v \in T_p(\d \Sigma)$ in the direction of orientation, we have $\langle s(p), v \rangle > 0$, where $\langle -, - \rangle$ denotes the natural pairing between cotangent and tangent vectors.

A {\it twisted $r$-spin structure} on a nodal marked disk is a complex line bundle~$S$ on~$C$ whose coarse underlying bundle $|S|$ satisfies
\[|S|^{\otimes r} \cong \omega_{|C|} \otimes \O\left(-\sum_{i \in I} a_i [z_i] - \sum_{i \in I} a_i [\overline{z_i}] - \sum_{j \in B} b_j[x_j]\right),\]
where $a_i\in \{-1,0,1,\ldots, r-1\}$ and $b_j \in \{0,1,\ldots, r-1\}$, together with an involution $\widetilde{\phi}: S \rightarrow S$ lifting $\phi$.  We insist that connected components of $C$ that meet $\text{Fix}(\phi)$ contain no $z_i$ with $a_i = -1$ and that connected components of $C$ that do not meet $\text{Fix}(\phi)$ contain at most one such $z_i$.  The numbers $a_i$ and $b_j$ are referred to as the internal and boundary {\it twists}, respectively.
\begin{obs}
\label{obs:open_rank1}
There exists a twisted $r$-spin structure with twists $a_i, b_j$ on a connected nodal marked disk $C$ if and only if
\begin{equation}\label{eq:open_rank1}
\frac{2\sum a_i + \sum b_j-(r-2)}{r}\in \Z.
\end{equation}
The analogue in the closed case is more well-known \cite{Witten93}: a twisted $r$-spin structure with twists $b_j$ exists on a (closed) connected nodal marked genus-zero curve if and only if 
\begin{equation}\label{eq:close_rank1}
\frac{\sum b_j -(r-2)}{r}\in \Z.
\end{equation}
\end{obs}

To extend the definition of twists to nodes, let $n: \widehat{C} \rightarrow C$ be the normalization morphism.  Then $n^*S$ might not be a twisted $r$-spin structure, because its connected components might contain too many marked points of twist $-1$.  However, there is a canonical way to choose a minimal subset $\mathcal{R}$ of the half-nodes of $C$ making
\begin{equation}
\label{eq:hatS}
\widehat{S}:= n^* S \otimes \O\left(-\sum_{q \in \mathcal{R}} r [q]\right)
\end{equation}
a twisted $r$-spin structure; see \cite[Section 2.3]{BCT1} for the details.  In particular, then, for each irreducible component $C_l$ of $\widehat{C}$ with internal marked points $\{z_i\}_{i \in I_l}$, boundary marked points $\{x_j\}_{j \in B_l}$ and half-nodes $\{p_h\}_{h \in N_l}$, we have
\[\bigg(|\widehat{S}|\big|_{|C_l|}\bigg)^{\otimes r} \cong \omega_{|C_l|} \otimes \O\left(-\sum_{i \in I_l} a_i[z_i] - \sum_{i \in I_l} a_i [\overline{z_i}] - \sum_{j \in B_l} b_j[x_j] - \sum_{h \in N_l} c_j [p_h]\right)\]
for numbers $a_i, c_h \in \{-1,0,1, \ldots, r-1\}$, $b_j \in \{0,1,\ldots, r-1\}$ satisfying the same conditions as above.  The numbers $c_h$ are called the {\it twists} of $S$ at its half-nodes.  If $p$ and $p'$ are the two branches of a node in $C$, then
\begin{equation}\label{eq:twist_sum_at_nodes}c_p + c_{p'} \equiv r-2 \mod r.\end{equation}
The node is said to be {\it Ramond} if one (hence both) of its branches satisfy $c_p \equiv -1 \mod r$, and it is said to be {\it Neveu--Schwarz} otherwise. The set $\mathcal{R}$ in equation~\eqref{eq:hatS} is chosen in a way that guarantees that each internal Ramond node has precisely one half-edge in $\mathcal{R}.$

\begin{rmk}
In orbifold language, Neveu--Schwarz nodes are nodes at which the isotropy group of $C$ acts nontrivially on the fiber of $S$.  Because a section of an orbifold line bundle is necessarily invariant under the action of the isotropy group, nontriviality of the action forces sections of $S$ to vanish at such nodes. This causes a splitting in the normalization exact sequence associated to $S$, which is the key reason why (as discussed in Proposition \ref{pr:decomposition} below) the Witten bundle decomposes in a far more straightforward way along Neveu--Schwarz than along Ramond nodes.
\end{rmk}

Associated to each twisted $r$-spin structure $S$, we define $J:=S^{\vee} \otimes \omega_{C}$. This bundle admits an involution, as well, induced by the involutions on $C$ and $S$; abusing notation slightly, we denote the involution on $J$ also by $\widetilde{\phi}$.

Assume that $C$ has no contracted boundary node, and let $A$ be the complement of the special points in $\d \Sigma$.  We say that a twisted $r$-spin structure on such $C$ is {\it compatible} if there exists a $\widetilde{\phi}$-invariant section
$v \in \Gamma\left(A, |S|^{\widetilde{\phi}}\right)$ (called a {\it lifting} of $S$) such that the image of $v^{\otimes r}$ under the map on sections induced by the inclusion $|S|^{\otimes r} \rightarrow \omega_{|C|}$ is positive.  Such a $v$ always admits a companion lifting
$w \in \Gamma\left(A, |J|^{\widetilde{\phi}}\right)$ of $J$ for which $\langle w, v \rangle \in \Gamma(A, \omega_{|C|})$ is positive, where $\langle -, - \rangle$ denotes the natural pairing between $|S|^{\vee}$ and $|S|$.  This $w$ is uniquely determined by $v$ up to multiplication by a continuous $\mathbb{R}^+$-valued function.

Given a lifting of a twisted $r$-spin structure, we say that a boundary marked point or boundary half-node $x_j$ is {\it legal}, or that the lifting {\it alternates} at $x_j$, if $w$ cannot be continuously extended to $x_j$ without vanishing; note that this definition depends only on $v$, not on the specific companion lifting $w$ chosen.  We say that $w$ is a {\it grading} if every boundary marked point is legal and furthermore, for every Neveu--Schwarz boundary node, one of the two half-nodes is legal and the other is illegal.

If $C$ has a contracted boundary node, then the definition of compatibility and grading must be adapted.  In this case, we say that a twisted $r$-spin structure on $C$ is {\it compatible} if the contracted boundary node $q$ is Ramond and there exists a $\widetilde{\phi}$-invariant element $v \in |S|\big|_q$ such that the image of $v^{\otimes r}$ under the map $|S|^{\otimes r}\big|_q \rightarrow \omega_{|C|}\big|_q$ is positive imaginary under the canonical identification of $\omega_{|C|}\big|_q$ with $\C$ given by the residue. See \cite[Definition 2.8]{BCT1} for more details.  As in the case without a contracted boundary node, such a $v$ admits a $\widetilde{\phi}$-invariant $w \in |J|\big|_q$ such that $\langle v, w \rangle$ is positive imaginary.  We refer to this $w$ as a {\it grading}.

We say that two gradings are {\it equivalent} if they differ by multiplication by a continuous positive function $A \rightarrow \mathbb{R}^+$ (in the case without a contracted boundary node) or multiplication by a positive real number (in the case with a contracted boundary node).  We have shown in \cite[Theorem 5.2]{BCT1} that the choice of grading determines a canonical relative orientation for the Witten bundle, which is one key ingredient in defining open $r$-spin intersection numbers; furthermore, in what follows, we will see that the grading is also crucial in defining canonical boundary conditions.

% \begin{rmk}
% The grading can be thought of as an orientation of the real spin line bundle $|J|^{\widetilde{\phi}}$ on $A$, and the legal points are those to which the orientation cannot be extended continuously.  This notion was first discovered in the spinless case (or, more precisely the $r=2$ case in the absence of Ramond markings) in \cite{ST0,ST1}. It was discovered to be precisely the geometric structure needed in order to solve two problems: orienting the moduli space in all genus, and defining canonical boundary conditions for $\CL_i$ in all genus.  See \cite{Tes15} for a summary and proofs of these facts.  In \cite{BCT1}, we have extended these definitions to the $r$-spin case, and shown that the grading determines a canonical relative orientation for the Witten bundle with desired properties (\cite[Theorem 5.2]{BCT1}). In what follows, we show that the grading is crucial also for defining the canonical boundary conditions in the $r$-spin case.  Roughly speaking, along boundary strata that parameterize surfaces with nodes whose illegal side has twist zero, we will be able to apply a forgetful morphism, whereas along other strata we will have a positivity phenomenon. Both the choice of the legal half of a node, and the notion of positivity, will be outcomes of the grading. 
% \end{rmk}

The relation between the twists and legality, and the obstructions to having a grading, are summarized in the following proposition (which summarizes \cite[Proposition 2.5 and Observation 2.13]{BCT1} and the behavior of a Neveu-Schwarz node in a graded structure):
\begin{prop}
\label{prop:compatibility_lifting_parity}
\begin{enumerate}
\item\label{it:compatibility_odd} When $r$ is odd, any twisted $r$-spin structure is compatible, and there is a unique equivalence class of liftings.
\item\label{it:lifting and parity_odd} Suppose $r$ is odd and $v$ is a lifting over a punctured neighborhood of a boundary marked point $x_j$.  Then $x_j$ is legal if and only if its twist is odd.
\item\label{it:compatibility_even} When $r$ is even, the boundary twists $b_j$ in a compatible twisted $r$-spin structure must be even.  In this case, there is a unique equivalence class of compatible structures.
\item\label{it:lifting and parity_even}
Suppose $r$ is even.  There exists a lifting over $\d \Sigma \setminus \{x_j\}_{j \in B}$ that alternates precisely at a subset $D \subset \{x_j\}_{j \in B}$ if and only if
\begin{equation}
\label{parity}
\frac{2\sum a_i + \sum b_j+2}{r} \equiv |D| \mod 2.
\end{equation}
\item\label{it:Ramond_nodes}
Ramond boundary nodes can appear in a graded structure only when $r$ is odd, and in this case, their half-nodes are illegal with twists $r-1.$
\item\label{it:NS nodes}
In a graded $r$-spin structure, any Neveu-Schwarz boundary node has one legal half-node and one illegal half-node.
\end{enumerate}
\end{prop}

In the case $\text{Fix}(\phi)=\emptyset,$ our notion of grading is vacuous, but we replace it with an additional datum (see \cite[Definition 2.8]{BCT1} for an equivalent definition):

\begin{definition}\label{def:graded_sphere}
A {\it smooth graded $r$-spin sphere} is a nodal marked disk for which $\text{Fix}(\phi) = \emptyset$ and $\Sigma$ is a smooth sphere, together with
\begin{enumerate}
\item a twisted $r$-spin structure $S$,
\item a choice of one distinguished internal marked point $z_i$, called the {\it anchor} and satisfying $m^I(z_i) = 0$, which is the marked point of twist $-1$ if one exists.
\end{enumerate}
If the twist $a_i$ of the anchor is $r-1$, there is a map
\[\tau': \left(|S|\otimes \O\left([z_i]\right)\right)^{\otimes r}\big|_{z_i}\rightarrow \omega_{|C|}([z_i])\big|_{z_i}\cong\C,\]
where the second identification is the residue map.  In this case, we also fix
\begin{enumerate}
\setcounter{enumi}{2}
\item an involution $\widetilde{\phi}$ on the fiber $\left(|S|\otimes \O\left([z_i]\right)\right)_{z_i}$ such that
\[
\tau'(\widetilde{\phi}(v)^{\otimes r})=-\overline{\tau'(v^{\otimes r})}\quad \text{for all $v\in\left(|S|\otimes \O\left([z_i]\right)\right)_{z_i}$},
\]
where $w\mapsto\overline{w}$ is the standard conjugation, and
\[\left\{\tau'(v^{\otimes r})\; | \; v\in \left(|S|\otimes \O\left([z_i]\right)\right)^{\widetilde\phi}_{z_i}\right\}\supseteq i\R_+,\]where $i$ is the root of $-1$ in the upper half-plane;
\item a connected component $V$ of $\left(|S|\otimes \O\left([z_i]\right)\right)^{\widetilde\phi}_{z_i}\setminus\{0\}$, called the {\it positive direction}, such that $\tau'(v^{\otimes r})\in i\R_+$ for any $v \in V$.
\end{enumerate}
\end{definition}

We can now define the primary objects of interest in this paper:

\begin{definition}
\label{def:stable_graded_rspin_disk}
A {\it stable graded $r$-spin disk} is a stable nodal marked disk, together with
\begin{enumerate}
\item a compatible twisted $r$-spin structure $S$ in which all boundary marked points have twist $r-2$ and all contracted boundary nodes are Ramond;
\item an equivalence class of gradings;
\item a choice of one distinguished special point (called the {\it anchor} and marked zero) in each connected component $C'$ of $C$ that is either disjoint from the set $\text{Fix}(\phi)$ or meets the set $\text{Fix}(\phi)$ in a single contracted boundary node.  We require that the anchor is the contracted boundary node, if one exists, or the unique marked point in $C'$ of twist $-1$, if one exists; otherwise, the anchor can be any marked point.  Finally, we require that the collection of anchors is $\phi$-invariant;
\item an involution $\widetilde{\phi}$ on the fiber $(|S| \otimes \O([z_i]))_{z_i}$ over any anchor $z_i$ of twist $r-1$, and an orientation of the $\widetilde{\phi}$-fixed subspace of $(|S| \otimes \O([z_i]))_{z_i}$ for any such $z_i$, as in Definition~\ref{def:graded_sphere}.
\end{enumerate}
\end{definition}

In \cite[Theorem 3.4]{BCT1}, we prove the existence of a moduli space $\M_{0,k,l}^{1/r}$ of such objects (with $k$ boundary and $l$ internal marked points), and we show that it is a compact smooth orientable orbifold with corners of real dimension $k+2l-3$.  More generally, $\M_{0,B,I}^{1/r}$ denotes the moduli space of stable graded $r$-spin disks with boundary points marked by $B$ and internal points marked by $I$, and for any vector $\vec{a} = \{a_i\}_{i \in I}$, we denote by $\M_{0,k,\vec{a}}^{1/r}$ the suborbifold consisting of those for which the $i$th internal marked point has twist $a_i$.

% This suborbifold is nonempty if and only if the twists satisfy
% $$
% \frac{2\sum a_i + (k-1)(r-2)}{r}\in\Z~~~\text{ and }~~~\frac{2\sum a_i + (k-1)(r-2)}{r} \equiv k+1 \mod 2.
% $$

%\begin{thm}{\cite[Theorem 3.4]{BCT1}}  There is a moduli space $\M_{0,B,I}^{1/r}$ of connected stable graded $r$-spin disks with boundary marked points marked by $B$ and internal marked points marked by $I$.  It is a compact smooth orientable orbifold with corners of real dimension $k+2l-3$, equipped with a universal bundle and a universal grading.
%\end{thm}

\subsection{Stable graded $r$-spin graphs}

Analogously to the more familiar setting of the moduli space of curves, $\M_{0,B,I}^{1/r}$ can be stratified according to decorated dual graphs.  The relevant dual graphs $\Gamma$ consist of
\begin{enumerate}[(i)]
\item a vertex set $V$, decomposed into {\it open} and {\it closed} vertices $V = V^O \sqcup V^C$;
\item a half-edge set $H$, decomposed into {\it boundary} and {\it internal} half-edges $H = H^B \sqcup H^I$.
\end{enumerate}
The half-edge set is equipped with an involution $\sigma_1$, which we view as reversing the two halves of an edge.  The fixed points of $\sigma_1$ are the {\it tails} $T$ of $\Gamma$, and we decompose $T$ into $T^B := H^B \cap T$ and $T^I := H^I \cap T$.  The two-element orbits of $\sigma_1$ are the {\it edges} $E$ of $\Gamma$, which can also be decomposed as $E = E^B \sqcup E^I$.  Furthermore, the data of $\Gamma$ includes a subset $H^{CB} \subseteq T^I$ of {\it contracted boundary tails}, as well as a marking $m^B$ on $T^B$ and a marking $m^I$ on $T^I \setminus H^{\text{CB}}$.  We say~$\Gamma$ is {\it closed} if $V^O = \emptyset$, and we say $\Gamma$ is {\it smooth} if $E = H^{CB} = \emptyset$.

In particular, each element of $\M_{0,k,l}^{1/r}$ has a corresponding dual graph $\Gamma$, and this $\Gamma$ is equipped with three additional decorations: a map
\[\text{tw}: H \rightarrow \{-1,0,1,\ldots, r-1\}\]
encoding the twist of $S$ at each marked point and half-node, a map
\[\text{alt}: H^B \rightarrow \mathbb{Z}/2\mathbb{Z}\]
given by $\text{alt}(h) = 0$ if the half-node given by $h$ is illegal and $\text{alt}(h) = 1$ otherwise, and a subset
$T^* \subseteq T^I$
given by the anchors.

% These additional decorations satisfy some relations that implement in the graph level the requirement that boundary markings have twist $r-2,$ together with properties of anchors and the constraints of \eqref{eq:open_rank1},~\eqref{eq:close_rank1},~\eqref{eq:twist_sum_at_nodes} and Proposition \ref{prop:compatibility_lifting_parity}.

A {\it genus-zero graded $r$-spin dual graph} is a graph for which each connected component is the dual graph of an element of some moduli space $\M_{0,k,l}^{1/r}$.  In particular, given a connected graph $\Gamma$ as above, there is a closed suborbifold with corners $\M_{\Gamma}^{1/r} \subseteq \M_{0,B,I}^{1/r}$ whose general point is a graded $r$-spin disk with dual graph $\Gamma$.  We write $\CM_{\Gamma}^{1/r}$ for the open substack of $\M_{\Gamma}^{1/r}$ consisting of curves whose dual graph is exactly $\Gamma$.  If $\Gamma$ is disconnected, we define $\M_{\Gamma}^{1/r}$ as the product of the moduli spaces associated to its connected components.

There are various forgetful maps between these moduli spaces: we denote~by
\[\text{For}_{\text{spin}}: \M^{1/r}_{0,k,l} \rightarrow \M_{0,k,l}\]
the map that forgets the spin structure, and for $B', I' \subseteq \mathbb{Z}$, we denote by
\[\text{For}_{B',I'}: \M_{\Gamma}^{1/r} \rightarrow \M_{\Gamma'}^{1/r}\]
the map that forgets all twist-zero internal marked points marked by $I'$ and all twist-zero illegal boundary marked points marked by $B'$.  Note that the latter procedure may create unstable components, which must be contracted, so we denote by $\Gamma' = \text{for}_{B',I'}(\Gamma)$ the new graph that results.

\subsection{The Witten bundle and the cotangent line bundles}
\label{subsec:Wittenbundle}

The crucial bundle for the definition of $r$-spin theory is the {\it Witten bundle} on the moduli space.  Roughly speaking, if $\pi: \mathcal{C} \rightarrow \M_{0,k,l}^{1/r}$ is the universal curve and $\mathcal{S} \rightarrow \mathcal{C}$ is the universal spin bundle with companion bundle $\mathcal{J}:= \mathcal{S}^{\vee} \otimes \omega_{\pi}$, then we~define
\begin{equation}
\label{eq:Wittenbundledef}
\mathcal{W}:= (R^0\pi_*\mathcal{J})_+ = (R^1\pi_*\mathcal{S})^{\vee}_-,
\end{equation}
where the subscripts $+$ and $-$ denote invariant or anti-invariant sections under the universal involution $\widetilde{\phi}: \mathcal{J} \rightarrow \mathcal{J}$ or $\widetilde{\phi}: \mathcal{S} \rightarrow \mathcal{S}$.  To be more precise, defining~$\mathcal{W}$ by \eqref{eq:Wittenbundledef} would require one to deal with derived pushforward in the orbifold-with-corners context, so to avoid this technicality, we define $\mathcal{W}$ by pullback of the analogous bundle from a subset of the closed moduli space $\M_{0,k+2l}^{1/r}$; see \cite[Section 4.1]{BCT1}.

% \begin{rmk}\label{def:scale_invariant}
% The Witten bundles are orbifold vector bundles. The action of the isotropy groups of the moduli space is the one induced from lifting the fiberwise action on $J$ (which, in turn, is induced from the action on the spin bundle) to the space of global sections of $J$.  In particular, the group acting on a generic fiber in the open case is trivial, while in the closed connected case it is the group of $r$th roots of unity that acts by scaling.
% \end{rmk}

The real rank of the Witten bundle, on a component of the moduli space with internal twists $\{a_i\}$ and boundary twists $\{b_j\}$, is
\begin{equation}\label{eq:bundle_rk}\frac{2 \sum a_i + \sum b_j - (r-2)}{r}.\end{equation}
Furthermore, in \cite[Theorem 5.2]{BCT1}, we prove that in the graded case, where all $b_j=r-2,$ the Witten bundle is canonically relatively oriented relative to the moduli space.  More precisely, for any set $\vec{a} = \{a_i\}_{i \in I}$ of internal twists, we construct an orientation of $\M_{0,B,\vec{a}}^{1/r}$ satisfying certain properties \cite[Proposition 3.12]{BCT1} and a compatible orientation of $T\M^{1/r}_{0,B,\vec{a}} \oplus \mathcal{W}.$ Although both orientations depend on choices, the orientation that they induce on the total space is canonical and independent of choices.

In \cite{BCT1} we prove that the Witten bundle satisfies certain decomposition properties along nodes.  We reiterate these properties here as they are needed in what follows, but we direct the reader to \cite[Section 4.2]{BCT1} for further details.

To set the stage, let $\Gamma$ be a genus-zero graded $r$-spin dual graph, and let $e\in E(\Gamma)$.  Define
$\text{detach}_e(\Gamma)$ to be the ``detaching" of $\Gamma$ at $e$, which is the disconnected graph obtained by cutting the edge $e$.  This creates new tails, so we must extend the marking and (possibly) the anchors: if $e$ is a boundary edge, we mark both newly-created tails $0$ and add no new anchor.  If $e$ is an internal edge, then exactly one of the connected components of $\text{detach}_e(\Gamma)$ is closed (that is, has no open vertices) and unanchored; let $h$ denote the half-edge of $e$ in this component, and let $h'$ denote the other half-edge of $e$.  In this case, we mark $h$ by $0$ and declare it to be an anchor, and we mark $h'$ by the union of the markings of the internal tails $h'' \neq h$ in the same component as $h$. Define $\text{detach}_t(\Gamma),$ for $t \in H^{\text{CB}},$ to be the graph that agrees with $\Gamma$ except that $t$ is no longer considered an element of $H^{\text{CB}}$; all other decorations (including the fact that $t$ is the anchor) remain the same. For $N\subseteq E(\Gamma)\cup H^{\text{CB}}(\Gamma)$, ~$\text{detach}_N(\Gamma)$ denotes the graph obtained from $\Gamma$ by detaching at all elements of $N.$

Given a genus-zero graded $r$-spin dual graph $\Gamma$, let $\widehat{\Gamma}$ be obtained by detaching either an edge or a contracted boundary tail.  Then there are morphisms
\begin{equation}
\label{eq:Wittendecompsequence}
\M_{\widehat{\Gamma}}^{1/r} \xleftarrow{q} \M_{\widehat{\Gamma}} \times_{\M_{\Gamma}} \M_{\Gamma}^{1/r} \xrightarrow{\mu} \M_{\Gamma}^{1/r} \xrightarrow{i_{\Gamma}} \M_{0,k,l}^{1/r},
\end{equation}
where $\M_{\Gamma} \subseteq \M_{0,k,l}$ is the moduli space of marked disks (without $r$-spin structure) specified by the dual graph $\Gamma$.  The morphism $q$ is defined by sending the $r$-spin structure $S$ to the $r$-spin structure $\widehat{S}$ defined by \eqref{eq:hatS}; it has degree one but is not an isomorphism because it does not induce an isomorphism on isotropy groups.  The morphism $\mu$ is the projection from the fiber product; it is an isomorphism, but we distinguish between its domain and codomain because they have different universal objects.  Finally, the morphism $i_{\Gamma}$ is the inclusion.

There are Witten bundles $\mathcal{W}$ on $\M_{0,k,l}^{1/r}$ and $\widehat{\mathcal{W}}$ on $\M_{\widehat{\Gamma}}^{1/r}$, and the decomposition properties are stated in terms of how these bundles are related under pullback via the morphisms \eqref{eq:Wittendecompsequence}.  Specifically, we have the following proposition, in which the symbol $\boxplus$ denotes a vector bundle on a fiber product obtained by taking the direct sum of pullbacks of bundles on the two factors.

\begin{pr}{\cite[Proposition 4.7]{BCT1}}
\label{pr:decomposition}
Let $\Gamma$ be a stable genus-zero graded $r$-spin dual graph with a single edge $e$, and let $\widehat{\Gamma}$ be the detaching of $\Gamma$ along $e$.  Then the Witten bundle decomposes as follows:
\begin{enumerate}[(i)]
\item\label{it:NS} If $e$ is a Neveu--Schwarz edge, then $\mu^*i_{\Gamma}^*\cW = q^*\widehat\cW$.

\item\label{it:Ramondbdryedge} If $e$ is a Ramond boundary edge, then there is an exact sequence
\begin{equation}
\label{eq:decompses}0 \rightarrow \mu^*i_{\Gamma}^*\mathcal{W} \rightarrow q^*\widehat{\mathcal{W}} \rightarrow \TTT_+ \rightarrow 0,
\end{equation}
where $\TTT_+$ is a trivial real line bundle.

\item If $e$ is a Ramond internal edge connecting two closed vertices, write $q^*\widehat\cW = \widehat\cW_1 \boxplus \widehat\cW_2$, in which $\widehat\cW_1$ is the Witten bundle on the component containing the anchor of $\Gamma$ and $\widehat\cW_2$ is the Witten bundle on the other component.  Then there is an exact sequence
\begin{equation}
\label{eq:decompses2}
0 \rightarrow \widehat\cW_2 \rightarrow \mu^*i_{\Gamma}^*\cW \rightarrow \widehat\cW_1 \rightarrow 0.
\end{equation}
Furthermore, if $\widehat\Gamma'$ is defined to agree with $\widehat\Gamma$ except that the twist at each Ramond tail is $r-1$, and $q': \M_{\widehat{\Gamma}} \times_{\M_{\Gamma}} \M_{\Gamma}^{1/r} \rightarrow \M_{\widehat\Gamma'}^{1/r}$ is defined analogously to $q$, then there is an exact sequence
\begin{equation}
\label{eq:decompses3}
0 \rightarrow \mu^*i_{\Gamma}^*\mathcal{W} \rightarrow (q')^*\widehat{\mathcal{W}}' \rightarrow \TTT \rightarrow 0,
\end{equation}
where $\widehat{\mathcal{W}}'$ is the Witten bundle on $\M_{\widehat\Gamma'}^{1/r}$ and $\TTT$ is a line bundle whose $r$th power is trivial.

\item If $e$ is a Ramond internal edge connecting an open vertex to a closed vertex, write $q^*\widehat\cW = \widehat\cW_1 \boxplus \widehat\cW_2$, in which $\widehat\cW_1$ is the Witten bundle on the open component (defined via $\widehat\cS|_{\mathcal{C}_1}$) and $\widehat\cW_2$ is the Witten bundle on the closed component.  Then the exact sequences \eqref{eq:decompses2} and \eqref{eq:decompses3} both hold.
\end{enumerate}
Analogously, if $\Gamma$ has a single vertex, no edges, and a contracted boundary tail~$t$, and $\widehat{\Gamma}$ is the detaching of $\Gamma$ along~$t$, then there is a decomposition property:
\begin{enumerate}[(i)]
\setcounter{enumi}{4}
\item\label{it:cont_bdry_tail} If $\cW$ and $\widehat\cW$ denote the Witten bundles on $\M_{0,k,l}^{1/r}$ and $\M_{\widehat\Gamma}^{1/r}$, respectively, then the sequence \eqref{eq:decompses} holds.
\end{enumerate}
\end{pr}

\begin{rmk}\label{rmk:dec_stronger}
If the edge $e$ in Proposition~\ref{pr:decomposition} is a boundary edge, then the map $q$ is an isomorphism, and in this case, the proposition says that the Witten bundle pulls back under the gluing morphism $\M_{\widehat\Gamma}^{1/r} \rightarrow \M_{0,k,l}^{1/r}$.

In the case of an internal Neveu-Schwarz edge, although $q$ is not an isomorphism, a multisection of $q^*\widehat\cW$ canonically and uniquely induces a multisection of $\mu^*i_{\Gamma}^*\cW$, since the orbits of the action of the automorphisms on the fibers of the latter are contained in the corresponding orbits of the former.  Indeed, on $\widehat\cW$, the group acts by independent scalings of its two summands by roots of unity, while on $\cW$, the action scales the two components by roots of unity that differ by some $\xi^d$, where $\xi$ is an $r$th root of unity, $d$ is a multiple of $\text{gcd}(m,r)$, and $m$ is the multiplicity at one of the half-nodes.
\end{rmk}

The last important line bundles for the calculations that follow are the cotangent line bundles at internal marked points.  These line bundles have already been defined on the moduli space $\M_{0,k,l}$ of stable marked disks (without spin structure), as the line bundles with fiber $T^*_{z_i}\Sigma$, or equivalently, the pullback of the usual cotangent line bundles under the doubling map $\M_{0,k,l} \rightarrow \M_{0,k+2l}$ that sends $\Sigma$ to $C$.  The bundle $\mathbb{L}_i$ on $\M_{0,B,I}^{1/r}$ is the pullback of this cotangent line bundle on $\M_{0,k,l}$ under the morphism that forgets the spin structure.  Because $\mathbb{L}_i$ is a complex line bundle, it carries a canonical orientation.

\section{The definition of open correlators}\label{sec:bc}
 
Our goal is to define intersection numbers by integrating the top Chern class of direct sums of the Witten bundle $\cW$ and the bundles $\mathbb{L}_i$---in other words, by counting the (weighted) number of zeroes of a section of this direct sum.  However, since the moduli space has boundary, such a count is only well-defined after prescribing the boundary behavior of the section.  Our aim, then, is to define canonical boundary conditions for sections of $\cW$ or $\CL_i$ at all codimension-$1$ boundary strata.  The na\"ive guess for such a condition is that the section should decompose into a direct sum according to Proposition~\ref{pr:decomposition}, but this is insufficient to determine intersection numbers uniquely.

To define (roughly) what it means for a smooth section $s$ of $\cW$ or $\CL_i$ to be {\it canonical} at a boundary stratum $\CM_{\Gamma}^{1/r}$, suppose first that $\Gamma$ has a single edge $e$ and the illegal half-edge of $e$ has twist zero.  In this case, the relevant idea has already appeared in \cite{PST14}: we define $s$ to be canonical if it is pulled back along the forgetful map $F_\Gamma:\CM_\Gamma^{1/r}\to\CM_{\CB\Gamma}^{1/r}$, where $\CB\Gamma$ is obtained by detaching $e$ and forgetting the illegal side.  This forgetting procedure is no longer defined, however, if the illegal side of $e$ has positive twist.  Instead, in this case, we define canonicity by the condition that $s$ evaluates ``positively" at the illegal half-node, meaning that it is a nonzero element of the fiber of $J$ in the direction of the grading.  A similar definition applies in the case where $\Gamma$ has a contracted boundary tail.   Since these conditions are defined stratum-wise, it is not at all clear that they can be simultaneously satisfied to yield a global canonical section; we show that this is the case, using deep properties of the positivity phenomenon, in Proposition \ref{prop:pointwise_positivity}.

One particular issue that arises is that there can exist a boundary point $p$ in the moduli space for which one cannot find a section of $\cW$ on a neigborhood $U$ of $p$ that satisfies the positivity constraints at all boundary strata intersecting~$U$.  We solve this problem by slightly modifying the moduli space, replacing it by a subspace $\oPMr \subseteq \oCMr$ obtained by excluding strata~$\CM_\Gamma^{1/r}$ with at least one illegal half-edge of positive twist.  We then define the positivity constraint for a section $s$ in a neighborhood $U$ of such strata by requiring that, for $\Sigma\in U\cap\oPMr$, the evaluation of $s_{\Sigma}$ at certain intervals $I_{h,\Sigma} \subseteq \d\Sigma$ is positive, where, for each $h \in H^B(\Gamma)$, the sets $I_{h,\Sigma}$ converge in the universal curve to a neighborhood of the boundary node corresponding to $h$.  We prove in Proposition \ref{prop:pointwise_positivity}, which is the key geometric idea of this work and the place where the twist $r-2$ at boundary marked points is required, that this requirement is consistent for higher-codimension corners of the moduli space.

Although $\oPMr$ is noncompact, and the boundary conditions depend on choices, the resulting intersection numbers are independent of choices. This is the content of the main theorem of this section, Theorem \ref{thm:int_numbers_well_defined}.

\subsection{Canonical multisections}\label{subsec:canonical_bc}

First, some notation for graph operations.

\begin{definition}\label{def:smoothingraph}\label{rmk:iota}
Let $\Gamma$ be a graded $r$-spin dual graph $\Gamma$ and $e$ an edge connecting vertices $v_1$ and $v_2$.  The {\it smoothing} of $\Gamma$ along $e$ is the graph $d_e\Gamma$ obtained by contracting $e$ and replacing $v_1$ and $v_2$ with a single vertex $v_{12}$ that is declared to be closed if and only if both $v_1$ and $v_2$ are closed. The {\it smoothing} of $\Gamma$ along $h \in H^{CB}$ is the graph $d_h\Gamma$ obtained by erasing $h$ and moving the vertex $v$ from which $h$ emanates from $V^C$ to $V^O$.
%\begin{rmk}\label{rmk:iota}
If $\Lambda$ is a smoothing of $\Gamma$, then each (half-)edge $h$ of $\Lambda$ corresponds to a unique (half-)edge $\iota_{\Lambda,\Gamma}(h)$ of $\Gamma.$
%\end{rmk}
\end{definition}

For a set $S$ of edges and contracted boundary tails, one can perform a sequence of smoothings, and the graph obtained is independent of the order in which those smoothings are performed; denote the result by $d_S\Gamma$.  Let
\begin{align*}
&\d^!\Gamma = \{\Lambda \; | \; \Gamma = d_S\Lambda \text{ for some } S\},\\
&\d \Gamma = \d^!\Gamma \setminus \{\Gamma\},\\
&\d^B\Gamma = \{\Lambda \in \d \Gamma \; | \; E^B(\Lambda) \cup H^{CB}(\Lambda) \neq \emptyset\}.
\end{align*}
We refer to a boundary half-edge $h,$ or the corresponding half-node, as {\it positive} if $\alt(h) = 0$ and $\text{tw}(h) > 0$.  We write $H^+(\Lambda)$ for the set of half-edges $h$ of a graph $\Lambda$ such that either $h$ or $\sigma_1(h)$ is positive.  We denote by
$\partial^{+}\Gamma \subseteq \partial^{!}\Gamma$
the graphs with at least one positive half-edge, and we write
$\partial^0\Gamma=\partial^B\Gamma\setminus\partial^+\Gamma.$

Let $\oPM_\Gamma$ and $\partial^+\oCM_\Gamma^{1/r}$ be the orbifolds with corners defined by
\[\oPM_\Gamma = \oCM_\Gamma^{1/r}\setminus
\left(\coprod_{\Lambda \in \partial^{+}\Gamma}\CM_\Lambda^{1/r}\right),\qquad \partial^+\oCM_\Gamma^{1/r}=\coprod_{\Lambda\in\partial^+_\Gamma}\CM_\Lambda^{1/r}.\]
Note that the boundary of $\oPM_\Gamma$, which we will denote by $\partial^0\oCM_\Gamma$ in keeping with the notation $\d^0\Gamma$ of the previous paragraph, contains only strata corresponding to graphs without positive half-edges.

We now define what it means for vectors in the fiber of the Witten bundle to ``evaluate positively," and to specify where such positive evaluation is required. % First, the definition of evaluation:
\begin{definition}
Let $C$ be a graded $r$-spin disk, $q\in C$ a special point, and $v\in\mathcal{W}_C$. The {\it evaluation} of $v$ at $q$ is $\text{ev}_q(v) := v(q) \in J_q$.
\end{definition}

\begin{definition}
A marked orbifold Riemann surface with boundary is {\it partially stable} if it has exactly one internal marking and no boundary markings.  Given a graded $r$-spin disk, let $n:\widehat{C} \to C$ be the normalization map.  A connected component of $\widehat{C}$ is \emph{\stronglypositive}if it contains either (a) a partially stable component, (b) a contracted boundary node, or (c) a component that becomes partially stable after forgetting all illegal half-nodes of twist zero. We analogously define a \stronglypositive dual graph. %A dual graph that represents a surface with \stronglypositive components is said to be \stronglypositive.

A vector $v \in \cW_C$ \emph{evaluates positively} at a component $D$ of $C$ if $\text{ev}_q(\NNN^*v)$ at each $q\in\partial D$ is positive, with respect to the grading of the boundary or the contracted boundary of $D.$ A vector that evaluates positively at any \stronglypositive component of $C$ is said to be \emph{positive on \stronglypositive components}.
\end{definition}

Positive evaluation will also be required on certain ``intervals" in $\d\Sigma$, which we now define; they should be viewed as a smoothly-varying family of intervals that approximates a family of neighborhoods of the nodes in a nodal surface.

\begin{definition}\label{def:positive_neighborhoods1}
Let $\Gamma$ be a graded graph and let $\Lambda \in \d^+\Gamma$.  A \emph{$\Lambda$-set with respect to $\Gamma$} is an open set $U\subseteq \oCM_\Gamma^{1/r}$ whose closure intersects precisely those~$\CM_\Xi^{1/r}$ with $\Lambda\in \partial^!\Xi$.
A \emph{$\Lambda$-neighborhood with respect to $\Gamma$} of $u\in\oCM_\Gamma^{1/r}$ is a neighborhood $U\subseteq\oCM_\Gamma^{1/r}$ of $u$ that is a $\Lambda$-set with respect to $\Gamma$.  Since there is a unique smooth graph $\Gamma$ with $\Lambda\in\partial^!\Gamma$, we refer to a $\Lambda$-neighborhood with respect to a smooth $\Gamma$ simply as a $\Lambda$-neighborhood.

Let $C$ be a graded marked disk. An \emph{interval} is a connected open set $I\subset \partial\Sigma$ such that, if we write $\d\Sigma = S/\sim$ for a space $S$ homeomorphic to $S^1$, the preimage of $I$ under the quotient map $q: S \rightarrow \d\Sigma$ is the union of an open set with a finite number of isolated points.

Suppose $C\in\CM_\Lambda^{1/r}$, and let $n_h$ be a boundary half-node corresponding to a half-edge $h\in H^B(\Lambda)$.  We say that $n_h$ \emph{belongs} to the interval $I$ if the corresponding node $n_{h/\sigma_1}$ lies in $I$.  In this case, again denoting by $\NNN: \widehat{C} \rightarrow C$ the normalization map and denoting $\widehat{\Sigma} = \NNN^{-1}(\Sigma)$, we write $n^{-1}(I) = I_1 \cup I_2$ in which $I$ is a half-open interval with endpoint $n_h$ and $I_2$ is a half-open interval with endpoint $n_{\sigma_1(h)}$, under the canonical orientation of $\d\widehat{\Sigma}$.

Suppose $U$ is contained in some $\Lambda$-set. Then a \emph{$\Lambda$-family of intervals} for $U$ is a family $\{I_{h}(u)\}_{h\in\Pos(\Lambda),u\in U}$ such that:
\begin{enumerate}
\item
Each $I_{h}(u)$ is an interval in $\partial\Sigma_u,$ where $\Sigma_u \subseteq \pi^{-1}(u)$ is the preferred half in the fiber of the universal curve $\pi:\CC\to U$, and the endpoints of each $I_h(u)$ vary smoothly with respect to the smooth structure of the universal curve restricted to $U$.
\item
For $h \notin\{ h', \sigma_1(h')\}$, we have $I_h(u')\cap I_{h'}(u')=\emptyset$.  If $h=\sigma_1(h')$, we have $I_h(u')\cap I_{h'}(u')\neq\emptyset$ if and only if $u'\in\CM_\Xi^{1/r},$ for some $\Xi$ with $\Lambda\in\partial^!\Xi,$ and $h\in\Pos(\Xi)$ under the locally-defined injection $\iota:H(\Xi)\to H(\Lambda)$.  In this case, $I_h(u')\cap I_{\sigma_1(h)}(u')$ consists exactly of the node $n_{h/\sigma_1}$.
\item There are no markings in $I_h(u')$, and the only half-node that belongs to $I_h(u')$ is $n_{h},$ which is defined precisely when $u'\in\CM_\Xi^{1/r}$ and $h\in\Pos(\Xi).$
\end{enumerate}

If the moduli point $u$ represents the stable graded disk $C$ we write $I_h(C)$ for $I_h(u).$ Figure \ref{fig:intervals} shows the local picture of intervals in smooth and nodal disks.
\end{definition}

\begin{figure}[t]
\centering

\vspace{-1cm}

\begin{tikzpicture}[scale=1]

\tkzDefPoint(1,0){A}\tkzDefPoint(0,1){B}\tkzDefPoint(-1,0){C}
\tkzCircumCenter(A,B,C)\tkzGetPoint{O}
\tkzDrawArc(O,A)(C)
\tkzDefPoint(0,1){D}\tkzDefPoint(-0.707,0.707){E}
\tkzDrawArc[densely dashed,line width = 2pt](O,D)(E)
\tkzDefPoint(0,1){D}\tkzDefPoint(0.707,0.707){E}
\tkzDrawArc[line width = 2pt](O,E)(D)

\tkzDefPoint(-1,2){A'}\tkzDefPoint(0,3){B'}\tkzDefPoint(1,2){C'}
\tkzCircumCenter(A',B',C')\tkzGetPoint{O}
\tkzDrawArc(O,A')(C')
\tkzDefPoint(-0.707,1.293){D'}\tkzDefPoint(0,1){E'}
\tkzDrawArc[densely dashed,line width = 2pt](O,D')(E')
\tkzDefPoint(0.707,1.293){D'}\tkzDefPoint(0,1){E'}
\tkzDrawArc[line width = 2pt](O,E')(D')

\node at (0,0.5) {$n_{h'}$};
\node at (0,1.5) {$n_{h}$};
\node at (1,1) {$I_h$};
\node at (-1,1) {$I_{h'}$};

\tkzDefPoint(1+3,0+1){A}\tkzDefPoint(0+3,1+1){B}\tkzDefPoint(0+3,-1+1){C}
\tkzCircumCenter(A,B,C)\tkzGetPoint{O}
\tkzDrawArc(O,C)(B)
\tkzDefPoint(0.707+3,-0.707+1){D}\tkzDefPoint(0.707+3,0.707+1){E}
\tkzDrawArc[densely dashed, line width = 2pt](O,D)(E)

\tkzDefPoint(3.5+3,0+1){A'}\tkzDefPoint(2.5+3,1+1){B'}\tkzDefPoint(2.5+3,-1+1){C'}
\tkzCircumCenter(A',B',C')\tkzGetPoint{O}
\tkzDrawArc(O,B')(C')
\tkzDefPoint(1.793+3,0.707+1){D'}\tkzDefPoint(1.793+3,-0.707+1){E'}
\tkzDrawArc[line width = 2pt](O,D')(E')

\node at (0.5+3,0+1) {$I_{h'}$};
\node at (2+3,0+1) {$I_h$};

\end{tikzpicture}

\caption{The intervals $I_h$ and $I_{h'}$, corresponding to half-nodes $n_h$ and $n_{h'}$, are drawn as thicker lines over the thinner lines that denote the boundary.  The picture on the right represents a nearby point in the moduli space to the picture on the left, but where the node has been smoothed.}

\label{fig:intervals}
\end{figure}

\begin{definition}\label{def:positive_section}
Let $C$ be a graded $r$-spin disk, and let $A \subseteq \partial\Sigma$ be a subset without legal special points.  Then an element $w\in\cW_\Sigma$ \emph{evaluates positively at~$A$} if $\ev_x(w)$ is positive with respect to the grading for every $x\in A$.

Let $\Gamma$ be a graded graph, $U$ a $\Lambda$-set with respect to $\Gamma$, and $\{I_h\}$ a $\Lambda$-family of intervals for $U$.  Given a multisection $s$ of $\cW$ defined in a subset of $\oPM_\Gamma$ containing $U\cap\oPM_\Gamma$, we say $s$ is \emph{$(U,I)$-positive} (with respect to $\Gamma$) if for any $u'\in U\cap \oPM_\Gamma$, any local branch $s_i(u')$ evaluates positively at each $I_h(u')$.

A multisection $s$ defined in $W\cap\oPM_\Gamma$, where $W$ is a neighborhood of $u\in\oCM_\Lambda^{1/r}$, is \emph{positive near} $u$ (with respect to $\Gamma$) if there exists a $\Lambda$-neighborhood $U\subseteq W$ of $u$ and a $\Lambda$-family of intervals $I_*(-)$ for $U$ such that $s$ is $(U,I)$-positive.

If $W$ is a neighborhood of $\partial^+\oCM_\Gamma^{1/r}$, then a multisection $s$ defined in a set
\begin{equation}\label{eq:U_+}
U_+=\left(W\cap\oPM_\Gamma\right)\cup \bigcup_{\Lambda~\text{\stronglypositive}}{\oPM_\Lambda}
\end{equation}
is \emph{positive} (with respect to $\Gamma$) if it is positive near each point of $\partial^+\oCM_\Gamma^{1/r}$ and at \stronglypositive components.  As above, we omit the phrase ``with respect to $\Gamma$" if $\Gamma$ is smooth.
\end{definition}

\begin{rmk}\label{rmk:different_notions_pos}
Since there are different notions of positivity in this paper, for the benefit of the reader we briefly summarize them and their relationships.
\begin{enumerate}
\item For a non-special boundary point $q$ of a stable $r$-spin disk $C,$ or a contracted boundary component $q$, a vector $s\in J_q,$ is positive if it belongs to the positive ray defined by the grading.
\item For a subset $A$ of the boundary of a stable $r$-spin disk $C$ (which may be a point, a union of intervals, or the whole boundary component), or a contracted boundary $A$, a vector $v \in \cW_C$ is positive at $A$ if, when $v$ is thought of as a global section of $J\to C,$ it evaluates positively at each point of $A,$ in the sense of the first item. 
\item A section of $\cW$ is positive at a subset of the moduli space if it satisfies the positivity constraints of the second item for every moduli point $C$ in this subset, and a properly chosen set $A=A(C)$ (as explained in the above definitions). 
\item A boundary half-edge $h$ and and the corresponding half-node are positive if they are illegal and have a positive twist. We require positivity conditions in the sense of the previous item on boundary strata whose graphs contain these half-edges, where for each $C$ in such a boundary stratum, the positive half-edges will dictate the shape of the set $A$.
\item A strongly positive component in an $r$-spin disk is a component of its normalization that is either partially stable, or becomes partially stable after forgetting all special illegal points of twist $0$, or contains a contracted boundary. On such components the set $A$ from the definition of positivity is the whole boundary or contracted boundary. 
\end{enumerate}
\end{rmk}

The notion of canonicity of a multisection of $\cW$ over (an open subset of) $\oPM_\Gamma$ has two parts: a positivity constraint, in terms of the above definitions, and a requirement that the multisection be pulled back from a smaller graph called the ``base" of $\Gamma$.  To define the base, we first define $\text{for}_{\text{illegal}}(\Gamma)$ to be the graph obtained by forgetting all illegal, twist-zero boundary tails with marking zero, and we define $E^0(\Gamma) \subseteq E^B(\Gamma)$ to be the set of boundary edges with one illegal side of twist zero. % For a set $N$ of edges, we denote by $\text{detach}_N(\Gamma)$ the graph obtained from $\Gamma$ by detaching all edges in $N$ (according to the procedure $\text{detach}_e$ explained in Section~\ref{subsec:Wittenbundle}) above.
\begin{definition}
The {\it base} of a graded graph $\Gamma$ is defined as
\[\CB \Gamma = \text{for}_{\text{illegal}} (\detach_{E^0(\Gamma)}(\Gamma)).\]
\end{definition}

The \emph{base moduli} of $\oCM_{\Gamma}^{1/r}$ is the moduli space $\oCM_{\CB\Gamma}^{1/r}.$ It admits a map
\begin{equation}
\label{eq:FGamma}
F_{\Gamma}:\M_{\Gamma}^{1/r}\to\M_{\CB\Gamma}^{1/r}
\end{equation}
that associates to a moduli point the result of the normalization at all boundary nodes \emph{which correspond to edges of $\Gamma$}, and forgets the resulting illegal boundary points which have twist zero.

\begin{nn}\label{nn:fGamma}
Let $f_{\Gamma}:\partial^!\Gamma\to \d^!\CB\Gamma$ be defined by
$f_{\Gamma}(\Gamma')=\Gamma''$ whenever $F_\Gamma$ maps $\CM_{\Gamma'}^{1/r}$ to $\CM_{\Gamma''}^{1/r}.$
Explicitly,
\[\Gamma''=\text{for}_{\text{illegal}} ( \detach_{\iota_{\Gamma,\Gamma'}(E^0(\Gamma))}(\Gamma')\]
(c.f. Definition \ref{rmk:iota}).  Observe that any $h\in H(\Gamma')$ whose vertex is not removed by the $\text{for}_{\text{illegal}}$ operation (though maybe contracted) naturally corresponds to an edge of $\Gamma''$ which we denote by $f_\Gamma(h)$.  Note that, because of contracted components, there may be $h\neq h'\in H(\Gamma')$ with $f_\Gamma(h)= f_\Gamma(h')$.
\end{nn}

\begin{rmk}\label{rmk:F_in_surface_level}
We denote by $\CB C$ the graded surface that corresponds to the moduli point $F_{\Gamma(C)}(C)$, or by $\CB\Sigma$ its preferred half.  Let $\widehat{C}'$ be the subsurface of the normalization of $C$ obtained by removing components that are neither stable nor partially stable. It admits a natural map
$\phi_\Gamma:\widehat{C}'\to \CB C.$
\end{rmk}
The following observations are straightforward. First, analogously to \cite[Observations 3.14 and 3.28]{PST14}, we have the following compatibility relation:

\begin{obs}\label{obs:for_comp}
Let $\Gamma$ be a graded graph and let $\Gamma'\in\partial\Gamma$.  Then $F_{\Gamma}$ takes $\oCM_{\Gamma'}^{1/r}$ to $\oCM_{f_{\Gamma}(\Gamma')}^{1/r},$ and $\CB f_{\Gamma}(\Gamma')=\CB\Gamma',$ and moreover $F_{f_{\Gamma}(\Gamma')}\circ F_{\Gamma}|_{\oCM_{\Gamma'}^{1/r}}=F_{\Gamma'}.$
\end{obs}

The next observation, analogous to \cite[Observation 3.15]{PST14} is central for proving the independence of choices for the intersection numbers.
\begin{obs}\label{obs:dim_of_base}
Let $\Gamma$ be a smooth graded graph and $\Lambda\in\d^0\Gamma$ be such that $\CB\Lambda$ has no partially stable components. Then
\[\dim_\R\CM_{\CB\Lambda}^{1/r}\leq\dim_\R\CM_\Gamma-2.\]
\end{obs}
The next claims summarize the behavior of $\cW$ and $\mathbb{L}_i$ under forgetful maps:
\begin{obs}\label{obs:Witten_pulled_back}
Because $\mathcal{W} = \text{For}^*_{B',I'}\mathcal{W}$ for all $B',I' \subseteq \mathbb{Z}$ (see \cite[Equation (4.3)]{BCT1}), there is a canonical isomorphism between the Witten bundle $\cW$ on $\oCM_{\Gamma}^{1/r}$ and the pullback $F_{\Gamma}^*\cW$ of the Witten bundle from $\oCM_{\CB\Gamma}^{1/r}$.
\end{obs}
\begin{obs}\label{obs:Li_pulled_pre}
If the component that contains the internal tail labeled~$i$ in $\CB\Gamma$ is stable, then $\CL_i\to\oCM_\Gamma^{1/r}\simeq F_\Gamma^*(\CL_i\to \oCM_{\CB\Gamma}^{1/r})$ canonically.
\end{obs}
Indeed, $\CL_i$ is pulled back along the maps appearing in the definition $F_\Gamma$, unless the stable component containing the internal tail $i$ becomes unstable along the way.  This happens only if, in $\Gamma$, the vertex containing tail $i$ has a unique internal half-edge and all boundary half-edges are illegal of twist zero.

We are now ready to define the canonical boundary conditions we seek.  In what follows, for an orbifold vector bundle $E$ over an orbifold with corners $M$, we denote by $C_m^{\infty}(E)$ the space of smooth multisections; see the appendix for our conventions and notation regarding multisections.
\begin{definition}\label{def:canonical for Witten}
Let $\Gamma$ be a smooth graded graph with only open vertices, let $U_+$ be as in \eqref{eq:U_+}, and let $U$ be a set containing $\partial^0\oPM_\Gamma\cup U_+$.  Then a smooth multisection $s$ of $\cW$ over $U$ is \emph{canonical} if
\begin{enumerate}
\item $s$ is positive, and
\item $s$ is \emph{pulled back from the base}, in the sense that for any $\Lambda\in \partial^0\Gamma$, we have
\begin{equation}\label{eq:pb_Witten}s|_{\CM_\Lambda^{1/r}}=F_\Lambda^*s^{{\CB\Lambda}}
~\text{for some }s^{\CB\Lambda}\in C_m^\infty(\CM_{\CB\Lambda}^{1/r},\cW).\end{equation}
\end{enumerate}
In case $U=\oPM_\Gamma$, we say that $s$ is a {\it global canonical multisection}. A section that is positive and satisfies the second item of the definition for $\Lambda$ in some set $A\subseteq \partial^0\Gamma$ is said to be \emph{canonical with respect to $\bigcup_{\Lambda\in A} \CM_\Lambda^{1/r}.$}
\end{definition}

\begin{definition}\label{def:canonical for L_i}
A smooth multisection $s$ of $\CL_i\to\partial\oPM_{\Gamma},$ for $\Gamma$ smooth and graded, is called \emph{canonical} if for every
$\Lambda\in\partial^0\Gamma$ with $\CB\Lambda$ stable, we have
\begin{equation}\label{eq:pb}s|_{\CM_{\Lambda}^{1/r}}=F_{\Lambda}^*s^{\CB\Lambda}\end{equation}
for some $s^{\CB\Lambda}\in C_m^\infty(\oPM_{\CB\Lambda},\cW)$.  A global multisection is canonical if its restriction to $\partial\oPM_{\Gamma}$ is canonical. If $A\subseteq\partial^0\Gamma$, then $s$ is \emph{canonical with respect to $\bigcup_{\Lambda\in A} \CM_\Lambda^{1/r}$} if~\eqref{eq:pb} holds for any $\Lambda\in A$ with a stable base.
\end{definition}

Throughout what follows, we consider the (uniquely determined) multisections $\{s^{\CB\Lambda}\}_{\Lambda\in\partial^0\Gamma}$ in the definitions above as part of the information of a canonical multisection, both for $\CL_i$ and for $\cW$.  For a bundle obtained as a direct sum of $\cW$ and cotangent line bundles, we define a canonical multisection to be a direct sum of canonical multisections of the corresponding bundles.

\begin{obs}\label{obs:sum_of_canonical}
Sums of canonical multisections of $\cW \to \oPM_{\Gamma}$ or $\CL_i$, or multiples of canonical multisections by a positive scalar, are canonical.

This is easy to see for multisections of $\CL_i$.  The proof for multisections of $\cW$ is also straightforward, except for verifying that the sum of positive multisections is positive.  For this, suppose that $s_1$ and $s_2$ are positive.  For any $u\in\CM_\Lambda^{1/r}\subseteq\partial^+\oCMr$, there are open $\Lambda$-neighborhoods $U_i$ and families of intervals $I_{i,h}$ for $U_i$ on which $s_i$ is positive. For each $h$, the intersection $I_{1,h}(u)\cap I_{2,h}(u)$ contains an interval around the node $n_h.$ Thus, in some open set $V\subseteq U_1\cap U_2,$ which is itself a $\Lambda$-neighborhood, the collection $\{\hat{I}_h=I_{1,h}(-)\cap I_{2,h}(-)\}$ forms a $\Lambda$-family of intervals for $V$. $s_1+s_2$ is then $(V,\hat{I})$-positive.
\end{obs}

\subsection{Definition of intersection numbers}\label{subsec:int_numbers}
We refer the reader to the appendix for our definition of relative Euler class and of the weighted cardinality of the zero locus of a multisection.
Our first main theorem is the following:
\begin{thm}\label{thm:int_numbers_well_defined}
Let $E\to\oCMr$ be the bundle
\[E := \bigoplus_{i\in [l]}\CL_i^{\oplus d_i}\oplus\cW,\] and assume that $\text{rank}(E)=\dim\oCMr.$
Then there exists $U_+=U\cap\oPMr$, where $U$ is a neighborhood of $\partial^+\oCMr$ such that $\oPMr\setminus U$ is a compact orbifold with corners, and a nowhere-vanishing canonical multisection $\mathbf{s}\in C_m^\infty(U_+\cup\partial^0\oCMr,E)$.
Thus, one can define, using the canonical relative orientation of $\mathcal{W}$ specified in \cite[Theorem 5.2]{BCT1}, the Euler number
\[\int_{\oPMr}e(E ; \mathbf{s})\in\mathbb{Q}.\]
The result is independent of the choice of $U$ and $\mathbf{s}$.
\end{thm}

An equivalent formulation is that transverse global canonical multisections exist, and that the weighted number of zeroes of such a multisection is independent of the choice of the transverse global canonical multisection.

\begin{definition}
With the above notation, when $\text{rank}(E)=\dim\oCM_\Gamma^{1/r}$, we define the \emph{open $r$-spin intersection numbers} (or {\it correlators})~by
\[
\left\langle \prod_{i\in [l]}\tau^{a_i}_{d_i}\sigma^k \right\rangle^{1/r,o} = \int_{\oPMr}e(E ; \mathbf{s})\in\mathbb{Q},
\]
for any canonical multisection $\mathbf{s}$ that does not vanish at $\partial^0\oCMr\cup U_+$, where~$U_+$ is as above.  When $\text{rank}(E)\neq\dim\oCM_\Gamma^{1/r}$, the integral is defined to be zero.
If $d_i=0$ for some $i$, we typically omit the subscript $0$ in the symbols $\tau^a_0$. In the unstable range $2l+k\le 2$, we define the open $r$-spin correlators to be zero.
\end{definition}
%
%In the case $k=0$ without boundary marked points, one can say more.
%
\begin{prop}\label{prop:no_zeroes_for_Witten_without_bdry_markings}
If $k=0$ then there exists a global canonical nowhere-vanishing $s\in C_m^\infty(\oPM_{0,0,\vec{a}},\cW).$
In particular, all open $r$-spin correlators on the moduli space of disks without boundary marked points vanish.
\end{prop}

The proof of the theorem and the proposition are relegated to Section \ref{sec:constructions_bc} below. However, we do prove here the key geometric proposition required for the construction, in order to illuminate the necessity of working with $\oPMr$ rather than $\oCMr$ and of constraining all boundary twists to be $r-2$:
\begin{prop}\label{prop:pointwise_positivity}
The following positivity claims hold:
\begin{enumerate}
\item
Let $C$ be a stable or a partially stable $r$-spin disk. Then there exists $w\in\cW_{C}$ that evaluates positively at each \stronglypositive~component~of~$C$.
\item
Let $U$ be a contractible $\Lambda$-neighborhood of $\Sigma\in\partial^+\oCM_\Gamma$, where $\Gamma$ is connected, smooth, and graded and $\Lambda\in\partial^!\Gamma$ has no internal edges nor illegal boundary half-edges of twist zero.  Let $(I_h)_{h\in\Pos(\Lambda)}$ be a $\Lambda$-family of intervals for $U.$ Then there exists a $(U,I)$-positive section $s$.
\end{enumerate}

\end{prop}

Our basic tool for constructing positive sections in the proof of Proposition~\ref{prop:pointwise_positivity} is the following lemma:

\begin{lemma}\label{lem:surj_of_eval}
Suppose $C$ is a smooth pre-stable $r$-spin disk with a lifting, or a smooth pre-stable graded $r$-spin sphere.  Let $h=deg(|J|)$, and let $\alpha$ and $\beta$ be non-negative integers with $\alpha+ 2\beta \leq h+1$.  Then for any distinct boundary points $p_1,\ldots, p_\alpha,$ and internal points $p_{\alpha+1},\ldots, p_{\alpha+\beta}$, the total evaluation map
\[\bigoplus_i \ev_{p_i}:\cW_{C}\to\bigoplus_{i=1}^\alpha |J|_{p_i}^{\tilde\phi}\oplus\bigoplus_{i=\alpha+1}^{\alpha+\beta} |J|_{p_i}\] is surjective; here we use the canonical identification $H^0(J) = H^0(|J|)$ to identify the fiber of $\cW$ with $H^0(|J|)$. In particular, for any $\alpha$ distinct boundary points and $\beta$ distinct internal points, with $\alpha+2\beta=h$, there exists a unique (up to a real scalar) nonzero section of $\cW$ vanishing at all of them. %The section is unique up to multiplication by a real non zero scalar.

If $C$ has a Ramond marking $p$ and no point with $-1$ twist, in particular if it has a contracted boundary, then the evaluation map at $p$ is surjective.
\end{lemma}
\begin{proof}
Let $n:= \alpha+2\beta$, and let $\{q_1, \ldots, q_n\} \subset C$ consist of the points $p_i$ together with the conjugates of $p_{\alpha+1}, \ldots, p_{\alpha+\beta}$.  Then the short exact sequence
\[0\to |J|\left(-\sum_{i=1}^n [q_i]\right)\to |J|\to\bigoplus_i |J|_{q_i}\to 0\]
induces the long exact sequence
\begin{equation}
\label{eq:evles}
\ldots\to H^{0}(|J|)\to\bigoplus_{i=1}^n |J|_{q_i}\to H^{1}(|J|(-\sum [q_i]))\to\ldots,
\end{equation}
and $\oplus_i \ev_{q_i}$ is obtained from \eqref{eq:evles} by taking $\tilde{\phi}$-invariant parts.  We have
\[h^1\left(|J|\left(-\sum q_i\right) \right) = h^0\left(|J|^{\vee} \left(\sum q_i\right) \otimes \omega_{|C|}\right) = 0,\]
since the latter has negative degree.  This proves the first part of the claim, where the uniqueness of the nonzero section follows from the fact that, if $v_1$ and $v_2$ both vanish at $p_1, \ldots, p_{\alpha+\beta}$, then some linear combination $\lambda v_1 + \mu v_2$ with $\lambda,\mu \in \mathbb{R}^*$ would have $\deg(|J|)+1$ zeroes, which is impossible for a nonzero section.  The proof of the second claim is similar.
\end{proof}
%
%Using this, we are prepared to prove the key proposition:
%
\begin{proof}[Proof of Proposition~\ref{prop:pointwise_positivity}]
For the first item, note that each connected component of $C$ may have at most one \stronglypositive component. Since the Witten bundle decomposes as a direct sum over different connected components, it is enough to prove the claim for a connected $\Sigma$ that is \stronglypositive\!\!\!.
We use Lemma \ref{lem:surj_of_eval}. If $C$ is closed and has a contracted boundary node, we choose $w$ to be an arbitrary vector that evaluates positively at that node. If $C$ is open and it has a single internal Ramond marked point and no boundary marked points, then $\cW_C$ is a real line and we choose $w$ to be any nonzero element that evaluates positively at one boundary point (hence at any boundary point, since $\deg(|J|_C)=0$ in this case). Note that $w$ is $\text{Aut}(C)$-invariant.

For the second item, assume first that $\Aut(\Lambda)$ is trivial. We also assume that $\rk(\cW)>0$ since otherwise the statement is trivially true. Fix a point in $U$ corresponding to a smooth, open $C$.  Note that if $w\in\cW_C$ is a nonzero vector that is positive at some boundary point $p$ and negative at another boundary point $q$, then the total number of zeroes of $w$ in the boundary arc from $p$ to $q$ (or from $q$ to $p$), plus the number of boundary markings along this arc, is odd. %Parity considerations show that we could have equivalently stated the same criterion but with the arc from $q$ to $p$.

Our strategy will thus be to show that one can choose a set $Z=Z_C$ of $z$ boundary points in $A=\partial\Sigma\setminus\left(\sqcup_{h\in\Pos(\Lambda)} I_{h}(C)\right)$, where
\[z\leq \rk\cW-1,\;\;\;\;\;\;z\equiv\rk(\cW)-1\!\mod 2,\]
so that in any connected component $K$ of $A$ (which is an interval), the total number of boundary markings in $K$ plus $|Z\cap K|$ is even. We shall then choose an arbitrary set $Z^c$ of $\frac{\rk(\cW)-1-z}{2}$ internal points.   By Lemma \ref{lem:surj_of_eval}, one can find a nonzero vector $w\in\cW_C$ with simple zeroes at $Z\cup Z^c,$ and, possibly by replacing $w$ with $-w,$ this vector satisfies the positivity constraints we want.

Since $U$ is contractible, we can then construct a diffeomorphism $\phi:\CC(U')\to U'\times C$ between the universal curve restricted to $U'=U\cap\oPM_{\Gamma}$ and $U'\times\Sigma$, which takes the fiber of $[C']$ to $([C'],\Sigma)$ and satisfies $\phi(I(C')_h)=([C'],I(C)_h)$.  We then define the section at $[C']$ as the unique section, up to rescaling, whose zeroes are $\phi^{-1}([C'],Z\cup Z^c).$ The real scaling factor is chosen so that the resulting section $s$ is smooth and $s_{[C]}=w$. By construction $s$ is $(U,I)$-positive.%, if we can construct~$w$.

It is left to construct $w.$ The first step is the following combinatorial claim, whose proof we give at the end of the proposition:
\begin{lemma}
\label{lem}
For a vertex $v$ of $\Lambda$, write $\tdeg(v) = \rk(\cW_v)+m_v-1$, where $m_v$ is the number of legal half-edges of $v$ with twist smaller than $r-2.$ Let $k_v$ denote the number of legal half-edges of $v$ with twist $r-2$, which by the assumptions on $\Lambda$ must be tails.
\begin{enumerate}
\item\label{it:lem_lem_`} For every vertex $v$ we have that $\tdeg(v) = k_v \! \mod 2$.
\item\label{it:lem_lem2} If $\rk(\cW)>0$, then for every vertex $v$ we have $\tdeg(v) \geq 0$.
\item\label{it:lem_lem3} Furthermore, $
\rk(\cW)-1=\sum_{v\in V(\Lambda)}\tdeg(v).
$
\end{enumerate}
\end{lemma}
%If $\rk(\cW_v) = 0$, then $\tdeg(v)$ is non-negative and even.  Furthermore, we have
%\begin{equation}\label{eq:tree_comb_claim}
%\rk(\cW)-1\geq\sum_{v\; | \;\rk(\cW_v)>0}\tdeg(v)
%\end{equation}
%and
%\begin{equation*}
%\rk(\cW)-1=\sum_{v\; | \; \rk(\cW_v)>0}\tdeg(v) \! \mod 2.
%\end{equation*}
%\end{lemma}
%We defer the proof of this lemma until later, but assuming it for now,
Using the lemma we can describe the strategy for constructing the vector~$w$.  We define $Z$ by its intersections $Z\cap K$ with each connected component $K$ of $A$.  We choose $Z_K$ to be a set of $k_K\in\{0,1\}$ points in $K,$ where $k_K$ is $|K\cap\{x_i\}_{i\in B}| \mod 2.$ % defined by
%\[k_K \equiv |K\cap\{x_i\}_{i\in B}| \mod 2.\]
Thus, in order to finish, we must show that
\begin{equation}\label{eq:intermidiate_for_positivity}
z:=\sum_{\substack{\text{$K$ a connected}\\\text{component of $A$}}}k_K\leq \rk(\cW)-1\quad\text{and}\quad z\equiv\rk(\cW)-1\mod 2.
\end{equation}

Any interval $K$ is associated a vertex of $\Lambda$, the vertex $v$ that corresponds to the component $C_v$ of the normalization of $C$ containing $K$. Denote this relation by $K\to v$.  By item \ref{it:lem_lem3} of Lemma~\ref{lem}, equation \eqref{eq:intermidiate_for_positivity} will follow from
\begin{align}
&\forall v\in V(\Lambda),~\sum_{K\to v}k_K \leq \tdeg(v)\label{eq:main_for_positivity},\\
&\forall v\in V(\Lambda),~\sum_{K\to v}k_K = \tdeg(v) \mod 2.\label{eq:main_for_positivity2}
\end{align}
And indeed, by the definition of $k_K$ and the first item of Lemma~\ref{lem}, we have
\begin{equation}
\label{eq:2_for_pos}
\sum_{K\to v}k_K  \equiv k_v \equiv \tdeg(v) \mod 2,
\end{equation}
which implies equation~\eqref{eq:main_for_positivity2}.  Moreover, for any vertex $v$ with $\rk(\cW_v)=0$, the second item of Lemma~\ref{lem} implies that equation~\eqref{eq:main_for_positivity} holds.

We are left with proving equation~\eqref{eq:main_for_positivity} when $\rk(\cW_v)>0$. Here, the boundary twists being $r-2$ is crucial.  Write $k_v$, $m_v$, and $n_v$ for the numbers of legal boundary points with twist $r-2$, legal boundary points with twists less than $r-2$, and illegal boundary points of $C_v$, respectively. By assumption, the twists of the illegal points are nonzero, and hence are at least $2$. Thus,
\begin{align}
\label{eq:1_for_pos}
\tdeg(v) &= \rk(\cW_{C_v})+m_v-1\\
\nonumber&\geq \left\lceil\frac{2n_v+(r-2)k_v-(r-2)}{r}\right\rceil+m_v-1\\
\nonumber&\geq \left\lceil\frac{2\min(k_v,n_v)+(r-2)\min(k_v,n_v)-(r-2)}{r}\right\rceil+m_v-1\\
\nonumber&= \left\lceil\frac{r\min(k_v,n_v)-(r-2)}{r}\right\rceil+m_v-1\\
\nonumber&= \min(k_v,n_v)+m_v-1=\min(k_v+m_v-1,n_v+m_v-1),
%\nonumber&=\min(k_v+m_v-1,n_v+m_v-1)
\end{align}
where we have used the integrality of $\tdeg(v).$
%In addition,
%\begin{equation}
%\label{eq:2_for_pos}
%\sum_{K\to v}k_K  \equiv k_v \equiv \tdeg(v) \mod 2,
%\end{equation}
%again using the parity constraints of \ref{prop:compatibility_lifting_parity}
The number of intervals that may contain a boundary marked point of twist $r-2$ is, of course, no more than $k_v.$

There now are several cases to check in order to prove \eqref{eq:main_for_positivity}:
\begin{itemize}
    \item The case $n_v\geq k_v$.  In this case, if $m_v \geq 1$, then $\tdeg(v)\geq k_v\geq \sum_{K\to v}k_K$, proving the claim for $v$.  If $m_v = 0$, then $\tdeg(v)\geq k_v-1$, but then, by~\eqref{eq:2_for_pos}, we again have $\tdeg(v)\geq k_v,$ so the claim still holds.
\item The case $k_v\geq n_v+2.$ Now ~\eqref{eq:1_for_pos} can be strengthened to give the bound
\begin{align*}
    \tdeg(v) &\geq \left\lceil\frac{2n_v+(r-2)(n_v+2)-(r-2)}{r}\right\rceil-1+m_v\\
    &\geq m_v-1+n_v+1=m_v+n_v,
    \end{align*}
where again we have used the integrality of $\tdeg(v).$ Since $n_v+m_v=|\{K \; |\; K\to v\}|\geq \sum_{K\to v}k_K$, equation \eqref{eq:main_for_positivity} holds in this case as~well.
%Note that exactly the same argument would work if $k_v=n_v+1$, as long as at least one of the $m_v$ legal points or one of the internal points has a positive twist.
\item The case $k_v=n_v+1.$ If at least one of the $m_v$ legal points or of the internal points has a positive twist, then the previous argument works here as well. Thus, suppose that $k_v=n_v+1$, all internal points have twist zero, and all $m_v$ legal points of twist less than $r-2$ have twist $0$. In this case, $\tdeg(v)=n_v+m_v-1$, and we can break into three sub-cases:
\begin{itemize}
    \item If $m_v\geq 2$ then clearly $\tdeg(v)\geq k_v.$
    \item If $m_v=1$, then $\tdeg(v)=n_v=k_v-1$, but this contradicts ~\eqref{eq:2_for_pos}.
    \item  If $m_v=0$, then $\tdeg(v)=n_v-1=k_v-2$.  But in this case, there must be at least one interval between illegal boundary points that contains more than one legal marking (of twist $r-2$), and hence, $\sum_{K\to v}k_K\leq k_v-2=\tdeg(v)$ and again we are done.  
\end{itemize}

\begin{figure}\label{fig:finding_positive_one_component}
 \centering
\begin{tikzpicture}[scale=0.4]
\draw (0,0) circle (3 cm);

\draw(-0.2,2.8) -- (0.2,3.2);
\draw(-0.2,3.2) -- (0.2,2.8);

\filldraw[fill=black, draw=black] (1.3017, 2.7029) circle (0.2 cm);

\filldraw[fill=black, draw=black] (2.3455, 1.8705) circle (0.2 cm);

\filldraw[fill=white, draw=black] (2.9245, 0.6676) circle (0.2 cm);

\filldraw[fill=white, draw=black] (2.9245, -0.6676) circle (0.2 cm);

%\draw(2.3455 -0.2, -1.8705 -0.2) -- (2.3455 + 0.2, -1.8705+0.2);
%\draw(2.3455 -0.2, -1.8705 +0.2) -- (2.3455 + 0.2, -1.8705-0.2);

\filldraw[fill=black, draw=black] (1.3017, -2.7029) circle (0.2 cm);

\draw(-0.2,-2.8) -- (0.2,-3.2);
\draw(-0.2,-3.2) -- (0.2,-2.8);

\filldraw[fill=black, draw=black] (-1.5, 2.598) circle (0.2 cm);

\filldraw[fill=white, draw=black] (-2.598, 1.5) circle (0.2 cm);

\filldraw[fill=black, draw=black] (-3,0) circle (0.2 cm);

\filldraw[fill=black, draw=black] (-2.598, -1.5) circle (0.2 cm);

\filldraw[fill=white, draw=black] (-1.5, -2.598) circle (0.2 cm);

\end{tikzpicture}
 \caption{The scheme by which we create a positive section, for simplicity in the case $m_v=0$, using the choice of zeroes.  Filled circles denote legal points (of twist $r-2$), unfilled circles denote illegal points, and $\times$s denote the placement of zeroes.}
\end{figure}
\end{itemize}
Equations~\eqref{eq:main_for_positivity} and \eqref{eq:main_for_positivity2} are now proven, and the proposition follows in the case $|\Aut(\Lambda)|=1$.  If $|\Aut(\Lambda)|>1$, then by our assumptions on $\Lambda,$ $\Gamma$ and internal labelings, it must hold that $k>1$ and some boundary points have the same label. Let $\widetilde U$ be a cover of $U$ obtained by changing the labels to injective labels, so that $|\text{Aut}(\widetilde{U})| = 1$ and
$\widetilde U/G\cong U$ for a finite group $G$.  Construct on $\widetilde{U}$, a section $\tilde s$ that satisfies the requirements of the proposition. Then, if $\tilde{q}: \widetilde{U}/G \rightarrow U$ is the quotient map, the section $s$ defined as the averaging $s(u) = |G|^{-1}\sum_{v\in \tilde{q}^{-1}(u)}\tilde{s}(v)$ satisfies the requirements. % This completes the proof of Proposition~\ref{prop:pointwise_positivity}.
\end{proof}
%Having used Lemma~\ref{lem} in the above construction, we now return to give the proof of this statement.
\begin{proof}[Proof of Lemma~\ref{lem}]
We start with the first claim, and we split into the cases of $2\mid r$ and $2\nmid r$.
When $2\mid r$, by \eqref{parity} and \eqref{eq:bundle_rk}, the quantity $k_v+m_v$, which is the number of legal (boundary) half-edges at $v$, is congruent modulo $2$ to
%\[\frac{2 \sum a_i + \sum b_j + 2}{r} = 1\]
\[\frac{2 \sum a_i + \sum b_j - (r-2)}{r}+1=\rk(\cW_v)+1.\]

If $2\nmid r$, then Proposition \ref{prop:compatibility_lifting_parity} implies that the legal half-edges are those whose twist is odd, of which there are $k_v+m_v$.  By equation \eqref{eq:bundle_rk} and $2\nmid r$,
\[\rk(\cW_v) +1= \frac{2 \sum a_i + \sum b_j +2}{r} = |\{j \; | \; 2\nmid b_j\}|~\text{mod}~2=k_v+m_v~\text{mod}~2.\]
We turn to the second statement.  First, note that $\tdeg(v)<0$ only if $\rk(\cW_v)=m_v=0$. In this case, by the previous part, $2\nmid k_v$ and is hence positive. Since $\rk(\cW)>0,$ $v$ must touch an edge.
Furthermore, since $\Lambda$ does not contain an illegal half-edge of twist $0$, and $m_v=0$, $v$ must touch the illegal side of this edge, whose twist must be positive. The sum of the twists of this %illegal 
half-edge and the legal tails of $v$ is more than $r-2$, hence using equation \eqref{eq:bundle_rk} we find that $\rk(\cW_v)>0$, contradicting the assumptions. We conclude that $\tdeg(v)\geq0.$

For the last claim, we use the decomposition properties of the $\cW$ to conclude
\[\rk(\cW)-1 = \sum_{v\in V(\Lambda)}(\rk(\cW_v)-1) + |E\setminus E^{\text{Ramond}}|=\sum_{v\in V(\Lambda)}\tdeg(v),\]
where $E^{\text{Ramond}}$ are the Ramond boundary edges and the right equality is a consequence of the fact that any boundary edge that is not Ramond has exactly one legal half-edge of twist is less than $r-2,$ by the assumption on $\Lambda$.  %Therefore, by the first claim,
%\[\sum_{v\in V(\Lambda)}\tdeg(v)\geq\sum_{v\;|\;\rk(\cW_v)>0}\tdeg(v)\]
%and
%\[\sum_{v\in V(\Lambda)}\tdeg(v)=\sum_{v\; |\; \rk(\cW_v)>0}\tdeg(v)\mod 2,\]
%completing the proof of Lemma~\ref{lem}.
\end{proof}

A na\"ive attempt to define the positivity boundary conditions would have been to require that, for all $C\in\partial^+\oCMr$, the evaluation of $w_C$ near each illegal boundary half-node of positive twist is positive. While this idea motivates our definition, the next example shows why it fails, and why we need to consider $\oPMr.$ 
%One aspect of the proof of Proposition~\ref{prop:pointwise_positivity} that perhaps deserves further comment is the need for working with space $\oPMr$.  To understand this, note that there is a na\"ive way one might try to define the positivity boundary conditions: one could require that, for any $C\in\partial^+\oCMr$, the evaluation of $w_C$ is positive near each illegal boundary half-node of positive twist (or at least, its evaluation is positive near a nonempty subset of such half-nodes, and is non-negative around the rest).  Indeed, this na\"ive idea is the motivation for our definition of positivity, and it works for corners of the moduli space of low codimension.  However, the next example illustrates why this simpler approach fails in general.

\begin{ex}\label{ex:perm is needed}

Figure \ref{fig:ex} shows a nodal disk made of four disk components meeting at nodes. The component $A$ contains three illegal boundary marked points with the same positive twist. It should be viewed as arising in a partial normalization of a nodal graded $r$-spin disk, where the labeled illegal points are half-nodes.  Suppose that the nodes of $B$, $C$, and $D$ are Neveu--Schwarz, that their legal sides are in $A$, and have twist $a<r-2$. Choose internal markings with twists so that the ranks of the Witten bundles on $A$, $B$, $C$, $D,$ are $0$, $1$, $1$, $1$, and so that the internal twists for the disks $B$, $C$, $D$ are the same.

The na\"ive boundary conditions on $\partial^+\oCMr$ require the evaluations at the illegal half-nodes of $B$, $C$, and $D$ to be positive, and, after smoothing the nodes of $B$, $C$, and $D$, the evaluations at the illegal marked points of $A$ also need to be positive.  But then, suppose we smooth the node where $D$ meets $A$. The Witten bundle for the smoothed component has rank one and degree zero, and hence the restriction of the section to this component either vanishes nowhere or is $0$. If it is nonzero, then it is easy to see that it must evaluate negatively at one illegal point of the smoothed component; thus, the restriction of the section to the smoothed component must be $0$.  Using continuity this implies, in the nodal limit, that it must vanish on $D$.  Analogously, it must vanish on $B$ and $C$, as well. Thus, a section that is positive on $B$, $C$, and $D$ cannot be extended to a section that is positive, in the na\"ive sense, in a neighborhood.

\begin{figure}
 \centering
  \begin{tikzpicture}[scale=0.75]

  \draw (0,0) circle (1 cm);
  \draw (-1.4142136, 1.4142136) circle (1 cm);
  \draw (-1.4142136, -1.4142136) circle (1 cm);
  \draw (2,0) circle (1 cm);
  \node at (0,0) {$A$};
  \node at (-1.4142136, 1.4142136) {$B$};
  \node at (-1.4142136, -1.4142136) {$C$};
  \node at (2,0) {$D$};

  \filldraw[fill=white, draw=black] (-1,0) circle (0.1 cm);
  \filldraw[fill=white, draw=black] (0.382683,0.92388) circle (0.1 cm);
  \filldraw[fill=white, draw=black] (0.382683,-0.92388) circle (0.1 cm);

  \end{tikzpicture}
 \caption{A configuration where the na\"ive positivity fails.}
 \label{fig:ex}
\end{figure}
\end{ex}

The na\"ive approach fails since for nodal disks there may be ``fewer zeroes in hand" than in the smoothing.  We work in $\oPMr$ to avoid this problem.

\subsection{Examples of calculations}\label{subsec:calculations}

Here we present some examples of calculations of the open $r$-spin intersection numbers using our definition.

\begin{ex}\label{ex:tatb}
Consider the integral $\langle\tau^a\tau^{b}\rangle^{1/r,o}_0,$ with $b=r-a-1.$ In this case canonical sections extend to the boundary in a non-vanishing way, so we can work on  $\oCM_{0,0,\{a,b\}}^{1/r}$ rather than on its subset $\oPM_{0,0,\{a,b\}}$. 

One end of the interval $\oCM_{0,0,\{a,b\}}^{1/r}$ is the surface obtained by contracting the boundary of the disk, while the other end is the nodal surface composed of two disks, each with one internal marking, connected by a boundary node. Here $\text{rk}_\R(\cW)=1,$ and $\text{deg}(\cW)=0.$ Thus, for $C\in\CM_{0,0,\{a,b\}}^{\frac{1}{r}}$, any nonzero $s_{\Sigma}\in\cW_{C}$ is nowhere-vanishing on $\partial\Sigma \subseteq C$ when considered as a global section over $C$.

Take $s\in C^\infty(\cW\to\CM_{0,0,\{a,b\}}^{\frac{1}{r}})$ such that, for every $C\in\CM_{0,0,\{a,b\}}^{\frac{1}{r}}$, the restriction of $\left(s_{C}\right)|_{\d\Sigma}$ is positive.  Rescaling if necessary, $s$ can be extended to $\oCM_{0,0,\{a,b\}}^{\frac{1}{r}},$ and by continuity, it agrees with the grading at the endpoints of the moduli space.  Hence $s$ is a canonical section.  Since $s$ is nowhere-vanishing, we have
$\langle\tau^a\tau^b\rangle^{1/r,o}_0=0.$
%We shall later see, either as a consequence of the topological recursions or as a direct implication of Proposition \ref{prop:no_zeroes_for_Witten_without_bdry_markings}, that when $k=0$, all intersection numbers are zero.
\end{ex}

\begin{ex}\label{ex:t1s^2}
Consider $\langle\tau^1\sigma^2\rangle^{\frac{1}{r},o}_0$.  Again we extend the section to $\oCM_{0,2,\{1\}}^{1/r}$, which is a closed interval. Each endpoint is a nodal disk composed of two disks meeting at a node, in which one disk contains the boundary markings and a half-node $n_-$ and the other contains the internal marking and a half-node $n_+$; here, $n_-$ is the illegal half-node. The endpoints differ in the cyclic order of $x_1,x_2,n_-$. Call the surface with order $x_1,x_2,n_-$ $C_1,$ and call the other $C_2$.

Again, for $C\in\CM_{0,2,\{1\}}^{\frac{1}{r}},$ an element $s_{C}\in\cW_{\Sigma}\setminus\{0\}$ does not vanish on $\partial\Sigma \subseteq C$, so its direction agrees with the grading in exactly one of the two arcs $x_1\to x_2$ or $x_2\to x_1.$
If $s$ is canonical then at $n_-$, the section $s_{\Sigma_i}$ agrees with the grading. Thus, $s_{\Sigma_1}$ agrees with the grading along the arc $x_2\to x_1$ that contains $n_-$, and $s_{\Sigma_2}$ agrees with the grading along the other arc. Hence, a canonical section must vanish at some point of $\CM_{0,2,\{1\}}^{\frac{1}{r}},$ and we have $\langle\tau^1\sigma^2\rangle^{1/r,o}_0=\pm 1.$

To determine the sign, suppose $\CM_{0,2,\{1\}}^{\frac{1}{r}}$ is oriented from $C_1$ to $C_2$; in the notations of \cite[Section 3.3]{BCT1}, this orientation is  $\tilde{\mathfrak{o}}^{(1,2)}.$ Then $\cW\to\CM_{0,2,\{1\}}^{\frac{1}{r}}$ has a canonical relative orientation, given by the section:
\[\sigma= \left((-1)^r\frac{i(\bar{z_1}-z_1)dw}{(w-z_1)(w-\bar{z}_1)} (\frac{(x_{2}-x_{1})dw}{(w-x_{1})(w-x_{2})})^{r-2}  \right)^{\frac{1}{r}};\]
see \cite[Definition 5.10 and Construction/Notation 5.7]{BCT1}. $\sigma^{\otimes r}$ is positive with respect to the canonical orientation of $T^*\partial\Sigma$ when $w$ is in the arc $x_1\to x_2,$ and when $r$ is even, we take the real $r$th root that is positive with respect to the grading for $w$ in the arc $x_1\to x_2.$ By taking the limit $C\to C_i$, we see that canonical boundary conditions correspond to a negative vector in $\cW_{C_1}$ and to a positive one in $\cW_{C_2}.$ Therefore, the intersection number is $+1.$
\end{ex}

\begin{ex}\label{ex:s4s3}
When $r=2$, a dimension count shows that $\langle\sigma^3\rangle^{\frac{1}{2},o}_0=\pm 1.$ The definition of orientation then shows that $\langle\sigma^3\rangle^{\frac{1}{2},o}_0=-1$.  (Note that this is a different convention than the one used in \cite{PST14}.)
\end{ex}

\subsection{The case $r=2$}

In the case where $r=2$, our open $r$-spin intersection numbers are related to the intersection numbers on $\oCM_{0,k,l}$ defined in \cite{PST14}.  In particular, the relationship is as follows.

\begin{prop}\label{prop:r=2}
We have $\<\prod_{i\in [l]}\tau^0_{d_i}\sigma^k\>^{\frac{1}{2},o}_0=(-2)^{\frac{k-1}{2}}\<\prod_{i\in [l]}\tau_{d_i}\sigma^k\>^{o}_0$.
In other words, the Neveu--Schwarz sector of our theory agrees with the genus-zero part of the theory defined in \cite{PST14}.
\end{prop}
\begin{proof}
When $r=2$ and there are no Ramond markings, the Witten bundle is of rank zero. There is a map $f=\text{For}_{\text{spin}}:\oPM_{0,k,\{0\}_{i\in[l]}}\to\oCM_{0,k,l}$, which is generically one-to-one. Note that when $r=2$ and there are no Ramond insertions $\oPM_{0,k,\{0\}_{i\in[l]}}=\oCM_{0,k,\{0\}_{i\in[l]}}.$  Now, $\<\prod_{i\in [l]}\tau_{d_i}\sigma^k\>^{o}$ were defined using a notion of special canonical multisections, and a simple verification of definitions shows that, if $s$ is special canonical in the sense of \cite{PST14}, then $f^*s$ is canonical in the sense of this paper.\footnote{If $s$ is canonical in the sense of \cite{PST14}, it is not automatically the case that $f^*s$ is canonical in the sense of this paper, because of a difference in the choice of boundary labels for the base. In our case, a canonical multisection may have to be more symmetric with respect to renaming boundary markings.} Furthermore, Lemmas 3.53 and 3.56 of \cite{PST14} show that nowhere-vanishing special canonical boundary conditions $s$ exist, and $\<\prod_{i\in [l]}\tau_{d_i}\sigma^k\>^{o}_0 = 2^{-\frac{k-1}{2}}\int_{\oCM_{0,k,l}}e\left(\bigoplus_{i\in [l]}\CL_i^{\oplus d_i};s\right)$. On the other hand, Lemma~\ref{lem:existence} below implies $\<\prod_{i\in [l]}\tau^0_{d_i}\sigma^k\>^{\frac{1}{2},o}_0 = \int_{\oPM_{0,k,\{0\}_{i\in[l]}}}e\left(\bigoplus_{i\in [l]}\CL_i^{\oplus d_i};f^*s\right)$. Finally, the $(-1)^{\frac{k-1}{2}}$ comes from the slight difference in our choice of orientation in the current work compared to the choice in \cite{PST14}; see \cite[Remark 3.13]{BCT1}.
\end{proof}

\section{Open topological recursion relations}\label{sec:recursions}
The open $r$-spin intersection numbers satisfy topological recursion relations (TRRs) that allow all genus-zero correlators to be computed.  These relations involve closed $r$-spin correlators with (at most) one marked point of twist $-1$, which we studied in \cite{BCT_Closed_Extended}. As in the closed setting with only non-negative twists, also here $R^1\pi_*\mathcal{S}$ is a bundle in genus zero, so Witten's class $c_W$ can be defined by the formula~\eqref{eq:Witten's class}.  We define {\it closed extended $r$-spin correlators} by
\begin{gather*}
\<\tau^{a_1}_{d_1}\cdots\tau^{a_n}_{d_n}\>^{\frac{1}{r},\text{ext}}_0:=r\int_{\M^{1/r}_{0,\{a_1, \ldots, a_n\}}} \hspace{-1cm} c_W \cap \psi_1^{d_1} \cdots \psi_n^{d_n}.
\end{gather*}
%
%Equipped with this definition, the topological recursion relations for the open $r$-spin intersection numbers are as follows:
%
\begin{thm}\label{thm:TRR}
\begin{itemize}
\item[(a)] (Boundary marked point TRR) Suppose $l,k\ge 1$.  Then
\begin{align*}
\< \tau_{d_1+1}^{a_1}\prod_{i=2}^l\tau^{a_i}_{d_i}\sigma^k\>^{\frac{1}{r},o}_0\hspace{-0.2cm}=&
\sum_{a=-1}^{r-2}\sum_{S \sqcup R = \{2,\ldots,l\}}\hspace{-0.1cm}\left\langle \tau_0^{a}\tau_{d_1}^{a_1}\prod_{i \in S}\tau_{d_i}^{a_i}\right\rangle^{\frac{1}{r},\text{ext}}_0
\hspace{-0.1cm}\left\langle \tau_0^{r-2-a}\prod_{i\in R}\tau^{a_i}_{d_i}\sigma^k\right\rangle^{\frac{1}{r},o}_0\\
&+\hspace{-0.1cm}\sum_{\substack{S \sqcup R = \{2,\ldots,l\} \\ k_1 + k_2 = k-1}} \hspace{-0.1cm}\binom{k-1}{k_1} \left\langle \tau^{a_1}_{d_1}\prod_{i \in S} \tau^{a_i}_{d_i}\sigma^{k_1}\right\rangle^{\frac{1}{r},o}_0 \hspace{-0.1cm}\left\langle \prod_{i \in R} \tau^{a_i}_{d_i} \sigma^{k_2+2}\right\rangle^{\frac{1}{r}, o}_0
\end{align*}
\item[(b)] (Internal marked point TRR) Suppose $l\ge 2$.  Then
\begin{align*}
\<\tau_{d_1+1}^{a_1}\prod_{i=2}^l\tau^{a_i}_{d_i}\sigma^k\>^{\frac{1}{r},o}_0\hspace{-0.2cm}=&
\sum_{a=-1}^{r-2}\sum_{S \sqcup R = \{3,\ldots,l\}}\hspace{-0.1cm}\left\langle \tau_0^{a}\tau_{d_1}^{a_1}\prod_{i \in S}\tau_{d_i}^{a_i}\right\rangle^{\frac{1}{r},\text{ext}}_0
\hspace{-0.1cm}\left\langle \tau_0^{r-2-a}\tau^{a_2}_{d_2}\prod_{i\in R}\tau^{a_i}_{d_i}\sigma^k\right\rangle^{\frac{1}{r},o}_0\\
&+\sum_{\substack{S \sqcup R = \{3,\ldots,l\} \\ k_1 + k_2 = k}} \hspace{-0.1cm}\binom{k}{k_1} \left\langle \tau^{a_1}_{d_1} \prod_{i \in S} \tau^{a_i}_{d_i}\sigma^{k_1}\right\rangle^{\frac{1}{r},o}_0 \hspace{-0.1cm}\left\langle \tau^{a_2}_{d_2}\prod_{i \in R} \tau^{a_i}_{d_i} \sigma^{k_2+1}\right\rangle^{\frac{1}{r}, o}_0
\end{align*}
\end{itemize}
\end{thm}
%The proof of this theorem is the content of the current section. 
In Section \ref{subsec:special_can} we define canonical multisections with certain special properties, which are needed for the proof.
The proof itself appears in Section~\ref{subsec:pf_trr}.  
\subsection{Special canonical multisections}\label{subsec:special_can}
\subsubsection{Coherent multisections}\label{subsec:coherent}
We work with the notation of Appendix~\ref{sec:ap_euler} for multisections.  Consider a (possibly disconnected) graded stable $r$-spin dual graph $\Gamma.$
For any twisted $r$-spin structure, we set $J' := J \otimes \O\left(\sum_t r[z_t]\right)$, where the sum is over all anchors $t$ with $\text{tw}(t)=-1$ and $z_t$ is the special point specified by $t$.  Similarly, we set $\mathcal{J}' := \mathcal{J} \otimes \O \left( \sum_t r\Delta_{z_t}\right)$, where $\Delta_{z_t}$ is the divisor in the universal curve corresponding to $z_t$.
Define an orbifold vector bundle $\TRAM_{\Gamma}$ on $\M_{\Gamma}^{1/r}$ by
$\TRAM_{\Gamma} = \bigoplus_t \sigma_{z_t}^*\mathcal{J}'$,
where the sum is again over all anchors $t$ with $\text{tw}(t)= -1$ and $\sigma_{z_t}$ is the corresponding section of the universal curve.\footnote{In the special case that no anchor is twisted $-1$, we have $\TRAM_\Gamma=\oCM_\Gamma^{1/r}$.  In general, $\TRAM_\Gamma$ is the parameter space of nodal graded genus-zero surfaces together with an element $u_t\in J'_{z_t}$ for each anchor twisted $-1$.  In what follows (except in the current section), the notation $\TRAM_\Gamma$ will not be used for graphs that are not closed, to avoid confusion.}   We use the notation $\TRAM_{\Gamma}$ both for the bundle and for its total space, and we denote by
$\TAU: \TRAM_{\Gamma} \rightarrow \M_{\Gamma}^{1/r}$ the projection.  In case $\Gamma' \in \d\Gamma$, there is a map $\TRAM_{\Gamma'}\rightarrow \TRAM_\Gamma$ between total spaces, which is an embedding on the coarse underlying level.

We denote by $\cW'$ the bundle
$\cW' = (R^0\pi_*\mathcal{J}')_+$ on $\oCM_\Gamma$, and by $\cW$ the analogous bundle $(R^0\pi_*\mathcal{J}')_+$ on $\TRAM_\Gamma$. 
(For closed components of $\Gamma$ with no contracted boundary tails, whose corresponding moduli space consists of closed $r$-spin structures with no involution, $(R^0\pi_*\mathcal{J'})_+$ just denotes $R^0\pi_*\mathcal{J'}$.)
%(It is possible that $\Gamma$ has closed components with no contracted boundary tails, whose corresponding moduli space consists of closed $r$-spin structures with no involution; in this case, the notation $(R^0\pi_*\mathcal{J'})_+$ should be interpreted simply as $R^0\pi_*\mathcal{J'}$.)

\begin{definition}
\label{def:coherent}
Let $\Gamma$ be a connected graded $r$-spin dual graph, and let $s$ be a section of $\mathcal{W}$ over a subset $U\subset\TRAM_{\Gamma}$. $s$ is {\it coherent} if either the anchor of $\Gamma$ is not twisted $-1$, or, for any point $\zeta=(C, u_t) \in U,$  corresponding to a graded $r$-spin disk $C$ and an element $u_t \in J'_{z_t},$ the element $s(\zeta) \in H^0(J')$ satisfies
$\ev_{z_t}s(\zeta) = u_t$. A coherent multisection $s$ is defined as a multisection each of whose local branches is coherent. For a disconnected $\Gamma,$ a (multi)section $s$ of $\cW$ over $U \subset \TRAM_{\Gamma}$ is coherent if it can be written as the restriction to $U$ of
$\underset{\Lambda \in \Conn(\Gamma)}{\boxplus}s_{\Lambda}$, where $\text{Conn}(\Gamma)$ is the set of connected components of $\Gamma$ and each~$s_\Lambda$ is a coherent (multi)section of $\mathcal{W}_\Lambda \rightarrow U_\Lambda$ for some $U_\Lambda\subseteq\TRAM_{\Lambda}.$
\end{definition}

Note that, if $s$ is a coherent multisection of $\mathcal{W} \rightarrow \TRAM_{\Gamma}$ and $\zeta\in\M_{\Gamma}^{1/r} \hookrightarrow \TRAM_{\Gamma}$, the evaluation $\ev_{z_t}(s(\zeta))$ is equal to zero whenever $z_t$ is an anchor with twist~$-1$, and thus $s(\zeta)$ is induced by a multisection of $J$.  In other words, the restriction of a coherent multisection $s$ to $\M_{\Gamma}^{1/r}$ is canonically identified with a multisection of $\cW \rightarrow \M_{\Gamma}^{1/r}$; we denote this induced section by $\overline{s}$ in what follows.  In case $\Gamma$ has no anchor of twist $-1$, we have $\overline{s} = s$.

Note, also, that adding a coherent multisection of $\cW \rightarrow \TRAM_{\Gamma}$ to a multisection of $\TAU^*(\cW\to\oCM_\Gamma^{1/r})$ yields another coherent multisection.

\subsubsection{An example of coherent sections}
\label{ex:simple_coherent}

Let $\Gamma$ be a connected closed graded $r$-spin dual graph such that the anchor, denoted $t$, has twist $-1$.  Then the fiber of $J'$ at the anchor can be identified with the trivial line $\C,$ canonically up to a global choice of a root of unity.  Let $\oCM_\Gamma^{1/r,\text{rigid}}$ be the degree-$r$ cover of $\oCM_\Gamma^{1/r}$ classifying pairs $(C,\rho)$, where $\rho:J'_{z_t}\xrightarrow{\sim}\C$ is one of the $r$ identifications. Denote by $\cW' \rightarrow \oCM_\Gamma^{1/r,\text{rigid}}$ the pullback of the Witten bundle. Similarly, define $\TRAM_\Gamma^{\text{rigid}}$ as the degree-$r$ cover of $\TRAM_\Gamma$ defined by adding an identification of the fiber at the anchor with $\C,$ and $\cW\to\TRAM_\Gamma^{\text{rigid}}$ the pullback of $\cW\to\TRAM_\Gamma$ to this cover. We again denote by $\TAU$ the projection $\TRAM_\Gamma^{\text{rigid}}\to\oCM_\Gamma^{1/r,\text{rigid}}.$

For $\Lambda\in\partial^!\Gamma$, let $E_*\subseteq E(\Lambda)$ be the minimal collection of Neveu--Schwarz edges such that, after detaching $E_*$, the connected component of the anchor has no Neveu--Schwarz edges. Write
$\text{detach}_{E_*}\Lambda=\Lambda_0\sqcup\Lambda_*$,
where $\Lambda_0$ is the connected component of the anchor and $\Lambda_*$ is the remaining graph.   The bundle $\cW'\to\oCM_{\Lambda}^{1/r,\text{rigid}}$ naturally decomposes as a direct sum $\cW'_{\Lambda_0}\boxplus\cW_{\Lambda_*}$ (up to the isotropy groups actions) by Item \eqref{it:NS} of Proposition \ref{pr:decomposition}.   Moreover, using Remark \ref{rmk:dec_stronger}, a section $s^{\Lambda_0}$ of $\cW'_{\Lambda_0}$ gives rise to a section $s^{\Lambda_0}\boxplus 0$ of $\cW'_\Lambda$.

Assume, that $\Gamma$ is smooth, and let $A_0(\Gamma)=\{\Lambda_0 \;| \;\Lambda\in\partial^!\Gamma\}$.  $\Gamma\in A_0(\Gamma)$ and, $A_0(\Xi)\subseteq A_0(\Gamma)$ for any $\Xi\in A_0(\Gamma)$.  We construct a family of sections
\[(s^\Xi\in\Gamma(\oCM_{\Xi}^{1/r,\text{rigid}},\cW'_{\Xi}))_{\Xi\in A_0(\Gamma)}\]
that evaluate to $1$ at the anchor of any moduli point, and that satisfy
\begin{equation}
\label{eq:comp_for_coherent}s^\Xi|_{\oCM_{\Lambda}^{1/r,\text{rigid}}} = s^{\Lambda_0}\boxplus 0\end{equation}
for any $\Lambda \in \d\Xi$ with a Neveu--Schwarz edge.

The construction is by induction on $\dim(\oCM_\Xi)$.  When $\dim(\oCM_\Xi)<0$, there is nothing to prove. Suppose we have constructed the requisite family for all~$\Xi'$ with $\dim(\oCM_{\Xi'}) < n$, and let $\Xi\in A_0(\Gamma)$ be such that $\dim(\oCM_\Xi)=n.$ Write
\[N=\bigsqcup_{\substack{\Lambda\in\partial\Xi,\\
\Lambda~\text{has a Neveu-Schwarz edge}}}\oCM_{\Lambda}^{1/r,\text{rigid}}.\] We define $s^\Xi|_N$ by \eqref{eq:comp_for_coherent}. This definition makes sense, since we have
\[(s^\Xi|_{\oCM_{\Lambda}^{1/r,\text{rigid}}})|_{\oCM_{\Lambda'_0}^{1/r,\text{rigid}}} = (s^{\Lambda_0}\boxplus 0)|_{\oCM_{\Lambda'_0}^{1/r,\text{rigid}}} = s^{\Lambda'_0}\boxplus 0=s^\Xi|_{\oCM_{\Lambda'}^{1/r,\text{rigid}}},\] whenever $\Lambda\in\partial\Xi,~\Lambda'\in\partial\Lambda$. For $(C,\rho)\in\oCM_{\Xi}^{1/r,\text{rigid}}\setminus N$, there exists an element $u_C\in\cW'_{(C,\rho)}$ with $\rho(\ev_{z_t}(u_{(C,\rho)}))=1$ by Lemma \ref{lem:surj_of_eval}; extend $u_{C}$ to a smooth section $s'_{(C,\rho)}\in\Gamma(U,\cW')$ in a neighborhood $U$ of $(C,\rho)$ not intersecting $N$.  For $(C,\rho)\in N$, define $s'_{(C,\rho)}$ as an extension of the already-defined $s^\Xi_{(C,\rho)}|_{N}$ to a small neighborhood $U$. In both cases, we can shrink $U$ so that $v(C',\rho):=\rho(\ev_{z_t}(s'_{(C,\rho)}(C',\rho)))\neq 0$ for $(C',\rho)\in U.$  Define $s_{(C,\rho)}=s'_{(C,\rho)}/v \in\Gamma(U,\cW')$. $s_{(C,\rho)}$ evaluates to $1$ at the anchor and extends $s^\Xi|_{U\cap N}.$

By compactness, we can cover $\oCM_\Xi^{1/r,\text{rigid}}$ by finitely many open sets $U_i$ with sections $s_i\in\Gamma(U_i,\cW')$ as above. Using a partition of unity $\{\rho_i\}$ subordinate to the cover $\{U_i\}$, we can construct $s^\Xi=\sum \rho_is_i\in\Gamma(\oCM_\Xi^{1/r,\text{rigid}},\cW')$ that evaluates to $1$ at anchors and satisfies \eqref{eq:comp_for_coherent}. The induction follows.

We now have a section $s^\Gamma$ that evaluates to $1$ at the anchor.
Define $u^{\text{rigid}}\in\Gamma(\TRAM_\Gamma^{\text{rigid}},\cW)$, for $(\lambda,C,\rho) \in \TRAM_\Gamma$ with $(C,\rho)\in\oCM_\Gamma^{1/r,\text{rigid}}$ and $\lambda\in(\mathcal{J}'_{C})_{z_t}$, by
\[u^{\text{rigid}}(\lambda,C,\rho)=\rho(\lambda)\TAU^*s^\Gamma(C,\rho).\]
Finally, define $u\in\Gamma(\TRAM_\Gamma,\cW)$ by the averaging
\[u(\lambda,C,\rho) = \frac{1}{r}\sum_{\rho:(J'_{\Sigma})_{z_t}\simeq\C}u^{\text{rigid}}(\lambda,C,\rho),~ C\in\oCM_\Gamma^{1/r},~\lambda\in(\mathcal{J}'_{C})_{z_t},\]
where we use the canonical identification between the fibers $\cW_{(\lambda,C)}$ and $\cW_{(\lambda,C,\rho)}$.  The section $u$ is coherent and vanishes precisely on $\oCM_\Gamma^{1/r}\hookrightarrow\TRAM_\Gamma.$

\subsubsection{The assembling operation $\Ass$}
Given a graded $r$-spin dual graph $\Gamma$ and a subset $E' \subset E^I(\Gamma)$, there is an ``assembling" procedure $\Ass$ by which a coherent multisection $s'$ of $\cW \rightarrow \TRAM_{\detach_{E'}(\Gamma)}$ is glued to give a coherent multisection $\Ass_{\Gamma,E'}(s')$ of $\cW \rightarrow \TRAM_{\Gamma}$; this is important in the proof of the TRRs and in the construction of boundary conditions.  We explain the procedure in the case where $\Gamma$ is connected, $E' = \{e\}$, and $s'$ is a coherent section, for simplicity of notation. The generalization to a set of edges is automatic.

Let $\widehat\Gamma:= \detach_{E'}(\Gamma)$, and let
$\text{Detach}_{E'}: \M_{\Gamma}^{1/r} \rightarrow \M_{\widehat\Gamma}^{1/r}$ be the map $q\circ \mu^{-1}$, in the notation of \eqref{eq:Wittendecompsequence}.  This map has degree one, and if $e$ is Neveu--Schwarz, there is an analogous map
$\text{Detach}_{E'}: \TRAM_{\Gamma}\rightarrow \TRAM_{\widehat\Gamma}$,
also of degree one.  Proposition \ref{pr:decomposition} extends to the Witten bundles over $\TRAM_\Gamma$ and $\TRAM_{\widehat\Gamma}$; in particular, item \eqref{it:NS} of this proposition shows that, up to the automorphism group actions, $\text{Detach}_{E'}^*\widehat\cW$ and $\cW$ are naturally isomorphic.  Moreover, Remark \ref{rmk:dec_stronger} extends as well and allows us to identify multisections of $\text{Detach}_{E'}^*\widehat\cW$ as multisection of $\cW$.  With this identification, we set
\[\Ass_{\Gamma,E'}(s') = \text{Detach}_{E'}^*(s')=\text{Detach}_{E'}^*(s'_1\boxplus s'_2),\]
where $s'_1,s'_2$ are the coherent multisections for the connected components of $\widehat\Gamma$.

If $e$ is Ramond, then  the fact that $\Gamma$ is anchored implies that exactly one of its half-edges has twist $-1$ after detaching.  Denote this half-edge by $h_2$, and let $\Gamma_2$ be its connected component in $\widehat\Gamma$; let $h_1$ be the other half-edge of $e'$, with connected component $\Gamma_1$.
Choose an identification $\rho:J'_{z_{h_1}}\to J'_{z_{h_2}},$ between the fibers of $J'$ at the half nodes $z_{h_i}$; there are $r$ possible identifications, which differ by $r$th roots of unity from one another.  By definition, a point in $\TRAM_{\widehat\Gamma}$ consists of $(\xi_1, \xi_2, \nu)$, where $\xi_1 \in \TRAM_{\Gamma_1}$, $\xi_2 \in \M_{\Gamma_2}^{1/r}$, and $\nu \in J'_{z_{h_2}}$.  A coherent multisection $s'$ associates to each such tuple an element
$s'(\xi_1, \xi_2, \nu) = s_1' \boxtimes s_2' \in H^0(J'),$ such that  $\ev_{z_{h_2}}s_2' = \nu,$ and $s'_1~\text{is coherent}.$
  We write $s_1' = s_1'(\xi_1)$ and $s_2' = s_2'(\xi_2, \nu)$, and consider\footnote{When, as in \eqref{eq:assemble}, we substitute a multisection $t_1$ into another multisection $t_2$, the result is the multisection whose local branches are $t^i_2(t^j_1)$, where $(t^i_2,t^j_1)$ are all possible pairs of local branches.  The weight (see Definition \ref{def:multisection}) of the local branch $t^i_2(t^j_1)$ is $\mu_2^i\mu_1^j$, where $\mu^*_*$ denotes the weight of the local branch $t_*^*.$}
\begin{equation}
\label{eq:assemble}
s'_1(\xi_1)\boxtimes s'_2(\xi_2,\rho(\ev_{z_1}(s'_1(\xi_1)))).
\end{equation}
This is a multisection of a bundle $J_1' \boxtimes J_2'$ on a disconnected curve together with an identification $\rho$ of the fibers of $J_1'$ and $J_2'$ over the half-nodes $z_{h_i}$, so it gives rise to a function on the fiber of $(q')^*\widehat{\cW}$.  Moreover, the multisection~\eqref{eq:assemble} is constructed so that its values on the two-half nodes agree under $\rho$, or in the language of Proposition~\ref{pr:decomposition}, its image in $\mathcal{T}$ vanishes.  It therefore induces a function on the fiber of $\mu^*i_{\Gamma}^*\cW$, which is a function on $H^0(J')$ on the connected curve; the induced multisection of $\cW\to\TRAM_\Gamma$ is $\Ass_{\Gamma,E}(s)(\zeta_1,\zeta_2)$. It is easy to see that $\Ass_{\Gamma,E}(s)$ is coherent and independent of the choice of $\rho.$
%It is straightforward to see that $\Ass_{\Gamma,E}(s)$ is coherent and independent of the choice of $\rho$. %We refer to $\Ass$ as the {\it assembling operator}.

\begin{obs}\label{obs:comp_of_Ass}
Let $\Gamma_1$ be a graded $r$-spin dual graph. Suppose that $\Gamma_2$ is obtained from $\Gamma_1$ by detaching edges $E_1$, and $\Gamma_3$ is obtained from $\Gamma_2$ by detaching edges $E_2$.  If $s$ is a coherent multisection of $\cW \rightarrow \TRAM_{\Gamma_3}$, then
\[\Ass_{\Gamma_1,E_1}(\Ass_{\Gamma_2,E_2}(s)) = \Ass_{\Gamma_1,E_1\cup E_2}(s).\]
\end{obs}

\subsubsection{Special canonical multisections and transversality}
\begin{definition}\label{def:abs_vertex}
Let $\mathcal{V}(\Gamma)$ be the collection of one-vertex dual graphs that may appear as connected components of
$\text{for}_{\text{marking}}( \text{for}_{\text{illegal}}(\detach_{E(\Lambda)}(\Lambda)))$ for some $\Lambda\in\partial^!\Gamma\setminus\partial^+\Gamma,$
where $\text{for}_{\text{marking}}$ is the map that changes all boundary markings to zero.  We call the elements of $\mathcal{V}(\Gamma)$ \emph{abstract vertices}.
For $i \neq 0$, let $\mathcal{V}^i(\Gamma)\subset\mathcal{V}(\Gamma)$ be the collection of abstract vertices that have an internal tail labeled $i$.  For an internal marking $i\neq 0$ and a graph $\Gamma$ with $i\in I(\Gamma),$
denote by $v_i^*(\Gamma)$ the (single-vertex) connected component of
$\text{for}_{\text{marking}}( \text{for}_{\text{illegal}}(\detach_{E(\Gamma)}(\Gamma))) $
that contains $i$ as an internal marking.

Denote by $\Phi_{\Gamma,v}:\oCM_{\Gamma}^{1/r}\to\oCM_{v}^{1/r}$ the natural map.  Namely, if $\text{For}_{\text{marking}}$ is the map (on the moduli level) that changes all boundary markings to zero and $\text{For}_{\text{illegal}}$ is the map that forgets all illegal, twist-zero boundary marked points with marking zero, then $\Phi_{\Gamma,v}$ is the composition of
$\text{For}_{\text{illegal}}\circ \text{For}_{\text{marking}}\circ \text{\Detach}_{E(\Gamma)}$
with the projection to $\oCM_{v}^{1/r}$.  Write $\Phi_{\Gamma,i}$ for $\Phi_{\Gamma,v_i^*(\Gamma)}.$
\end{definition}
Analogously to Observation \ref{obs:Li_pulled_pre} we have the following:
\begin{obs}\label{obs:CL_i pulled}
The map $\Phi_{\Gamma,i}|_{\oCM_{\Gamma}^{1/r}}$ factors through $F_{\Gamma}.$
If $v_i^*(\CB\Gamma)$ is stable, then $\CL_i\to\oCM_\Gamma^{1/r}$ is canonically identified with $\Phi_{\Gamma,i}^*(\CL_i\to \oCM_{v^*_i(\Gamma)}^{1/r})$.
\end{obs}
\begin{definition}\label{def:special_canonical_for_L_i}
A canonical multisection $s\in C_m^\infty(\oCM^{1/r}_\Gamma,\CL_i)$ is \emph{special canonical} if for every stable $v\in \mathcal{V}^i(\Gamma),$ there exists a canonical multisection $s^v$ of $\CL_i\to\oCM_{v}$ such that for every $\Lambda\in\d\Gamma$, we have
\begin{equation}\label{eq:special_canonical_L_i}
s|_{\CM_{\Lambda}^{1/r}}=\Phi_{\Lambda,i}^*s^{v_i^*(\Lambda)}.
\end{equation}
%Note that this equality uses Observation \ref{obs:CL_i pulled}.  
Write $\CS_i=\CS_i(\Gamma)$ for the vector space of special canonical multisections of $\CL_i.$
\end{definition}

\begin{definition}\label{def:special canonical for Witten}
A canonical multisection $s\in C_m^\infty(\oPM_\Gamma,\cW)$ is \emph{special canonical} if the for any abstract vertex $v\in \mathcal{V}(\Gamma)$, there is a coherent multisection $s^v\in C_m^\infty(\TRAM_v,\cW),$ and those multisections satisfy
\begin{enumerate}
\item if $v$ is closed and the anchor $t$ has twist $r-1$ then $\ev_t(s^v)$ is always positive with respect to the orientation of Definition \ref{def:stable_graded_rspin_disk}, Item (4);
\item for any $\Lambda\in \partial^!\Gamma\setminus\partial^+\Gamma$,
\begin{equation}\label{eq:special_canonical_Witten}s|_{\CM_\Lambda^{1/r}} =F_\Lambda^*(\Ass_{\CB\Lambda,E(\CB\Lambda)}((s^v)_{v\in \Conn(\detach(\CB\Lambda))}))\end{equation}
where $\Conn(-)$ denotes the set of connected components of a graph and $\text{detach}(-)$ denotes the detaching of a graph along all of its edges.
\end{enumerate}
\end{definition}

We consider $s^v$ for $v\in\mathcal{V}(\Gamma)$ or $v\in\mathcal{V}^i(\Gamma)$ as part of the information of the special canonical multisection.

Canonical multisections are positive multisections that are pulled back from their base moduli; special canonical multisections satisfy the more restrictive property that their $\cW_v$ component, for $v\in\mathcal{V}(\Gamma)$, is pulled back from the corresponding moduli $\CM_v^{1/r}$.  Because of cases (when $v$ is closed and the anchor is Ramond) in which $\cW$ does not decompose as a direct sum, we use coherent multisections and the map $\Ass$ instead of usual multisections and direct sums.

Since $s,s^v$ are global multisections, continuity considerations give:

\begin{obs}\label{obs:s_v_canonical}
Equations \eqref{eq:special_canonical_L_i},and \eqref{eq:special_canonical_Witten} hold also for $s|_{\oPM_\Lambda}$.%if we replace $s|_{\CM_\Lambda^{1/r}}$ by $s|_{\oPM_\Lambda}$.
% A special canonical multisection for Witten's bundle or for a tautological line $\CL_i$ is canonical.
\end{obs}

\begin{rmk}
The requirement in the definitions of special canonical multisections that $s^v$ be canonical (and in the case of $\mathbb{L}_i$, the requirement that $s$ be canonical) is actually redundant. In fact, each $s^v$ for open $v$ is even special canonical. We do not need this fact, but we sketch the idea. 

Let $s$ be a special canonical multisection of the Witten bundle. If $v$ is an abstract vertex appearing in $\detach_{E(\Lambda)}(\Lambda)$, if $\Xi\in\partial v$, and if $\Lambda'\in\partial \Lambda$ is a the graph obtained from $\Lambda$ by replacing $v$ in $\Xi$ in the obvious way, then from \eqref{eq:special_canonical_Witten} applied to $\Lambda$ and extended to the boundary, we have
\[s|_{\CM_{\Lambda'}^{1/r}} =F_\Lambda^*(\Ass_{\CB\Lambda,E(\CB\Lambda)}((s^u|_{\CM_u^{1/r}})_{u\in \Conn(\detach(\CB\Lambda))\setminus\{v\}},s^v|_{\CM_\Xi^{1/r}})).\]
On the other hand, by applying directly to $\Lambda'$, we get
\[s|_{\CM_{\Lambda'}^{1/r}} =F_{\Lambda'}^*(\Ass_{\CB\Lambda',E(\CB\Lambda')}((s^u|_{\CM_u^{1/r}})_{u\in \Conn(\detach(\CB\Lambda'))})).\] Comparing the two equations using Observations \ref{obs:for_comp} and \ref{obs:comp_of_Ass} and the definition of $\Ass_*$, we see
\[s^v|_{\CM_\Xi^{1/r}} =F_\Lambda^*(\Ass_{\CB\Xi,E(\CB\Xi)}((s^u|_{\CM_u^{1/r}})_{u\in \Conn(\detach(\CB\Xi))})).\]
The positivity for $s^v$ is easily deduced from the positivity of $s$.  A similar argument works for special canonical multisections of $\CL_i.$
\end{rmk}

The following lemma, proven in Section \ref{subsec:constructions_bc_cons}, constructs special canonical multisections with strong transversality properties, used in the proof of TRRs.

\begin{lemma}\label{lem:existence}
Choose $k,l \geq 0$, and for each $i \in \{1,\ldots, l\}$, let $a_i \in \{1,\ldots, r-1\}$ be such that $\oPMr\neq \emptyset$.  Let $\Gammar$ denote the graded graph with a single vertex $v$, boundary tails marked by $\{1,\ldots, k\}$ and internal tails marked $\{1,\ldots, l\}$, such that the $i$th internal tail has twist $a_i$.  Then, for any  $d_1, \ldots, d_l \geq 0$, there exist global special canonical multisections
\[s\in C_m^\infty(\oPMr,\cW),~s_{ij}\in\CS_i,\; 1 \leq i \leq l, \; 1 \leq j \leq d_i,\] with the following transversality property:  For every vertex $v\in \mathcal{V}(\Gammar)$ and every
$
K\subseteq \bigcup_{i\in I(v)}\left\{i\right\}\times \{1, \ldots, d_i\},
$
one has
\[
\hat{s}^{v}\oplus\bigoplus_{ab\in K} s_{ab}^{v}\pitchfork 0,
\]
where $\hat{s}^v = s^v$ if $v$ is open and otherwise $\hat{s}^v = \bar{s}^v$ (which also equals $s^v$ if the anchor of $v$ is not twisted $-1$).
In particular, if
\begin{equation}\label{eq:open_rank_with_psi}
\frac{2\sum_{i=1}^l a_i+(k-1)(r-2)}{r}+2\sum_{i=1}^l d_i = 2l+k-3,\end{equation} then $s,s_{ij}$ can be chosen so that $\mathbf{s} = s\oplus\bigoplus_{i,j} s_{ij}$ vanishes nowhere on $\partial\oPMr.$
\end{lemma}

\subsection{Proof of Theorem \ref{thm:TRR}}\label{subsec:pf_trr}
The basic idea of the proof of Theorem~\ref{thm:TRR} is similar to the proof of Theorem 1.5 in \cite{PST14}: We first construct a global section $t$ of $\CL_1$ and show that its zero locus consists of internal strata $\PM_\Gamma$ with a single boundary node. We then show that the number of zeroes of the intersection problem we consider, when a canonical section $s_{11}$ of $\CL_1$ is replaced by $t,$ is the sum of products of open and closed contributions. We then use homotopy arguments to compare this zero count with the correlator. The homotopies contribute to the zero count; the contributions correspond to real codimension-$1$ boundaries of the moduli space. As we shall see, boundaries on which we have positivity constraints will not contribute. The contribution of the remaining boundaries will be products of two open intersection numbers for ``smaller" problems.
There are two main conceptual differences between our treatment and that of \cite{PST14}: the usage of positivity and the contribution from internal Ramond nodes.  (The latter is a surprising feature of the theory, and it is also the place where the closed extended $r$-spin theory enters the picture.)  There are also technical differences, related to orientations, multiplicities of zeroes, and the fact that we work on $\oPMr$ rather than on $\oCMr$.

We require the following two lemmas concerning homotopies. The first is proven in Section \ref{subsec:homotopies} below, the second is the noncompact orbifold analogue of \cite[Lemma 3.55]{PST14} and the proof is analogous as well.

\begin{lemma}\label{lem:partial_homotopy}
Let $E_1$ and $E_2$ be bundles on $\oCMr$ of the form
\[
E_j = \cW^{\varepsilon^j}\oplus\bigoplus_{i\in \left[l\right]}\CL_i^{\oplus d^j_i}\]
where $d_i^j \geq 0$ and $\varepsilon^j \in \{0,1\}$ satisfy $\varepsilon^1 + \varepsilon^2 = 1$. % Let $d_i = d_i^1 + d_i^2$ for each $i$, and define
Write $E = E_1\oplus E_2,$ and %\[E = E_1\oplus E_2=\cW\oplus\bigoplus_{i\in \left[l\right]}\CL_i^{\oplus d_i}.\]
assume that $\rk_{\R}(E) = k+2l-3$. Let $A\subseteq\partial^0\Gammar$ and set
$
C = \coprod_{\Gamma\in A} \CM_{\Gamma}^{1/r} \subseteq\partial^0\oCMr.
$
Let $\mathbf{s},~\mathbf{r}$ be multisections of $E|_{\partial^0\oPM_\Gamma\cup V_+}$, where
\begin{equation}\label{eq:V_+}V_+=\left(W\cap\oPMr\right)\cup \bigcup_{\Lambda~\text{\stronglypositive}}{\oPM_\Lambda}
\end{equation}
for some neighborhood $W$ of $\partial^+\oCMr$.  Assume that $\mathbf{s}$ and $\mathbf{r}$ satisfy
\begin{enumerate}
\item\label{it:can}
$\mathbf{s}|_{C\cup V_+}$ and $\mathbf{r}|_{C\cup V_+}$ are canonical, and
\item\label{it:id_and_trav}
the projections of $\mathbf{s}$ and $\mathbf{r}$ to $E_1$ are identical and transverse to $0$.
\end{enumerate}
Then there is a homotopy $H$ between $\mathbf{s}$ and $\mathbf{r}$, of the form
\begin{equation}\label{eq:form_of_hom}
H(p,t) = (1-t) \mathbf{s}(p) + t \mathbf{r}(p) + t(1-t) w(p),
\end{equation}
where $w(p)\in C_m^\infty(\partial^0\oPM_\Gamma\cup V_+,E)$ is canonical, such that
$H|_{\CM_\Lambda^{1/r}\times[0,1]}$ is transverse to $0$ for any stratum $\CM_\Lambda^{1/r},$ has constant-in-$t$ projection to $E_1$, and vanishes nowhere on $\left(C\cup V'_+\right)\times\left[0,1\right],$ where $V'_+\subseteq V_+$ is of the form \eqref{eq:V_+}.
\end{lemma}
%
%This lemma is proven in Section %\ref{subsec:homotopies} below.
%
\begin{lemma}\label{lem:zero diff as homotopy}
Let $E \to M$ be an orbifold vector bundle over an orbifold with corners, with $\text{rank}(E) = \dim(M)$, and let $s_0$ and $s_1$ be nowhere-vanishing smooth multisections. Denote by $p : [0,1] \times M\to M$ the projection, and, for any open set $U \subseteq M$ with $M \setminus U$ compact, consider the homotopy
\[
H \in C_m^\infty((U\cup \partial M)\times[0,1],p^*E),~~H(x,i) = s_i(x),~i=1,2.
\]
%be such that $H(x,i) = s_i(x)$ for $i=1,2$.  
Suppose that $H \pitchfork 0$ and that $H(u,t) \neq 0$ for all $t \in [0,1]$ and $u \in U$.  Then
\[
\int_M e(E ; s_1) - \int_M e(E ; s_0) = \# Z(H)=\# Z(H|_{\partial M\times[0,1]}).
\]
\end{lemma}
%This lemma is the noncompact orbifold analogue of \cite[Lemma 3.55]{PST14}, and the proof is analogous as well.

\begin{proof}[Proof of Theorem \ref{thm:TRR}]
Define a section $\tildet \in C^\infty(\oPMr,\CL_1)$ as follows.  For $C\in\CMr$, identify the preferred half $\Sigma$ with the upper half-plane and set
\begin{equation}\label{eq:t}
\tildet\left(C\right) = dz\left.\left(\frac{1}{z - x_1}-\frac{1}{z-\bar{z}_1}\right)\right|_{z = z_1} \in T_{z_1}^*\Sigma = T_{z_1}^*C.
\end{equation}
This section, which is pulled back from the zero-dimensional moduli space $\oCM_{0,1,1}$, extends to a smooth global section $\tildet$ over $\oPMr.$  For a graded $r$-spin disk $C$ that is not necessarily smooth, set $\tildet(C)$ to be the evaluation at $z_1\in C$ of the unique meromorphic differential $\varphi_C$ on $C$ with simple poles at $\bar{z}_1$ and $x_1$ and at no other marked points or smooth points, such that the residue at $x_1$ is $1,$ at $\bar{z}_1$ is $-1$, and the residues at every pair of half-nodes sum to $0$.

Let ${U_{oc}} \subset \partial \Gammar$ be the collection of graphs $\Gamma$ consisting of an open vertex $v_{\Gamma}^o$ and a closed vertex $v_{\Gamma}^c$ joined by a unique edge $e$, where $1 \in I(v_{\Gamma}^c)$.  In what follows, we use the same notation $v^o_\Gamma$ (respectively, $v^c_\Gamma$) both for the vertex and for the corresponding graph that is a connected component of $v^o_{\Gamma}$ (respectively, $v^c_\Gamma$) in $\detach(\Gamma)$.  For $\Gamma \in {U_{oc}},$ let
$\text{Detach}_e:\oPM_\Gamma \to \oPM_{v^o_\Gamma} \times \oCM_{v^c_\Gamma}^{1/r}$
be the restriction to $\oPM_\Gamma$ of the composition $q \circ \mu^{-1}$, where $q$ and $\mu$ are again as in \eqref{eq:Wittendecompsequence}.  As above, we note that $\text{Detach}_e$ is generically one-to-one.

The $\oCM_{v^c}^{1/r}$ and its Witten bundle carry a canonical complex orientation.  Write $\tilde{\mathfrak{o}}_{v^c}$, $\mathfrak{o}_{v^c}$, and ${o}_{v^c}$ for the complex orientation of the Witten bundle on $\oCM_{v^c}^{1/r}$, the complex orientation of $\oCM_{v^c}^{1/r}$, and the complex relative orientation described in \cite[Theorem 5.2]{BCT1}, respectively.  To briefly describe the latter, let $\pi$ be an arbitrary order of the boundary markings, and let $\hat\pi$ be the induced cyclic order, which indexes a connected component $\oCM_{v^o}^{1/r,\hat\pi}$ of $\oCM_{v^o}^{1/r}$.  In \cite[Notation 3.14]{BCT1}, an orientation $\tilde{\mathfrak{o}}^\pi_{v^o_\Gamma}$ of $\oCM_{v^o}^{1/r,\hat\pi}$ is given, while in \cite[Definition 5.15]{BCT1}, a corresponding orientation $\mathfrak{o}^\pi_{v^o_\Gamma}$ of the Witten bundle over $\oCM_{v^o}^{1/r,\hat\pi}$ is defined. The canonical relative orientation ${o}_{v^o_\Gamma}$ is defined using them by \cite[Equation (5.8)]{BCT1}, and it is independent of the choice of $\pi$. % See \cite{BCT1} for the details.
Write $\mathfrak{o}^\pi_\Gamma=\text{Detach}_e^*(\mathfrak{o}_{v^c_\Gamma}\boxtimes \mathfrak{o}^\pi_{v^o_\Gamma})$ and $\tilde{\mathfrak{o}}_\Gamma^\pi=\text{Detach}_e^*(\tilde{\mathfrak{o}}_{v^c_\Gamma}\boxtimes \tilde{\mathfrak{o}}_{v^o_\Gamma}^\pi)$.  The first key lemma is the following.
\begin{lemma}\label{lem:zero_locus_tr_section}
The zero locus of $\tildet$ is $\bigcup_{\Gamma \in {U_{oc}}} \oPM_{\Gamma}.$ The multiplicity of $\oPM_\Gamma$ in the zero locus is $r$. Moreover, if we orient by $\tilde{\mathfrak{o}}^\pi_{0,k,\vec{a}}$ the connected component $\overline{\mathcal{PM}}^{\frac{1}{r}, \hat\pi}_{0,k,\vec{a}}$ of $\oPMr$ in which the boundary marked points have cyclic order $\hat\pi$, then $\tildet$ induces on its zero locus the orientation $\tilde{\mathfrak{o}}_\Gamma^\pi.$  Hence, if $\cW\to \oCM^{1/r}_{0,k\{a_i\}_{i\in [l]}}$ is canonically relatively oriented, then any $\cW\to\oPM_\Gamma$ appears in the zero locus %inside the total space 
with the relative orientation ${o}_\Gamma=\text{Detach}_e^*({o}_{v^c_\Gamma}\boxtimes {o}_{v^o_\Gamma}).$

\end{lemma}
\begin{proof}
The orientation of $\oCMr$ is defined as the pullback of the orientation $\tilde{\mathfrak{o}}^\pi_{0,k,l}$ of $\oCM_{0,k,l}^{\hat\pi}$ by $\text{For}_{\text{spin}}$. The section $\tildet$ is also pulled back from the same moduli space, and we use the same notation also for the section over $\oCM_{0,k,l}.$

In order to understand the vanishing locus setwise, both for the spin moduli and for the moduli without spin, note that, by definition, on a component of the closed curve $C$ containing $x_1$ or $\bar{z}_1,$ the differential $\varphi_C$ is nowhere-vanishing.  Similarly, $\varphi_C$ vanishes nowhere on components whose removal disconnects $x_1$ from $\bar{z}_1$. On other components, it vanishes identically. Thus, $\tildet$ vanishes exactly on stable disks $C$ such that the component containing $z_1$ is not on the path of components between the components of $\bar{z}_1$ and $x_1$. This happens exactly when $z_1$ belongs to a sphere component of $\Sigma$.  Thus, $\bigcup_{\Gamma \in {U_{oc}}} \oCM_{\Gamma}$ (or its spin analogue) is indeed the vanishing locus of $\tildet.$

Let $\text{for}_{\text{spin}}(\Gamma)$ be the stable graph obtained from $\Gamma\in U_{oc}$ by forgetting the extra structure, and let $\tilde{\mathfrak{o}}^{\hat\pi}_{\text{for}_{\text{spin}}(\Gamma)}$ be the orientation on
$\oCM^{\hat\pi}_{\text{for}_{\text{spin}}(\Gamma)} \cong \oCM_{\text{for}_{\text{spin}}(v^c_\Gamma)}\times \oCM^{\hat\pi}_{\text{for}_{\text{spin}}(v^o_\Gamma)},$
given as the product of the complex orientation with the canonical orientation of \cite{BCT1}.  We have $\tilde{\mathfrak{o}}^{\hat\pi}_{\Gamma}=\text{For}_{\text{spin}}^*\tilde{\mathfrak{o}}^{\hat\pi}_{\text{for}_{\text{spin}}(\Gamma)}$. The vanishing of the section $\tildet$ on $\oCM_{0,k,l}^{\hat\pi}$ at $\oCM_{\text{for}_{\text{spin}}(\Gamma)}^{\hat\pi}$ is transversal, and the induced orientation as (a part of) the zero locus is $\tilde{\mathfrak{o}}^{\hat\pi}_{\text{for}_{\text{spin}}(\Gamma)}$.  Indeed, the transversality is proven similarly to Lemma 3.43 of \cite{PST14}, while the orientation claim is easily verified in the two-dimensional case, and the general case follows by an inductive argument, identical to the one given in the proof of \cite[Lemma 3.15]{BCT1}.

Returning to the spin case, there is additional $r$-fold isotropy present in $\CM_{\Gamma}^{1/r}$ that was not present in the non-spin case, so the pulled-back section has order of vanishing $r$ in the orbifold sense.

As $\cW\to \oCM_{0,k, \vec{a}}^{1/r}$ is canonically relatively oriented, and we orient $\oCM_{0,k,  \vec{a}}^{\hat\pi}$ by $\tilde{\mathfrak{o}}_{0,k,\vec{a}}^\pi$, the definition of the canonical relative orientation implies that $\cW$ is oriented by $\mathfrak{o}_{0,k, \vec{a}}^\pi$.  Thus, by the definition of $\mathfrak{o}^\pi_{v^o_\Gamma}$,
\[\mathfrak{o}_{0,k, \vec{a}}^\pi\bigg|_{{\oCM_\Gamma^{1/r}}^{\hat\pi}} = \mathfrak{o}_{v^o}^\pi\boxtimes\mathfrak{o}_{v^c}.\]
Combining with the induced orientation of the zero locus, we see that the induced relative orientation is as claimed, and Lemma~\ref{lem:zero_locus_tr_section} is proved.
\end{proof}

Write $n=d_1$, and let
\[E=\cW\oplus \bigoplus_{i=1}^l \CL_i^{\oplus d_i}\rightarrow\oPMr,\quad E_1 =  \cW\oplus\CL_1^{\oplus n-1} \oplus \bigoplus_{i=2}^l \CL_i^{\oplus d_i} \to \oPMr.
\]
Take
$
\mathbf{s} = s\oplus\bigoplus_{\substack{ 1 \leq i \leq l\\ 1 \leq j \leq d_i}} s_{ij}
$,
with $s$ a global special canonical multisection of $\cW$ and $s_{ij}\in \CS_i$, so that $\mathbf{s}$ satisfies the refined transversality of Lemma \ref{lem:existence}. Thus,
by definition (and Theorem \ref{prop:euler_as_zero_locus}), the number of zeroes of $\s$ is
$\#Z(\s)=\left\langle\prod_{i=1}^l\tau_{d_i}^{a_i}\sigma^k\right\rangle^{\frac{1}{r},o}_0$. Put
$
\srest=s\oplus\bigoplus_{\substack{ 1 \leq i \leq l \\ 1 \leq j \leq d_i \\ (i,j) \neq (1,1)}} s_{ij},
$
and set $\ttts = \tildet\oplus\srest$.  Let $U_+=U\cap\partial^0\oCMr$ where $U$ is a neighborhood of $\partial^+\oCMr$ in which $s$ is nowhere vanishing; positivity implies that such a neighborhood exists.

The zero locus $Z\left(\tildet|_{\partial^0\oPMr}\right)$ consists of boundary strata of real codimension at least three in $\oPMr$. The refined transversality of Lemma \ref{lem:existence} guarantees that on such boundary strata, $\srest$ does not vanish. Thus, $\ttts$ does not vanish on $\partial^0 \oPMr$. 
By the refined transversality of $\tsrest$ again, its zero locus is transverse to $\oPM_\Gamma$ for $\Gamma \in U_{oc}$ and does not intersect $\partial\oPM_\Gamma$ for $\Gamma \in U_{oc}.$  Hence, by Lemma~\ref{lem:zero_locus_tr_section} and the definition of weighted cardinality in Notation \ref{nn:weighted_signed}, $\ttts$ has isolated zeroes and one has
\begin{equation}\label{eq:eulert_and_s}
\#Z(\ttts) = \# Z(\tildet \oplus \tsrest) = \sum_{\Gamma \in {U_{oc}}}r\#Z(\tsrest|_{\oPM_\Gamma}).
\end{equation}
Using Theorem \ref{prop:euler_as_zero_locus}, 
$\int_{\oPMr}e\left(E;\ttts|_{\partial^0\oPMr\cup U_+}\right)$
is defined and equals $\#Z(\ttts)$.

\begin{lem}\label{lem:closed_contribution}
\[
\#Z(\ttts) = \sum_{a=-1}^{r-2}\sum_{A \sqcup B = \{2,\ldots,l\}}\left\langle \tau_0^{a}\tau_{n-1}^{a_1}\prod_{i \in A}\tau_{d_i}^{a_i}\right\rangle^{\frac{1}{r},ext}_0
\left\langle \tau_0^{r-2-a}\prod_{i\in B}\tau^{a_i}_{d_i}\sigma^k\right\rangle^{\frac{1}{r},o}_0.
\]
\end{lem}
\begin{proof}
For any $\Gamma \in U_{oc}$, we have a degree-$1$ morphism
\[
 \text{Detach}_e: \oPM_\Gamma \longrightarrow \oPM_{{v_\Gamma^o}} \times \oCM_{{v_\Gamma^c}}^{1/r}.
\]
Let
$
{\text{Proj}}_o : \oPM_{{v_\Gamma^o}} \times \oCM_{{v_\Gamma^c}} \longrightarrow \oPM_{{v_\Gamma^o}}, $ and ${\text{Proj}}_c: \oPM_{{v_\Gamma^o}} \times \oCM_{{v_\Gamma^c}} \longrightarrow \oCM_{{v_\Gamma^c}}
$
be the projection maps.  Using the notation $\TRAM_{v}$ of Section~\ref{subsec:coherent}, we write
\[E^c_\Gamma = \cW^c\oplus\CL_1^{\oplus n-1}\oplus \bigoplus_{i \in I(v^c_\Gamma)\setminus \{1\}} \CL_i^{ \oplus d_i} \longrightarrow \TRAM_{{v_\Gamma^c}}, ~~E^o_\Gamma = \cW^o\oplus\bigoplus_{i \in I(v^o_\Gamma)} \CL_i^{\oplus d_i} \longrightarrow \oPM_{{v_\Gamma^o}},
\]
with $
\cW^c = \cW\to\TRAM_{v_\Gamma^c},~~\cW^o = \cW\to\oPM_{v_\Gamma^o},$
%\begin{align*}
%&\cW^c = \cW\to\TRAM_{v_\Gamma^c},\qquad\cW^o = \cW\to\oPM_{v_\Gamma^o},\\
%&E^c_\Gamma = \cW^c\oplus\CL_1^{\oplus n-1}\oplus \bigoplus_{i \in I(v^c_\Gamma)\setminus \{1\}} \CL_i^{ \oplus d_i} \longrightarrow \TRAM_{{v_\Gamma^c}}, \\
%&E^o_\Gamma = \cW^o\oplus\bigoplus_{i \in I(v^o_\Gamma)} \CL_i^{\oplus d_i} \longrightarrow \oPM_{{v_\Gamma^o}},
%\end{align*}
and where we abuse notation, in the definition of $E^c_\Gamma,$ by letting $\CL_i$ denote the pullback of the tautological line bundle $\CL_i\to\oCM_{v_\Gamma^c}^{1/r}$ to $\TRAM_{{v_\Gamma^c}}$ under the natural map $\TAU:\TRAM_{{v_\Gamma^c}}\to\oCM_{{v_\Gamma^c}}^{1/r}$.  Recall that if $e$ is Neveu--Schwarz, then $\TRAM_{{v_\Gamma^c}}=\oCM_{{v_\Gamma^c}}^{1/r}.$
By the definition of canonical multisections and Observation \ref{obs:s_v_canonical}, the $E_\Gamma^o$-component of $\s$
may be written as
\[({\text{Proj}}_o\circ \text{Detach}_e)^* \s_\Gamma^o,\quad\text{where}\quad\s_\Gamma^o=s^o\oplus\bigoplus_{i\in I(v^o),j\in[d_i]}s_{ij}^{v^o}
\]
for canonical $s_{ij}^{v^o}\in C_m^\infty(\oPM_{v^o},\CL_i)$ and $s^o\in C_m^\infty(\oPM_{v^o},\cW).$ Moreover, $\s_\Gamma^o$ is transverse, since it satisfies the refined transversality of Lemma \ref{lem:existence}.

Similarly, we have multisections
\[
\s_\Gamma^c=s^c\oplus\bigoplus_{\substack{ i\in I(v^c), \; 1 \leq j \leq d_i \\ (i,j)\neq(1,1)}}s_{ij}^{v^c}\in C_m^\infty(\TRAM_{v^c},E_\Gamma^c),
\]
where $\s_\Gamma^c$ is coherent and $\bar{\s}_\Gamma^c=\bar{s}^c\oplus\bigoplus s_{ij}^{v^c}$ is transverse to zero.  We can write
\[\tsrest|_{\oPM_\Gamma}=\Ass_{\Gamma,e}({s}_\Gamma^o,\s_\Gamma^c).
\]

Since $\s_\Gamma^o\pitchfork 0$, it does not vanish unless
$\rk(E_\Gamma^o)\leq \dim(\oPM_{v^o})$.
Write
\[{Z}^o = \text{Detach}_e^{-1}({\text{Proj}}^{-1}_o (Z(\s_\Gamma^o)))\subseteq \oPM_\Gamma,\]
the preimage in $\oPM_\Gamma$ of the zero locus of
$\s_\Gamma^o$.  In case $e$ is Neveu--Schwarz, $\s_\Gamma^c$ also does not vanish unless $\text{rank}(E_\Gamma^c) \leq \dim(\oPM_{v^c})$.  In case $e$ is Ramond, let $\bar{\s}_\Gamma^c$ be the restriction $\s_\Gamma^c|_{\oCM_{v_\Gamma^c}},$ which we identify with a multisection of
\[\bar{E}^c_\Gamma = \bar{\cW}^c\oplus\CL_1^{\oplus n-1}\oplus \bigoplus_{i \in I(v^c_\Gamma)\setminus \{1\}} \CL_i^{ \oplus d_i} \longrightarrow \oCM_{v_\Gamma^c}.\]
(We use the notation $\bar{\cW}^c$ to denote the Witten bundle over $\oCM_{v_\Gamma^c}$ in order to avoid confusion with the restriction to $\oCM_{v_\Gamma^c}$ of $\cW^c\to\TRAM_{v_\Gamma^c}$.)  Then,
\begin{equation}\label{eq:0_for_Ram}
\tsrest|_{{Z}^o}=
\Ass_{\Gamma,e}({\s}_\Gamma^o ,\s_\Gamma^c)|_{{Z}^o}
= ({\text{Proj}}_c\circ \text{Detach}_e)^*\bar{\s}_\Gamma^c,
\end{equation}
where the equation is interpreted using the injection of \eqref{eq:decompses2}.
By the transversality requirement of Lemma \ref{lem:existence}, $\bar{\s}_\Gamma^c$ vanishes nowhere unless $\rk(\bar{E}_\Gamma^c)\leq \dim(\oCM_{v_\Gamma^c})$. Thus, $e(E_1|_{\oPM_\Gamma};\s_1|_{\oPM_\Gamma})=0$ unless
\begin{equation}\label{eq:rks=dims}
\rk(E_\Gamma^o)= \dim(\oCM_{v_\Gamma^o}), \qquad\text{rank}(\bar{E}_\Gamma^c)=\dim\oCM_{v_\Gamma^c},
\end{equation}
independently of whether $e$ is Ramond. Assuming these equalities, we have
\begin{equation}\label{eq:1_for_Ram}
\# Z({\s}_\Gamma^o) = \sum_{p \in Z({\s}_\Gamma^o)} \eps_p,
\end{equation}
where $\eps_p$ is the weight of the zero $p$; see Notation \ref{nn:weighted_signed}.

In case $e$ is Ramond, using \eqref{eq:0_for_Ram}, we can write,
\begin{equation}\label{eq:2_for_Ram}
\#Z(\tsrest|_{\oPM_\Gamma})=
\sum_{p \in Z({\s}_\Gamma^o)}\eps_p \# Z \left(({\text{Proj}}_c\circ \text{Detach}_e)^*\bar{\s}_\Gamma^c|_{\text{Detach}_e^{-1}(({\text{Proj}}_o)^{-1}(p))}\right).
\end{equation}
Since $\left(\text{Proj}_c\circ\text{Detach}_e\right)\big|_{\text{Detach}_e^{-1}(({{\text{Proj}}}_o)^{-1}(p))}$ is a degree-$1$ map between smooth orbifolds $\text{Detach}_e^{-1}(({{\text{Proj}}}_o)^{-1}(p))$ and $\oCM_{v_\Gamma^c}^{1/r}$ (which even induces a diffeomorphism of the coarse spaces), we have
\begin{equation}\label{eq:3_for_Ram}
\# Z ((\text{Proj}_c\circ \text{Detach}_e)^*\bar{\s}_\Gamma^c|_{\text{Detach}_e^{-1}((\text{Proj}_o)^{-1}(p))})=\# Z(\bar{\s}_\Gamma^c),
\end{equation}
By transversality of $\bar{\s}_\Gamma^c$ and the interpretation of the Euler class of a bundle as the weighted zero count of a transversal multisection (Theorem \ref{prop:euler_as_zero_locus}), we have
\begin{align}\label{eq:4_for_Ram}
\# Z(\bar{\s}_\Gamma^c)=\int_{\oCM_{v_\Gamma^c}^{1/r}}e(\bar{E}^c_\Gamma) = \frac{1}{r}\left\langle \tau^{-1}_0\tau^{a_1}_{n-1}\prod_{\substack{i\in I(v^c_\Gamma)\\ i \neq 1}}\tau^{a_i}_{d_i}\right\rangle^{1/r,\text{ext}}_0.
\end{align}
Furthermore, since $\s_\Gamma^o$ is transverse and canonical, we have
\begin{align}\label{eq:5_for_Ram}
\# Z({\s}_\Gamma^o) = \left\langle \tau^{r-1}_0\prod_{\substack{i\in I(v^o_\Gamma)}}\tau^{a_i}_{d_i}\sigma^k\right\rangle^{1/r,o}_0.
\end{align}
Putting together equations \eqref{eq:1_for_Ram},~\eqref{eq:2_for_Ram},~\eqref{eq:3_for_Ram},~\eqref{eq:4_for_Ram}, and \eqref{eq:5_for_Ram}, and using \eqref{eq:eulert_and_s}, the Ramond case of the lemma is proven.

If $e$ is Neveu--Schwarz, we have the following analogue of equation \eqref{eq:0_for_Ram}:
\[\tsrest|_{{Z}^o}=({\text{Proj}}_c\circ \text{Detach}_e)^*\s_\Gamma^c\]
(more precisely, we use $({\text{Proj}}_c\circ \text{Detach}_e)^*\s_\Gamma^c$ to denote $\text{Detach}_e^*({\text{Proj}}_c^*\s_\Gamma^c\boxplus0)$).
From this it follows that
\begin{equation}\label{eq:2_for_NS}
\#Z(\tsrest|_{\oPM_\Gamma})=
\sum_{p \in Z({\s}_\Gamma^o)} \eps_p \# Z (({\text{Proj}}_c\circ \text{Detach}_e)^*\s_\Gamma^c|_{\text{Detach}_e^{-1}(({\text{Proj}}_o)^{-1}(p))}).
\end{equation}
Again the restricted map $\left(\text{Proj}_c\circ\text{Detach}_e\right)\big|_{\text{Detach}_e^{-1}(({{\text{Proj}}}_o)^{-1}(p))}$ is a degree-$1$ map between smooth orbifolds $\text{Detach}_e^{-1}(({{\text{Proj}}}_o)^{-1}(p)),~\oCM_{v_\Gamma^c}^{1/r},$ and thus
\begin{equation}\label{eq:3_for_NS}
\# Z (({\text{Proj}}_c\circ \text{Detach}_e)^*\s_\Gamma^c|_{\text{Detach}_e^{-1}((\text{Proj}_o)^{-1}(p))})=\# Z(\s_\Gamma^c),
\end{equation}
%where the right-hand side is the cardinality of the zero locus of ${\s}_\Gamma^c$ as a transverse multisection of $\cW\to\oCM_{v_\Gamma^c}^{1/r}$.  
By transversality of $\s_\Gamma^c$ and $\s_\Gamma^o$, and the fact that $\s_\Gamma^o$ is canonical, we obtain
\begin{align*}
&\int\limits_{\oPM_{{v_\Gamma^o}}} \hspace{-0.1cm}e(E^o_\Gamma,\s^o_\Gamma) = \left\langle \tau^a_0\hspace{-0.1cm}\prod_{i\in I(v^o_\Gamma)}\hspace{-0.1cm}\tau^{a_i}_{d_i}\sigma^k\right\rangle^{\frac{1}{r},o}_0, \int\limits_{\oCM_{{v_\Gamma^c}}} \hspace{-0.1cm}e(E^c_\Gamma) =
\frac{1}{r}\left\langle \tau^b_0\tau_{n-1}^{a_1}\hspace{-0.1cm}\prod_{\substack{i\in I(v^c_\Gamma)\\ i \neq 1}}\hspace{-0.1cm}\tau^{a_i}_{d_i}\right\rangle^{\frac{1}{r},c}_0
\end{align*}
where $a$ is the twist of the half-edge of $e$ that lies in the open side and $b$ is the twist of the other half-edge.  These equations, together with~\eqref{eq:1_for_Ram},~\eqref{eq:2_for_NS},~\eqref{eq:3_for_NS}, and \eqref{eq:eulert_and_s}, prove the Neveu--Schwarz case, so Lemma~\ref{lem:closed_contribution} is proved.
\end{proof}We now analyze the difference between $\ttt$ and a canonical multisection.
\begin{lem}\label{lem:bc_contirbution}
We have
\begin{multline*}
\#Z(\s)-\#Z(\ttts)
=\hspace{-0.2cm} \sum_{\substack{A \sqcup B = \{2,\ldots,l\} \\ k_1 + k_2 = k-1}} \hspace{-0.1cm}\binom{k}{k_1} \left\langle \tau_{n-1}^{a_1} \prod_{i \in A} \tau^{a_i}_{d_i} \sigma^{k_1}\right\rangle^{\frac{1}{r},o}_0 \hspace{-0.1cm}\left\langle \prod_{i \in B} \tau^{a_i}_{d_i} \sigma^{k_2+2}\right\rangle^{\frac{1}{r}, o}_0.
\end{multline*}
\end{lem}
\begin{proof}
The proof of this lemma resembles the proof of Lemma 4.13 in \cite{PST14}, and for some of details we will refer the reader to the relevant places in \cite{PST14}.

Let ${U_{oo}} \subset \d^0\Gammar$ be the collection of graphs $\Gamma$ with exactly two vertices $v^\pm_\Gamma,$ both open, and a single edge $e_\Gamma,$ such that $1 \in I(v^-_\Gamma)$, $1 \in B(v^{+}_\Gamma)$, and such that if $h^\pm\in H^B(\Gamma)$ is the half-edge belonging to $v^\pm_\Gamma$, then $\alt(h^-)=\tw(h^-)=0$.  In particular, this implies that $\alt(h^+)=1$ and $\tw(h^+)=r-2$.  Let
$U_{\text{can}}$ be the collection of graphs in $\partial^0\Gammar\setminus U_{oo}$ which either have a contracted boundary half edge or a single edge which is a boundary edge.

$\ttt|_{\CM_\Gamma}$ is canonical for every $\Gamma \in U_{\text{can}}$.  This is proven exactly as \cite[Lemma 4.8]{PST14}, relying on the observation that, in this case, $t$ is pulled back from $\CM_{\CB\Gamma}$.

Let us now describe the difference in behavior between a canonical multisection of $\CL_1$ and $\ttt$ on $\oPM_\Gamma$ for $\Gamma \in U_{oo}$, which is responsible for the extra contribution to the topological recursion. Let $p \in \oPM_{\CB\Gamma}$ and let $F_p$ be the fiber over $p$ of the map $F_\Gamma : \oPM_\Gamma \to \oPM_{\CB\Gamma}$, equipped with its natural orientation. Generically, $F_p$ is the union of $a = k(v^-_\Gamma)$ closed intervals, each corresponding to a segment between marked points on which the illegal half-node can lie.
By \cite[Observation 4.9]{PST14}, we have:
\begin{obs}\label{obs:trivial_on_fiber}
The restriction line $\CL_i|_{F_p}$ is canonically trivialized.
\end{obs}

We shall thus think of sections of $\CL_i|_{F_p}$ as complex-valued functions, well-defined up to multiplication by a constant in $\C^*$.  By definition, we have:
\begin{obs}[c.f. Observation 4.10 of \cite{PST14}]\label{obs:canonical_const}
A canonical section of $\CL_i|_{F_p}$ is constant.
\end{obs}

The section $\ttt,$ on the other hand, winds non-trivially around $F_p$.   Indeed, for $i \in B(v^{-}_\Gamma),$ let $\Gamma_i$ be the unique graded graph in $\d^B\Gamma$ with three open vertices $v_i^0$ and $v_i^\pm$ and two boundary edges $e^\pm$, such that $B(v_i^0) = \{i\}.$
The boundary $\partial F_p$ corresponds to two stable disks modelled on the graphs $\Gamma_i$ for $i \in B(v^-_\Gamma),$ one for each cyclic order of $B(v^0_i).$ Let $\hat F_p$ be the quotient space of $F_p$ obtained by identifying, for each $i \in B(v^-_\Gamma)$, the two boundary points corresponding to $\Gamma_i.$  $\hat F_p$ is homeomorphic to $S^1.$ We have (\cite[Observations 4.11 and 4.12] {PST14}):
\begin{obs}\label{obs:cont_on_fiber}
The section $\ttt|_{F_p}$ descends to a continuous function $\theta_p : \hat F_p \to \C^*$.  Furthermore, the winding number of $\theta_p$ is $-1.$
\end{obs}

Returning to the proof of Lemma \ref{lem:bc_contirbution}, let
$E_2 = \CL_1 \to \oPMr$, so that $E = E_1 \oplus E_2.$ Since $t|_{\mathcal{M}_{\Gamma}}$ is canonical for every $\Gamma \in U_{\text{can}}$, we may apply Lemma~\ref{lem:partial_homotopy} to the multisections $\s$ and $\ttts,$ with the preceding choice of $E_1$ and $E_2$, $C=\coprod_{\Gamma\in U_{\text{can}}}\CM_\Gamma^{1/r}$ and $V_+=U_+.$  Thus, we may find a homotopy $H$ between $\mathbf{s}$ and $\ttts$ of the form~\eqref{eq:form_of_hom} such that
\begin{enumerate}[(i)]
\item the restriction of $H$ to each stratum $\CM_\Lambda^{1/r}\times[0,1]$ is transverse to zero;
\item $H$ is nowhere-vanishing on $\CM_\Gamma^{1/r} \times [0,1]$ for $\Gamma \in U_{\text{can}}$;
\item $H$ is nowhere-vanishing on $V'_+\times[0,1]$, where $V'_+$ is of the form \eqref{eq:V_+};%as in the statement of Lemma \ref{lem:partial_homotopy};
\item  the projection of $H$ to $E_1$ equals $\srest$ at all times.
\end{enumerate}
In particular, by transversality, $H$ does not vanish on $\CM_\Gamma \times [0,1]$ if $\dim\CM_\Gamma\leq\dim\oPMr-2$. Hence, by Lemma \ref{lem:zero diff as homotopy}, we can write
\begin{equation}\label{eq:rsH}
\#Z(\s)-\#Z(\ttts)= -\# Z(H) = -\sum_{\Gamma \in U_{oo}} \# Z\left(H|_{\oPM_\Gamma\times[0,1]}\right).
\end{equation}
Write $\pi : \partial \oPMr \times [0,1] \to \partial\oPMr$ for the projection to the first factor, and decompose $H = H_1 \oplus H_2,$ where $H_i \in C^\infty_m(\pi^* E_i).$ Then $H_1 = \pi^* \srest.$

Since $\srest$ is canonical, we can write $\srest|_{\oPM_{\Gamma}} = F^*_{\Gamma} \srest^{\CB\Gamma}$ for each $\Gamma\in\partial^0\Gammar$.  Transversality implies that $\srest^{\CB\Gamma}$ has isolated zeroes in $\CM_{\CB\Gamma}^{1/r}$ and
$
Z\left(\srest|_{\oPM_\Gamma}\right) \subset F_\Gamma^{-1}\left(\CM_{\CB\Gamma}^{1/r}\right).
$

Write $\# Z\left(\srest^{\CB\Gamma}\right) = \sum_{p \in Z\left(\srest^{\CB\Gamma}\right)} \eps_p$, where $\eps_p$ is the weight of the zero $p$ defined in Notation \ref{nn:weighted_signed}.  It follows from \cite[Theorem 5.2]{BCT1}
that $\# Z\left(H|_{\oPM_\Gamma\times[0,1]}\right)$ for $\Gamma \in U_{oo}$ equals
\begin{equation}\label{eq:HG}
\# Z\left(\pi^* F_\Gamma^* \srest^{\CB\Gamma}\right) \cap Z(H_2)
= \sum_{p \in Z\left(\srest^{\CB\Gamma}\right)}\eps_p\cdot \# Z\left(H_2|_{F_p \times [0,1]}\right).
\end{equation}
Since $H$ is of the form~\eqref{eq:form_of_hom}, we have, for some canonical multisection $w,$
\begin{equation}\label{eq:H2}
H_2(q,s) = t(q)s + s_{11}(q)(1-s) + s(1-s) w(q).
\end{equation}
%where $w$ is a canonical multisection.

Let $p \in Z(\srest^{\CB\Gamma}).$ Observations \ref{obs:trivial_on_fiber}, \ref{obs:canonical_const}, and~\ref{obs:cont_on_fiber} and equation~\eqref{eq:H2} imply that $H_2|_{F_p \times [0,1]}$ descends to a homotopy $\hat H_{2,p}$ on $\hat F_p \times [0,1]$ that can be viewed as taking values in $\C.$ As a consequence,
\begin{equation}\label{eq:HhH}
\#Z\left(H_2|_{F_p \times [0,1]}\right) = \#Z\left(\hat H_{2,p}\right).
\end{equation}
Now, $\hat H_{2,p}|_{\hat F_p \times \{0\}}$ is canonical and $\hat H_{2,p}|_{\hat F_p\times \{1\}} = \theta_p,$ so by Observations~\ref{obs:canonical_const} and~\ref{obs:cont_on_fiber}, we obtain $\#Z\left(\hat H_{2,p}\right) = -1.$
%\begin{equation}\label{eq:hH-1}
%\#Z\left(\hat H_{2,p}\right) = -1.
%\end{equation}
Combining the last equality, with equations \eqref{eq:rsH},~\eqref{eq:HG}, and~\eqref{eq:HhH}, we have
\begin{equation}\label{eq:sumU}
\#Z(\s)-\#Z(\ttts)= \sum_{\Gamma \in U_{oo}} \# Z(\srest^{\CB \Gamma}).
\end{equation}

In order to finish, we need to calculate $\# Z\left(\srest^{\CB \Gamma}\right)$ for $\Gamma \in U_{oo}.$ Denote by $v^\pm_{\CB\Gamma} \in V(\CB \Gamma)$ the vertices corresponding to $v^\pm_{\CB\Gamma}  \in V(\Gamma).$ We consider them both as vertices and as graphs with single vertex.
We have a canonical identification
\[
\oPM_{\CB\Gamma} \simeq \oPM_{v_{\CB\Gamma}^+} \times \oPM_{v_{\CB\Gamma}^-}.
\]

From now on write $\cW$ for $\cW\to\oPM_{\CB\Gamma},$ $\cW^\pm$ for $\cW\to\oPM_{v_{\CB\Gamma}^\pm},$ and
\begin{align*}
&E_{\CB\Gamma} = \cW\oplus\CL_1^{\oplus n-1} \oplus \bigoplus_{i = 2}^l \CL_i^{\oplus d_i} \longrightarrow \oPM_{\CB\Gamma},~E^-_{\CB\Gamma} = \cW^-\oplus\bigoplus_{i \in I(v^-_\Gamma)} \CL_i^{\oplus d_i} \longrightarrow \oPM_{v_{\CB\Gamma}^-},\\
&E^+_{\CB\Gamma} = \cW^+\oplus\CL_1^{\oplus n-1}\oplus\bigoplus_{i \in I(v^+_\Gamma)\setminus\{1\}}\CL_i^{\oplus d_i} \longrightarrow \oPM_{v_{\CB\Gamma}^+}.
\end{align*}
%\begin{align*}
%&E^+_{\CB\Gamma} = \cW^+\oplus\CL_1^{\oplus n-1}\oplus\bigoplus_{i \in I(v^+_\Gamma)\setminus\{1\}}\CL_i^{\oplus d_i} \longrightarrow \oPM_{v_{\CB\Gamma}^+}, \\
%&E^-_{\CB\Gamma} = \cW^-\oplus\bigoplus_{i \in I(v^-_\Gamma)} \CL_i^{\oplus d_i} \longrightarrow \oPM_{v_{\CB\Gamma}^-}, \\
%&E_{\CB\Gamma} = \cW\oplus\CL_1^{\oplus n-1} \oplus \bigoplus_{i = 2}^l \CL_i^{\oplus d_i} \longrightarrow \oPM_{\CB\Gamma},
%\end{align*}
Let $
\text{Proj}_\pm : \oPM_{v_{\CB\Gamma}^+} \times \oPM_{v_{\CB\Gamma}^-} \longrightarrow \oPM_{v_{\CB\Gamma}^\pm}
$
denote the projections. We have
\[
E_{\CB\Gamma}=\text{Proj}_+^*E^+_{\CB\Gamma} \boxplus \text{Proj}_-^* E^-_{\CB\Gamma}.
\]

Since $\s_1$ is special canonical, and using Observation \ref{obs:s_v_canonical}, we can decompose
\begin{equation*}
\srest^{\CB\Gamma} = \text{Proj}_+^* \s_{\CB\Gamma}^+ \oplus \text{Proj}_-^* \s_{\CB\Gamma}^-,
\end{equation*}
where $\s^\pm_{\CB\Gamma}\in C_m^\infty(\oPM_{v_{\CB\Gamma}^\pm},E_{\CB\Gamma}^\pm)$ are canonical.
Since $\s$ was chosen to satisfy the refined transversality of Lemma \ref{lem:existence}, the multisections $\s^\pm_{\CB\Gamma}\pitchfork 0$. Thus,
\begin{align}\label{eq:3-before-last}
\# Z\left(\srest^{\CB\Gamma}\right) =& \# Z\left(\text{Proj}_+^* \s_{\CB \Gamma}^+\right) \cap Z\left(\text{Proj}_-^* \s_{\CB \Gamma}^-\right)=\\
=&\# \text{Proj}_+^{-1}\left(Z\left(\s_{\CB \Gamma}^+\right)\right) \cap \text{Proj}_-^{-1}\left(Z\left(\s_{\CB \Gamma}^-\right)\right).\notag
\end{align}

Now, $\# Z\left(\srest^{\CB\Gamma}\right)$ vanishes unless $\rk E_{\CB \Gamma}^\pm = \dim_\C \oCM_{\CB\Gamma}^\pm,$ by transversality and dimension counting. In case the ranks of the bundles $E_{\CB \Gamma}^\pm$ do agree with the dimensions of the corresponding moduli spaces, transversality implies that $\s_{\CB\Gamma}^\pm|_{\partial^0 \oCM_{\CB\Gamma}^\pm}$ is nowhere-vanishing. It follows that
\begin{equation}\label{eq:2-before-last}
\# \text{Proj}_+^{-1}\left(Z\left(\s_{\CB \Gamma}^+\right)\right) \cap \text{Proj}_-^{-1}\left(Z\left(\s_{\CB \Gamma}^-\right)\right)= \#Z\left(\s_{\CB \Gamma}^+\right)\#Z\left(\s_{\CB \Gamma}^-\right).
\end{equation}
Since $\s_{\CB\Gamma}^\pm$ is canonical and transverse,
\begin{equation}\label{eq:last}
\#Z\left(\s_{\CB\Gamma}^\pm\right)=\left\langle \prod_{i \in I(v_{\CB\Gamma}^\pm)} \tau_{a_i} \sigma^{k\left( v_{\CB\Gamma}^\pm\right)} \right \rangle_0^{\frac{1}{r},o}.
\end{equation}
For each $\Gamma \in U_{oo}$, $1 \in B(v^{+}_\Gamma)$, and $e^+$ is legal with twist $r-2.$ Thus, $k\left( v_{\CB\Gamma}^+\right) \geq 2$ for $\Gamma \in U_{oo}$. Equations~\eqref{eq:sumU},~\eqref{eq:3-before-last},~\eqref{eq:2-before-last},~\eqref{eq:last},
and the facts $k\left( v_{\CB\Gamma}^+\right) + k\left(v_{\CB\Gamma}^-\right) = k + 1$, $I(v_{\CB\Gamma}^+) \cup I(v_{\CB\Gamma}^-) = \{2,\ldots,l\}$ imply Lemma~\ref{lem:bc_contirbution}.
\end{proof}

The first item of Theorem \ref{thm:TRR} now follows from Lemmas \ref{lem:closed_contribution} and \ref{lem:bc_contirbution}. The proof of the second item of Theorem \ref{thm:TRR} is similar, except that $\varphi_C$ is defined to be the unique meromorphic differential on the normalization of $C$ with simple poles at $\bar{z}_1$ and $z_2$ and possibly at the nodes. The residues at $\bar{z}_1$ and $z_2$ are is $-1$ and $1$, respectively, and at any node, the sum of residues at the two half-nodes is zero.  The section $\tildet\left(C\right)$ is defined as the evaluation of $\varphi_C$ at $z_1$. The rest of the proof is exactly as for the first item. % This completes the proof of Theorem~\ref{thm:TRR}.
\end{proof}

\section{Open intersection numbers and the Gelfand-Dickey hierarchy}\label{section:open numbers and GD hierarchy}
In this section, we prove Theorems~\ref{thm:main},~\ref{theorem:primary numbers}, and~\ref{thm:open-closed}, and we derive open string and dilaton equations for the open $r$-spin intersection numbers.

\subsection{Primary extended closed $r$-spin intersection numbers}

As a preliminary step, let us derive an explicit formula for the correlators $\<\tau^{-1}\prod\tau^{\alpha_i}\>^{\frac{1}{r},\text{ext}}_0$.

\begin{prop}\label{prop:primary extended closed}
Let $n\ge 2$ and $0\le\alpha_1,\ldots,\alpha_n\le r-1$. Then we have
\begin{gather}
\<\tau^{-1}\prod_{i=1}^n\tau^{\alpha_i}\>^{\frac{1}{r},\text{ext}}_0=
\begin{cases}
\frac{(n-2)!}{(-r)^{n-2}},&\text{if $\frac{\sum\alpha_i-(r-1)}{r}=n-2$},\\
0,&\text{otherwise}.
\end{cases}
\end{gather}
\end{prop}
\begin{proof}
Equivalently, we have to prove that
\begin{gather}\label{eq:primary extended closed,equiv}
\<\tau^{-1}\prod_{i=1}^l\tau^{\alpha_i}(\tau^{r-1})^k\>^{\frac{1}{r},\text{ext}}_0=
\frac{(k+l-2)!}{(-r)^{k+l-2}},
\end{gather}
where $l+k\ge 2$, $0\le\alpha_i\le r-2$, and the condition $\sum_{i=1}^l(r-\alpha_i)+k=r+1$ is satisfied. For $l\le 1$, formula~\eqref{eq:primary extended closed,equiv} says that
\begin{gather}\label{eq:one-point extended closed}
\<\tau^{-1}\tau^\alpha(\tau^{r-1})^{\alpha+1}\>^{\frac{1}{r},\text{ext}}_0=\frac{\alpha!}{(-r)^\alpha},\quad 0\le \alpha\le r-1,
\end{gather}
and it was proved in~\cite[Section~4.4]{BCT_Closed_Extended}.

Suppose that $l\ge 2$. For a subset $I\subset [l]$, let
$k_I:=r+1-\sum_{i\in I}(r-\alpha_i)$ and $A_I:=\<\tau^{-1}\prod_{i\in I}\tau^{\alpha_i}(\tau^{r-1})^{k_I}\>^{\frac{1}{r},\text{ext}}_0$. In~\cite{BCT_Closed_Extended}, we found the recursion
\begin{gather}\label{eq:recursion for extended for many points}
\frac{(r+k-1)!}{k!(-r)^{r-1}}A_{[l]}=\hspace{-0.2cm}\sum_{\substack{I\sqcup J=[l]\\1\in I,\,l\in J}}\hspace{-0.1cm}{r+k-1\choose k_I-1}A_IA_J-\hspace{-0.2cm}\sum_{\substack{I\sqcup J=[l]\\1,l\in I,\,J\ne\emptyset}}\hspace{-0.2cm}{r+k-1\choose k_I}A_IA_J,
\end{gather}
which allows one to compute all primary closed extended intersection numbers starting from the intersection numbers~\eqref{eq:one-point extended closed}. Therefore, we have to check that the right-hand side of~\eqref{eq:primary extended closed,equiv} satisfies recursion~\eqref{eq:recursion for extended for many points}, i.e., to check that
\begin{align*}
\frac{(r+k-1)!}{k!(-r)^{r-1}}\frac{(k+l-2)!}{(-r)^{k+l-2}}=&\sum_{\substack{I\sqcup J=[l]\\1\in I,\,l\in J}}{r+k-1\choose k_I-1}\frac{(|I|+k_I-2)!}{(-r)^{|I|+k_I-2}}\frac{(|J|+k_J-2)!}{(-r)^{|J|+k_J-2}}\\
&-\sum_{\substack{I\sqcup J=[l]\\1,l\in I,\,J\ne\emptyset}}{r+k-1\choose k_I}\frac{(|I|+k_I-2)!}{(-r)^{|I|+k_I-2}}\frac{(|J|+k_J-2)!}{(-r)^{|J|+k_J-2}}\notag
\end{align*}
for $l\ge 2$, $k\ge 0$, and $0\le \alpha_1,\ldots,\alpha_l\le r-2$ satisfying $\sum (r-\alpha_i)+k=r+1$. Since $I\sqcup J=[l]$ implies $k_I+k_J=r+1+k$, the last identity is equivalent~to
\begin{align}
\frac{(k+l-2)!}{k!}=&\sum_{\substack{I\sqcup J=[l]\\1\in I,\,l\in J}}\left(\prod_{i=1}^{|I|-1}(k_I-1+i)\right)\left(\prod_{j=1}^{|J|-1}(k_J-1+j)\right)\label{eq:primary extended,identity2}\\
&-\sum_{\substack{I\sqcup J=[l]\\1,l\in I,\,J\ne\emptyset}}\left(\prod_{i=1}^{|I|-2}(k_I+i)\right)\left(\prod_{j=1}^{|J|}(k_J-2+j)\right).\notag
\end{align}
Let $b_i:=r-\alpha_i$ for $i=1,\ldots,l$, and define $b_I:=\sum_{i\in I}b_i$ for $I\subset [l]$. Then $k_I=r+1-b_I$, and we can rewrite identity~\eqref{eq:primary extended,identity2} in the following way:
\begin{align}
\frac{(k+l-2)!}{k!}=&\sum_{\substack{I\sqcup J=[l]\\1\in I,\,l\in J}}\left(\prod_{i=1}^{|I|-1}(r-b_I+i)\right)\left(\prod_{j=1}^{|J|-1}(k+b_I-1+j)\right)\label{eq:primary extended,identity3}\\
&-\sum_{\substack{I\sqcup J=[l]\\1,l\in I,\,J\ne\emptyset}}\left(\prod_{i=1}^{|I|-2}(k+b_J+i)\right)\left(\prod_{j=1}^{|J|}(r-b_J-1+j)\right).\notag
\end{align}

Consider now $b_1,\ldots,b_l$ as formal variables. Note that the right-hand side of~\eqref{eq:primary extended,identity3} is a polynomial in the variables $b_i$, and, moreover, it does not depend on $b_l$. Let us prove that identity~\eqref{eq:primary extended,identity3} holds as an identity between polynomials in the ring $\mbC[b_1,\ldots,b_{l-1}]$. We prove this by induction on $l$. The case $l=2$
is trivial. Suppose $l\ge 3$. Denote the right-hand side of~\eqref{eq:primary extended,identity3} by $R_l(b_1,\ldots,b_{l-1})$. Note that the degree of $R_l$ is equal to $l-2$. So, in order to prove identity~\eqref{eq:primary extended,identity3}, it is sufficient to check that
\begin{align}
&R_l(b_1,\ldots,b_{l-1})|_{b_i=0}=\frac{(k+l-2)!}{k!},\quad\text{for any $2\le i\le l-1$},\label{eq:primary extended,property1}\\
&\Coef_{b_2\cdots b_{l-1}}R_l(b_1,\ldots,b_{l-1})=0.\label{eq:primary extended,property2}
\end{align}

Let us check~\eqref{eq:primary extended,property1}. Note that the polynomial $R_l$ is symmetric with respect to permutations of the variables $b_2,\ldots,b_{l-1}$. So, without loss of generality, we can assume that $i=2$. A direct computation gives
\begin{align*}
R_l(b_1,\ldots,b_{l-1})|_{b_2=0}=&(r+k+l-2)R_{l-1}(b_1,b_3,\ldots,b_{l-1})-\frac{(k+l-3)!}{k!}r=\\
\stackrel{\begin{smallmatrix}\text{by the induction}\\\text{assumption}\end{smallmatrix}}{=}&(r+k+l-2)\frac{(k+l-3)!}{k!}-\frac{(k+l-3)!}{k!}r=\\
=&\frac{(k+l-2)!}{k!}.
\end{align*}
Thus, equation~\eqref{eq:primary extended,property1} is true.

Let us check equation~\eqref{eq:primary extended,property2}. It is easy to see that the coefficient of $b_2\cdots b_{l-1}$ in a term in the first sum on the right-hand side of~\eqref{eq:primary extended,identity3} is equal to $(-1)^l(l-2)!$, if $I=[l]\backslash\{l\}$, and is equal to zero
otherwise. Similarly, the coefficient of $b_2\cdots b_{l-1}$ in a term in the second sum on the right-hand side of~\eqref{eq:primary extended,identity3} is equal to $(-1)^l(l-2)!$, if $J=[l]\backslash\{1,l\}$, and is equal to zero otherwise. So equation~\eqref{eq:primary extended,property2} is also true. This completes the proof of the proposition.
\end{proof}

\subsection{Proof of Theorem~\ref{theorem:primary numbers}}\label{subsection:open primaries,proof}

We see that the open intersection numbers satisfy the following properties:
\begin{align}
&\<\prod_{i=1}^l\tau^{\alpha_i}_{d_i}\sigma^k\>^{\frac{1}{r},o}_0\hspace{-0.25cm}=0,\,\text{if}\,\,\frac{(k-1)(r-2)+2\sum\alpha_i}{r}+2\sum d_i\ne k+2l-3,\label{eq:property1}\\
&\<\prod_{i=1}^l\tau^{\alpha_i}_{d_i}\>^{\frac{1}{r},o}_0=0,\label{eq:property2}
\end{align}
\begin{align}
&\frac{\d^2 F^{\frac{1}{r},o}_0}{\d t^\alpha_{p+1}\d t^\beta_q}=\sum_{\mu+\nu=r-2}\frac{\d^2 F^{\frac{1}{r},c}_0}{\d t^\alpha_p\d t^\mu_0}\frac{\d^2 F^{\frac{1}{r},o}_0}{\d t^\nu_0\d t^\beta_q}+\frac{\d F^{\frac{1}{r},\text{ext}}_0}{\d t^\alpha_p}\frac{\d^2 F^{\frac{1}{r},o}_0}{\d t^{r-1}_0\d t^\beta_q}+\frac{\d F^{\frac{1}{r},o}_0}{\d t^\alpha_p}\frac{\d^2 F^{\frac{1}{r},o}_0}{\d s\d t^\beta_q},\label{eq:property3}\\
&\frac{\d^2 F^{\frac{1}{r},o}_0}{\d t^\alpha_{p+1}\d s}=\sum_{\mu+\nu=r-2}\frac{\d^2 F^{\frac{1}{r},c}_0}{\d t^\alpha_p\d t^\mu_0}\frac{\d^2 F^{\frac{1}{r},o}_0}{\d t^\nu_0\d s}+\frac{\d F^{\frac{1}{r},\text{ext}}_0}{\d t^\alpha_p}\frac{\d^2 F^{\frac{1}{r},o}_0}{\d t^{r-1}_0\d s}+\frac{\d F^{\frac{1}{r},o}_0}{\d t^\alpha_p}\frac{\d^2 F^{\frac{1}{r},o}_0}{\d s^2},\label{eq:property4}\\
&\<\tau^1\sigma^2\>^{\frac{1}{r},o}_0=1.\label{eq:property5}
\end{align}
Property~\eqref{eq:property1} follows from the formula for the rank of the open Witten bundle. Vanishing~\eqref{eq:property2} is a consequence of Proposition~\ref{prop:no_zeroes_for_Witten_without_bdry_markings}. Equations~\eqref{eq:property3} and~\eqref{eq:property4} are the topological recursion relations from Theorem~\ref{thm:TRR}. Intersection number~\eqref{eq:property5} was computed in Example~\ref{ex:t1s^2}. Let us prove that these properties imply formula~\eqref{eq:primary open numbers}.

For $l=1$, formula~\eqref{eq:primary open numbers} says that
\begin{gather}\label{eq:primary open one-point}
\<\tau^\alpha\sigma^{\alpha+1}\>^{\frac{1}{r},o}_0=\alpha!,\quad 0\le\alpha\le r-1.
\end{gather}
Consider the intersection number $\<\tau^1_1\tau^\gamma\sigma^{\gamma+2}\>^{\frac{1}{r},o}_0$ with $0\le\gamma\le r-2$
and apply two topological recursions to it. First, we can apply recursion~\eqref{eq:property3} with $\alpha=1$, $p=0$, $\beta=\gamma$, $q=0$, and we get that $\<\tau^1_1\tau^\gamma\sigma^{\gamma+2}\>^{\frac{1}{r},o}_0$ is equal to
\begin{gather}\label{first expression for auxiliary for open one-point}
{\gamma+2\choose 2}\<\tau^1\sigma^2\>^{\frac{1}{r},o}_0\<\tau^\gamma\sigma^{\gamma+1}\>^{\frac{1}{r},o}_0.
\end{gather}
On the other hand, by formula~\eqref{eq:property4} with $\alpha=1$ and $p=0$, the intersection number $\<\tau^1_1\tau^\gamma\sigma^{\gamma+2}\>^{\frac{1}{r},o}_0$ is equal to
\begin{align}\label{second expression for auxiliary for open one-point}
&{\gamma+1\choose 2}\<\tau^1\sigma^2\>^{\frac{1}{r},o}_0\<\tau^\gamma\sigma^{\gamma+1}\>^{\frac{1}{r},o}_0+\\
&+\<\tau^1\tau^\gamma\tau^{r-3-\gamma}\>^{\frac{1}{r},\text{ext}}_0\<\tau^{\gamma+1}\sigma^{\gamma+2}\>^{\frac{1}{r},o}_0+\delta_{\gamma,r-2}\<\tau^1\tau^{r-2}\>^{\frac{1}{r},o}_0\<\sigma^{r+1}\>^{\frac{1}{r},o}_0.\notag
\end{align}
By property~\eqref{eq:property2}, the last summand here vanishes. Note also that \[\<\tau^1\tau^\gamma\tau^{r-3-\gamma}\>^{\frac{1}{r},\text{ext}}_0=1\] (see~Proposition~\ref{prop:primary extended closed} and, for example, \cite[Section~0.6]{PPZ}). Therefore, equating expressions~\eqref{first expression for auxiliary for open one-point} and~\eqref{second expression for auxiliary for open one-point}, we obtain \[(\gamma+1)\<\tau^1\sigma^2\>^{\frac{1}{r},o}_0\<\tau^\gamma\sigma^{\gamma+1}\>^{\frac{1}{r},o}_0=\<\tau^{\gamma+1}\sigma^{\gamma+2}\>^{\frac{1}{r},o}_0,\] for $0\le\gamma\le r-2$. Since $\<\tau^1\sigma^2\>^{\frac{1}{r},o}_0=1$, we get formula~\eqref{eq:primary open one-point}.

The intersection number $\<\prod_{i=1}^l\tau^{\alpha_i}\sigma^k\>^{\frac{1}{r},o}_0$ is zero unless $\frac{(k-1)(r-2)+2\sum\alpha_i}{r}=2l+k-3$, which is equivalent to $\sum(r-\alpha_i)+k=r+1$. For a subset $I\subset[l]$ denote $k_I:=r+1-\sum_{i\in I}(r-\alpha_i)$ and $A_I:=\<\prod_{i\in I}\tau^{\alpha_i}\sigma^{k_I}\>^{\frac{1}{r},o}_0$. 

Let us fix $l\ge 2$ and numbers $0\le\alpha_1,\ldots,\alpha_l\le r-1$ satisfying $k:=k_{[l]}\ge 0$. Let us compute the intersection number
$\<\tau_1^{\alpha_1}\prod_{i=2}^l\tau^{\alpha_i}_0\sigma^{k+r}\>^{\frac{1}{r},o}_0$ in two ways:
\begin{enumerate}[(i)]
\item applying relation~\eqref{eq:property3} with $\alpha=\alpha_1$, $\beta=\alpha_l$, and $p=q=0$, we obtain
\begin{gather*}%\label{first for auxiliary for open for many points}
\sum_{\substack{I\sqcup J=[l]\\1\in I,\,l\in J}}\sum_{\substack{\mu\ge -1,\,\nu\ge 0\\\mu+\nu=r-2}}\<\prod_{i\in I}\tau^{\alpha_i}\tau^\mu\>^{\frac{1}{r},\text{ext}}_0\underline{\<\tau^\nu\prod_{j\in J}\tau^{\alpha_j}\sigma^{k+r}\>^{\frac{1}{r},o}_0}+\sum_{\substack{I\sqcup J=[l]\\1\in I,\,l\in J}}{r+k\choose k_I}A_IA_J,
\end{gather*}
where the underlined term vanishes, because $r-\nu+\sum_{j\in J}(r-\alpha_j)+k+r>r+1$, since $J\ne\emptyset$;
\item applying relation~\eqref{eq:property4} with $\alpha=\alpha_1$ and $p=0$, we obtain
\begin{gather*}%\label{second for auxiliary for open for many points}
\delta_{k,0}\<\prod_{i=1}^l\tau^{\alpha_i}\tau^{-1}\>^{\frac{1}{r},\text{ext}}_0\<\tau^{r-1}\sigma^r\>^{\frac{1}{r},o}_0+\sum_{\substack{I\sqcup J=[l]\\1\in I}}{r+k-1\choose k_I}A_IA_J.
\end{gather*}
\end{enumerate}
Equating the resulting two expressions, we get
\begin{align}\label{eq:recursion for open many-point}
&{r+k-1\choose k}\<\sigma^{r+1}\>^{\frac{1}{r},o}_0A_{[l]}+\delta_{k,0}\<\prod_{i=1}^l\tau^{\alpha_i}\tau^{-1}\>^{\frac{1}{r},\text{ext}}_0\<\tau^{r-1}\sigma^r\>^{\frac{1}{r},o}_0=\\
=&\sum_{\substack{I\sqcup J=[l]\\1\in I,\,l\in J}}{r+k-1\choose k_I-1}A_IA_J-\sum_{\substack{I\sqcup J=[l]\\1,l\in I,\,J\ne\emptyset}}{r+k-1\choose k_I}A_IA_J.\notag
\end{align}

Denote $-\frac{1}{r!}\<\sigma^{r+1}\>^{\frac{1}{r},o}_0=C$. For $l=2$, $\alpha_1=r-1$, and $\alpha_2=1$, relation~\eqref{eq:recursion for open many-point} gives $-Cr\<\tau^{r-1}\tau^1\sigma\>=1$, which, in particular, implies that $C\ne 0$. More generally, using~\eqref{eq:recursion for open many-point} we obtain
$$
\<\tau^\alpha\tau^\beta\sigma^{\alpha+\beta+1-r}\>^{\frac{1}{r},o}_0=-\frac{k!}{Cr},\quad \alpha+\beta+1-r\ge 1.
$$
Relation~\eqref{eq:recursion for open many-point} with $l=3$, $k=0$, $\alpha_1=\alpha_3=r-1$, and $\alpha_2=1$, gives
\begin{align*}
\<\tau^{r-1}\tau^1\tau^{r-1}\tau^{-1}\>^{\frac{1}{r},\text{ext}}_0\<\tau^{r-1}\sigma^r\>^{\frac{1}{r},o}_0=&2\<\tau^{r-1}\tau^1\sigma\>^{\frac{1}{r},o}_0\<\tau^{r-1}\sigma^r\>^{\frac{1}{r},o}_0\\
&-\<\tau^{r-1}\tau^{r-1}\sigma^{r-1}\>^{\frac{1}{r},o}_0\<\tau^1\sigma^2\>^{\frac{1}{r},o}_0.
\end{align*}
Substituting the computed expressions for the correlators in this formula, we obtain $C=1$.

Let us finally prove formula~\eqref{eq:primary open numbers}. We see that relation~\eqref{eq:recursion for open many-point} allows us to compute all the intersection numbers $\<\prod_{i=1}^l\tau^{\alpha_i}\sigma^k\>^{\frac{1}{r},o}_0$ starting from the ones with $l\le 1$ or with $k=0$. So it is sufficient to check that the right-hand side of~\eqref{eq:primary open numbers} satisfies equation~\eqref{eq:recursion for open many-point}, i.e., to check the identity
\begin{align*}
{r+k-1\choose k}(-r!)\frac{(k+l-2)!}{(-r)^{l-1}}=&\hspace{-0.25cm}\sum_{\substack{I\sqcup J=[l]\\1\in I,\,l\in J}}\hspace{-0.2cm}{r+k-1\choose k_I-1}\frac{(|I|+k_I-2)!}{(-r)^{|I|-1}}\frac{(|J|+k_J-2)!}{(-r)^{|J|-1}}\\
&\hspace{-1.5cm}-\sum_{\substack{I\sqcup J=[l]\\1,l\in I,\,J\ne\emptyset}}{r+k-1\choose k_I}\frac{(|I|+k_I-2)!}{(-r)^{|I|-1}}\frac{(|J|+k_J-2)!}{(-r)^{|J|-1}}
\end{align*}
for $k\ge 1$, $l\ge 2$, and numbers $0\le\alpha_1,\ldots,\alpha_l\le r-1$ satisfying $\sum (r-\alpha_i)+k=r+1$. It is easy to see that this identity is equivalent to identity~\eqref{eq:primary extended,identity2}, which we have already proved (strictly speaking, in equation~\eqref{eq:primary extended,identity2} we required the condition $0\le\alpha_i\le r-2$, but we actually proved it for any integers $\alpha_1,\ldots,\alpha_l$ satisfying $\sum (r-\alpha_i)+k=r+1$). Theorem~\ref{theorem:primary numbers} is proved.

%%%%%%%%%%%%%%%%%%%%%%%%%%%%%%%%%%%%%%%%%%%%%%%%%%%%%%%%%%%%%%%%%

\subsection{Proof of Theorems~\ref{thm:main} and~\ref{thm:open-closed}}

In~\cite[Theorem 4.6]{BCT_Closed_Extended}, we proved that
\begin{gather*}
F_0^{\frac{1}{r},\text{ext}}(t^{\le r-2}_*,t^{r-1}_*)=\sqrt{-r}\phi_0\left(t^{\le r-2}_*,\frac{1}{\sqrt{-r}}t^{r-1}_*\right).
\end{gather*}
Therefore, Theorems~\ref{thm:main} and~\ref{thm:open-closed} are equivalent.

Let us prove Theorem~\ref{thm:open-closed}. Denote
\begin{align}
&\tF(t^*_*,s):=-\frac{1}{r}\left.F^{\frac{1}{r},\text{ext}}_0\right|_{t^{r-1}_d\mapsto t^{r-1}_d-r\delta_{d,0}s}+\frac{1}{r}F^{\frac{1}{r},\text{ext}}_0,\notag\\
&\<\prod_{i=1}^l\tau^{\alpha_i}_{d_i}\sigma^k\>^{\tF}:=\left.\frac{\d^{l+k}\tF}{\d t^{\alpha_1}_{d_1}\cdots\d t^{\alpha_l}_{d_l}\d s^k}\right|_{\substack{t^*_*=0\\s=0}}.\label{eq:tF-correlators}
\end{align}
We claim that the correlators~\eqref{eq:tF-correlators} satisfy properties~\eqref{eq:property1} -- \eqref{eq:property5}. 

The fact that a correlator $\<\prod_{i=1}^l\tau^{\alpha_i}_{d_i}\sigma^k\>^{\tF}$ is zero unless $\frac{(k-1)(r-2)+2\sum\alpha_i}{r}+2\sum d_i=k+2l-3$ is clear from the definition of the closed extended $r$-spin correlators~\cite[Section 3.1]{BCT_Closed_Extended}. 

The property $\<\prod_{i=1}^l\tau^{\alpha_i}\>^{\tF}=0$ is obvious. 

In~\cite[Lemma 3.6]{BCT_Closed_Extended}, we proved that the function~$F^{\frac{1}{r},\text{ext}}_0$ satisfies 
\begin{gather*}
\frac{\d^2 F^{\frac{1}{r},\text{ext}}_0}{\d t^\alpha_{p+1}\d t^\beta_q}=\sum_{\mu+\nu=r-2}\frac{\d^2 F^{\frac{1}{r},c}_0}{\d t^\alpha_p\d t^\mu_0}\frac{\d^2 F^{\frac{1}{r},\text{ext}}_0}{\d t^\nu_0\d t^\beta_q}+\frac{\d F^{\frac{1}{r},\text{ext}}_0}{\d t^\alpha_p}\frac{\d^2 F^{\frac{1}{r},\text{ext}}_0}{\d t^{r-1}_0\d t^\beta_q},
\end{gather*}
which implies that relations~\eqref{eq:property3} and~\eqref{eq:property4} are true with $F^{\frac{1}{r},o}_0$ replaced by~$\tF$.

Finally, in~\cite[Lemma 3.8]{BCT_Closed_Extended}, we proved that $\<\tau^{-1}\tau^1(\tau^{r-1})^2\>^{\frac{1}{r},\text{ext}}_0=-\frac{1}{r}$ and, therefore, $\<\tau^1\sigma^2\>^{\tF}=1$.

In Section~\ref{subsection:open primaries,proof}, we proved that properties~\eqref{eq:property1} -- \eqref{eq:property5} are sufficient to reconstruct all primary open intersection numbers $\<\prod_{i=1}^l\tau^{\alpha_i}\sigma^k\>^{\frac{1}{r},o}_0$. Since the correlators $\<\prod_{i=1}^l\tau^{\alpha_i}_{d_i}\sigma^k\>^{\tF}$ also satisfy these properties, we can conclude that $\<\prod_{i=1}^l\tau^{\alpha_i}\sigma^k\>^{\frac{1}{r},o}_0=\<\prod_{i=1}^l\tau^{\alpha_i}\sigma^k\>^{\tF}$. Note that the dimension constraint in property~\eqref{eq:property1} implies that $k+l\ge 2$. Therefore, using topological recursion relations~\eqref{eq:property3} and~\eqref{eq:property4}, one can reconstruct all open intersection numbers starting from primary numbers. The same is true for the correlators~\eqref{eq:tF-correlators} and, thus, $F^{\frac{1}{r},o}_0=\tF$. Theorem~\ref{thm:open-closed} is proved.

\subsection{Open string and dilaton equations}

\begin{prop}\label{prop:open string and dilaton}
We have
\begin{align*}
&\<\tau^0_0\prod_{i=1}^l\tau^{\alpha_i}_{d_i}\sigma^k\>^{\frac{1}{r},o}_0=
\begin{cases}
\displaystyle \sum_{\substack{1\le i\le l\\d_i>0}}\<\tau_{d_i-1}^{\alpha_i}\prod_{j\ne i}\tau^{\alpha_j}_{d_j}\sigma^k\>^{\frac{1}{r},o}_0,&\text{if $l\ge 1$},\\
\delta_{k,1},&\text{if $l=0$},
\end{cases}\\
&\<\tau^0_1\prod_{i=1}^l\tau^{\alpha_i}_{d_i}\sigma^k\>^{\frac{1}{r},o}_0=(k+l-1)\<\prod_{i=1}^l\tau^{\alpha_i}_{d_i}\sigma^k\>^{\frac{1}{r},o}_0.
\end{align*}
\end{prop}
\begin{proof}
By Theorem~\ref{thm:open-closed}, the required equations follow from
\begin{align}
&\<\tau^{-1}_0\tau^0_0\prod_{i=1}^n\tau^{\beta_i}_{b_i}\>^{\frac{1}{r},\text{ext}}_0=
\begin{cases}
\displaystyle \sum_{\substack{1\le i\le n\\b_i>0}}\<\tau^{-1}_0\tau_{b_i-1}^{\beta_i}\prod_{j\ne i}\tau^{\beta_j}_{b_j}\>^{\frac{1}{r},\text{ext}}_0,&\text{if $n\ge 2$},\\
\delta_{\beta_1,r-1}\delta_{b_1,0},&\text{if $n=1$},
\end{cases}\label{eq:string for extended}\\
&\<\tau^{-1}_0\tau^0_1\prod_{i=1}^n\tau^{\beta_i}_{b_i}\>^{\frac{1}{r},\text{ext}}_0=(n-1)\<\tau^{-1}_0\prod_{i=1}^n\tau^{\beta_i}_{b_i}\>^{\frac{1}{r},\text{ext}}_0,\label{eq:dilaton for extended}
\end{align}
where $0\le\beta_i\le r-1$ and $b_i\ge 0$. Equation~\eqref{eq:string for extended} is proved in~\cite[Lemma~3.7]{BCT_Closed_Extended}. Equation~\eqref{eq:dilaton for extended} is proved similarly using~\cite[Lemma~3.5]{BCT_Closed_Extended} and mimicking the proof of the ordinary dilaton equation on $\oCM_{0,n}$.
\end{proof}

\section{Constructions of boundary conditions, independence of choices}\label{sec:constructions_bc}

This section is devoted to constructing the different types of multisections we use throughout the article, and to proving Theorem \ref{thm:int_numbers_well_defined}.   %First, we discuss the construction of canonical boundary conditions.

\subsection{Constructions of a single global canonical multisection}
%The following definition is required:
\begin{definition}\label{def:strongly_pos}
Let $\Gamma$ be a graded graph.
A canonical multisection $s$ is \emph{strongly positive} if, for any $\Lambda\in\partial^!\Gamma\setminus\partial^+\Gamma$ such that $\CB\Lambda$ has some connected components $\Xi_1,\ldots, \Xi_a$ without boundary tails, and every $C\in\CM_\Lambda^{1/r}$, each local branch of the projection of $s_C$ to a $\Xi_i$-component evaluates positively on some nonempty subset of the boundary $\partial(\CB\Sigma^{\Xi_i})$ of the $\Xi_i$-component of $\CB\Sigma$ (defined in Remark \ref{rmk:F_in_surface_level}).
\end{definition}

The first step towards constructing canonical boundary conditions was the local Proposition \ref{prop:pointwise_positivity}. The second step is the following global proposition.
\begin{prop}\label{prop:a_single_section_for_Witten_or_L_i}Let $\Gamma$ be a smooth, graded connected graph.
\begin{enumerate}
\item\label{it:single_Witten}
There exists a transverse special canonical multisection of $\cW_\Gamma$. %Witten's bundle.
\item\label{it:single_Witten_strong_pos}
There exists a transverse special canonical strongly positive multisection of $\cW_\Gamma.$ Moreover, for any such multisection $s$ one can find a strongly positive multisection $\hat{s}$ such that, for any $C\in\oPM_\Gamma$ and for any local branch $\hat{s}^i_C$ of $\hat{s}_C$ and any local branch $s_C^j$ of $s_C$, the subset of $\partial\Sigma \subseteq C$ on which $\hat{s}^i_C$ is positive intersects the subset of $\partial\Sigma$ on which $s_C^j$ is positive.
\item\label{it:single_Witten-span}
For any $p\in\oPM_\Gamma$, there exist $s_1,\ldots, s_N\in C_m^\infty(\oPM_\Gamma,\cW)$ with compact support such that for any $\Lambda\in\partial^0\Gamma$, the section $s_i|_{\oPM_\Lambda}$ is pulled back from $\oPM_{\CB\Lambda}$, and for any choice of local branches $i_1,\ldots, i_N$ of $s_1(p),\ldots,s_N(p)$ respectively, the set of vectors $\{s_j^{i_j}(p)\}_{j=1,\ldots, N}$ spans $\cW_p.$
\item\label{it:single_L_i}
For any $C\in\partial\oCM^{\frac{1}{r}}_{0,k,\vec{a}}$ such that $z_i$ does not belong to a partially stable component of $\CB C,$ one can find $s\in \CS_i$ that does not vanish at $C$.

\end{enumerate}
\end{prop}

The proof %of this proposition 
relies on three lemmas.  The first one is a transversality claim:

\begin{lemma}\label{obs:trans_for_Ass}
Let $\Lambda$ be a graded graph (open or closed) without boundary edges, and let
$\{s^v\}_{v\in \Conn(\detach(\Lambda))}$ be multisections such that
\begin{enumerate}
\item if $v$ is closed and has an anchor of twist $-1$, then $s^v\in C_m^\infty(\TRAM_v,\cW)$ is a coherent multisection;
\item if $v$ is closed without an anchor of twist $-1$, then $s^v\in C_m^\infty(\oCM_v^{1/r},\cW)$;
\item if $v$ is open, then $s^v\in C_m^\infty(\oPM_v,\cW)$;
\item $\bar{s}^v\pitchfork 0$ (recall that $\overline{s}^v = s^v$ unless the anchor of $v$ has twist $-1$).
%where, we recall, if the anchor of $v$ does not have twist $-1$, then $\overline{s}^v = s^v$).
\end{enumerate}
Then also $\bar{s}\pitchfork 0$ for $s=\Ass_{\Lambda,E(\Lambda)}((s^v)_{v\in \Conn(\detach(\Lambda))})$.
\end{lemma}
\begin{proof}
By induction, we may restrict to the case where $\Lambda$ has a single edge $e$ and $\Conn(\detach(\Lambda))=\{v_1,v_2\}$. If $v_1$ and $v_2$ are not connected by a Ramond internal node, then the Witten bundle on $\oCM_\Lambda$ decomposes as a direct sum, so if each $s^{v_i}$ is transverse to zero, then so is $s$.

Suppose, then, that $\Lambda$ is obtained by gluing the anchor of $v_2$ to a Ramond internal tail of $v_1.$ Let $u_i\in\oCM_{v_i}^{1/r}$ be a zero of $\bar{s}^{v_i}$. Then $\Detach_e^{-1}(u_1,u_2)$ is a zero of $\bar{s}=\overline{\Ass_{\Lambda,e}(s^{v_1},s^{v_2})}.$
By the transversality assumption, the derivatives of (each local branch of) $\bar{s}^{v_i}(u_i)$ span $\cW_{u_i}.$ Thus the derivatives of the branches of $\overline{\Ass_{\Lambda,e}(s^{v_1},s^{v_2})}(\Detach_e^{-1}(u_1,u_2))$ in the $\CM_{v_i}^{1/r}$ directions span subspaces of $W_i$ of $\cW_{\Detach_e^{-1}(u_1,u_2)}.$
By~\eqref{eq:decompses2} of Proposition \ref{pr:decomposition}, the sequence \begin{equation}\label{eq:for_trans}0\to \cW_{u_2}\to\cW_{\Detach_e^{-1}(u_1,u_2)}\to \cW_{u_1}\to 0\end{equation} is exact.
The space $W_2$ is just the image, under the left map of \eqref{eq:for_trans}, of $\cW_{u_2}$ in $\cW_{u_1,u_2},$ while $W_1$ surjects to $\cW_{u_1}$ by the right map of \eqref{eq:for_trans}. Thus, at $\Detach_e^{-1}(u_1,u_2),$ the derivatives of the branches of $\overline{\Ass_{\Lambda,e}(s^{v_1},s^{v_2})}$ span $\cW_{\Detach_e^{-1}(u_1,u_2)},$ and the vanishing at $\Detach_e^{-1}(u_1,u_2)$ is transverse.  This completes the proof of the lemma.
\end{proof}

The next lemmas are extension results that allow inductive construction of mutlisections with the properties specified in Proposition~\ref{prop:a_single_section_for_Witten_or_L_i}~ \eqref{it:single_Witten}, \eqref{it:single_Witten_strong_pos}, \eqref{it:single_Witten-span}.

\begin{lemma}\label{lem:a single section for Witten-closed step}
Let $\Gamma$ be a smooth, connected, closed, graded $r$-spin graph.  Suppose that
$s^v\in C_m^\infty(\TRAM_v,\cW)$
are coherent multisections for each $v\in \mathcal{V}(\Gamma)\setminus\Gamma,$ such that
\begin{enumerate}
\item\label{it:1_for_closed_ext} $\bar{s}^v\pitchfork 0$;
\item\label{it:2_for_closed_ext} for each $\Lambda\in\partial v,$
$s^v|_{\TRAM_\Lambda}=\Ass_{\Lambda,E(\Lambda)}((s^\Xi)_{\Xi\in \Conn(\detach(\Lambda))});$
\item\label{it:3_for_closed_ext} 
$s^v$ evaluates positively at the anchor,
if the anchor of $v$ has twist $r-1$. %then $s^v$ evaluates positively %(with respect to the grading) at the anchor.
\end{enumerate}
Then one may construct a multisection $s^\Gamma\in C_m^\infty(\TRAM_\Gamma,\cW)$ that satisfies the above properties with $v$ replaced by $\Gamma$.
\end{lemma}

\begin{lemma}\label{lemma:extension}
Let $\Gamma$ be a smooth, connected, open, graded $r$-spin dual graph, and let $\zeta\in C_m^\infty(\bigcup_{\Lambda\in\partial\Gamma\setminus\partial^+\Gamma}\CM_\Lambda^{1/r},\cW).$
\begin{enumerate}
\item\label{it:extension_general}
Suppose that for any $\Lambda\in\partial\Gamma\setminus\partial^+\Gamma$, it holds that 
$\zeta|_{\oPM_\Lambda}=F_\Lambda^*\zeta^{\CB\Lambda},$
for some $\zeta^{\CB\Lambda}\in C_m^\infty(\oPM_{\CB\Lambda},\cW)$, which is positive with respect to $\CB\Lambda$, $\Aut(\CB\Gamma)$-invariant, and transverse.
Then one can extend $\zeta$ to a positive, transverse, $\Aut(\Gamma)$-invariant
multisection $\sigma\in C_m^\infty(\oPM_\Gamma,\cW).$
\item\label{it:extension_pos}
Suppose we make the additional assumption on $\zeta$ that for any $\Lambda\in\partial\Gamma\setminus\partial^+\Gamma$ such that $\CB\Lambda$ has at least one vertex without boundary tails, $\zeta|_{\CM_\Lambda^{1/r}}$ is strongly positive in the sense of Definition \ref{def:strongly_pos}. Then the multisection $\sigma$ of item \eqref{it:extension_general} can be chosen to be strongly positive on $\oPM_\Gamma.$
\item\label{it:extension_pos_moreover}
Let $s$ be a strongly positive multisection on $\oPM_\Gamma$, and assume, in addition to the assumptions of item \eqref{it:extension_general}, that for any $\Lambda\in\partial\Gamma\setminus\partial^+\Gamma$ such that $\CB\Lambda$ has at least one vertex without boundary tails, and for any $C\in{\CM_\Lambda^{1/r}}$ and any branch of $\zeta_C$, the subset of $\partial\Sigma \subseteq C$ on which $\zeta_{C}$ evaluates positively intersects all of the subsets of $\partial\Sigma$ on which the branches of $s_C$ are positive. Then the multisection $\sigma$ of item \eqref{it:extension_general} satisfies the following further condition: if $\Gamma$ has no boundary tails, then for any $C\in{\oPM_\Gamma}$ and any branch of $\sigma_C,$ the subset of $\partial\Sigma$ on which $\sigma_{C}$ evaluates positively intersects all subsets of $\partial\Sigma$ on which the branches of $s_C$ are positive.
\end{enumerate}
\end{lemma}

The proofs of these lemmas %Lemmas~\ref{lem:a single section for Witten-closed step} and \ref{lemma:extension} 
are quite involved, so we defer them until later.  %First, we show how to deduce Proposition~\ref{prop:a_single_section_for_Witten_or_L_i} from these lemmas.
\begin{proof}[Proof of Proposition~\ref{prop:a_single_section_for_Witten_or_L_i}]
To prove item \eqref{it:single_Witten} of the proposition, we first note that, by Lemma \ref{lem:a single section for Witten-closed step} and induction on $\dim(\CM_v^{1/r})$, there exists a family of multisections $s^v$ for each closed $v\in\mathcal{V}(\Gamma)$
satisfying the conditions of the proposition.  What remains, then, is to consider open vertices.

We shall construct, for any abstract open vertex $v\in \mathcal{V}(\Gamma)$, a transverse positive multisection $s^v\in C_m^\infty(\oPM_v,\cW)$ such that
for any $\Lambda\in\partial v\setminus\partial^+ v,$
\begin{equation}\label{eq:s^v_ind}s^v|_{\oPM_\Lambda}=F_\Lambda^*\left(\Ass_{\CB\Lambda,E(\CB\Lambda)}((s^f)_{f\in \Conn(\detach (\CB\Lambda))})\right).\end{equation}
The resulting $s^\Gamma$ will satisfy the requirements of item \eqref{it:single_Witten} by its definition.

The construction of $s^v$ is by induction on $\dim(\oCM_v)$.  The base case is when $v$ is partially stable, and this case is settled by the first item of Proposition \ref{prop:pointwise_positivity}.
Suppose, then, that we have constructed multisections $s^v$ for any $v\in \mathcal{V}(\Gamma)$ with $\dim\oCM_v<n\leq \dim\oCM_\Gamma$, and consider $v\in \mathcal{V}(\Gamma)$ with $\dim\oCM_v=n$.  We define a multisection $\zeta$ on $\coprod_{\Lambda\in \partial v\setminus\partial^+ v}\oPM_\Lambda$ by \eqref{eq:s^v_ind}. By Observation \ref{obs:for_comp}, the multisection $\zeta$ descends to $\bigcup_{\Lambda\in \partial v\setminus\partial^+ v}\oPM_\Lambda$ as a continuous multisection that is smooth on each $\oPM_\Lambda$.  This section is automatically $\Aut(v)$-invariant, and it is transverse on each $\oPM_\Lambda,$ by Lemma \ref{obs:trans_for_Ass}.  Moreover, $\zeta$ satisfies the assumptions of item \eqref{it:extension_general} of Lemma \ref{lemma:extension} by construction, so we can use the lemma to extend $\zeta$ to a smooth, $\text{Aut}(v)$-invariant and transverse multisection $s^v$ over all $\oPM_v$.   The inductive step for the construction of $s^v$ follows.

We now turn to item \eqref{it:single_Witten_strong_pos}.  The proof of the first part is by induction, similarly to the one performed in the previous case, by using the fact that the boundary conditions in the partially stable case (which is the base case of the induction) trivially satisfy the strong positivity requirement, as they evaluate positively on the whole boundary of the partially stable components.  We then apply item \eqref{it:extension_pos} of Lemma \ref{lemma:extension} for the inductive step, and the first part is proved.

The ``moreover" part of item \eqref{it:single_Witten_strong_pos} is proven similarly. This time, we start with a strongly positive multisection $s$  and use induction.  Again the induction basis is the partially stable case, and it satisfies the requirements by the same reasoning as above. The induction step follows by applying Lemma \ref{lemma:extension},~\eqref{it:extension_pos_moreover}.

For item \eqref{it:single_Witten-span} of the proposition, first note that the claim cannot be achieved using only multisections that satisfy the second property of special canonical multisections in Definition \ref{def:special canonical for Witten}, since they are too symmetric: if $C$ is a graded surface such that $\Aut(\CB C)$ has an element that permutes two components $D_1$ and $D_2$ of $\CB C$, then for any canonical multisection there is some branch with the same projections onto $\cW_{D_1}$ and $\cW_{D_2}$ (with respect to the natural identification). The second difficulty in working with multisections that have the second property of special canonical multisections is that, because of their strong decomposition properties, it is difficult to construct them with compact support in $\oPMr$, unless they vanish on the closed moduli space.

To overcome these issues, we construct multisections as follows. We assume that the internal labels are all different. Let $r(\Lambda)$ be the maximal possible number of vertices with $r+1$ boundary labels and no internal labels that may appear in any $\Xi'=\detach(\Xi)$ for any $\Xi\in\partial^!\Lambda.$ Note that $r(\Lambda)=r(\detach(\Lambda))$ and $r(\Lambda)\geq r(\Xi)$ for any $\Xi\in\partial^!\Lambda.$

Fix $C\in\oPM_\Gamma$, and suppose that $C\in\CM_\Lambda^{1/r}.$ Note that we have $\cW_\Sigma\simeq \boxplus_{\Xi\in\Conn\CB\Lambda}\cW_{\Sigma^{\Xi}},$ where $C^\Xi$ is the $\Xi$-component of $\CB C$. The construction of multisections now splits into two cases.

The first case is when $\Aut(\CB C)$ preserves $C^\Xi,$ which happens, in particular, if $\Xi$ has internal tails.  In this case, all elements of $\Aut(C^\Xi)$ are induced by permuting the boundary markings\footnote{Such permutations induce automorphisms of $C^\Xi$ only when the open vertex $\Xi^o$ of $\detach(\Xi)$ has exactly one internal half edge, and for very specific configurations of markings.}, though there are additional automorphisms of the bundle if $C^\Xi$ has internal nodes.  For an arbitrary $w\in\cW_{C^\Xi},$ let $G(w)=\{w_1=w,\ldots, w_a\}$ be its orbit under the action of $\Aut(C^\Xi)$, lifted to Witten's bundle.  We will show that one can construct a multisection as in the statement of item \eqref{it:single_Witten-span} whose branches at $C$ are $w_1,\ldots, w_a$ when projected to the $\Xi$ component, and are zero when projected to other components.

The second case is when $\Aut(\CB C)$ does not preserve $C^\Xi$. This means, in particular, that $\Xi$ has no internal half-edges or tails.  Let $\Xi_1,\ldots, \Xi_b$ be the components of the $\Aut(\CB C)$-orbit of $C^\Xi$.  For a vector $w\in \cW_{C^\Xi}$, let $G(w)$ be the orbit of $w$ under the action of $\Aut(C^\Xi)$, and for $w_1,\ldots, w_b\in\cW_{C^\Xi},$ let $G(w_1,\ldots, w_b)$ be the collection of vectors in $\widetilde{\cW}_{C^\Xi}:=\bigoplus_{i\in[b]}\cW_{C^{\Xi_i}}$ defined by 
\[
G(w_1,\ldots,w_b)=\left\{\bigoplus\nolimits_{i\in [b]} w'_{i}|w'_1\in G(w_1),\ldots,w'_b\in G(w_b)\right\}, 
\]
where we use the canonical identification between the different $\cW_{C^{\Xi_i}}.$
Fix a set $F\subset \N$ of size at least $r(\Gamma)$, and let $W_F\subset\cW_{C^\Xi}$ be a set of vectors $w_Q$ for each $Q\subset F$ such that $|Q|=r(\Xi).$ Let 
\[
G(W_F)=\bigcup_{Q_1,\ldots, Q_b\subset F,\,|Q_i|=r(\Xi)}G(w_{Q_1},\ldots,w_{Q_b}).
\]
We show how to construct a multisection $s$ as in the statement of item \eqref{it:single_Witten-span} whose branches at $C$ are precisely $G(W_F),$ under the natural injection $\widetilde{\cW}_{C^\Xi}\to\cW.$% as a direct summand.

Write $m=\rk\cW_{C^\Xi}$.  There is no difficulty, in the first case, in finding $w^1,\ldots, w^m$ such that each element of $G(w^1)\times\cdots\times G(w^m)$ is a basis for $\cW_{C^\Xi}$, and in the second case, in finding sets $W^1_F,\ldots, W^{mb}_F$ as above such that each element of $G(W^1_F)\times\cdots\times G(W^{mb}_F)$ is a basis for $\widetilde\cW_{C^\Xi}$.  Thus, constructing multisections as above for any $w$ (in the first case) or any $W_F$ (in the second case) will prove item \eqref{it:single_Witten-span}. Write $\mathcal{V}'(\Gamma)$ for the collection of graphs that appear as connected components of $\CB\Upsilon$ for graphs $\Upsilon\in\partial^0\Gamma$ with no internal edges or contracted boundary tails.  We define, for any $\Xi$ and $w$ as in the first case and any $\Upsilon\in\mathcal{V}'(\Gamma)$, a multisection $s^\Upsilon$.  Similarly, for any $\Xi$ and $W_F$ as in the second case, any $\Upsilon\in\mathcal{V}'(\Gamma)$, and any $Q\subseteq F$ of size at least $r(\Upsilon)$, we define a multisection $s_{Q}.$ We require these multisections to satisfy:
\begin{enumerate}
\item\label{it:1_for_extending_span}
For any $\Upsilon'\in\partial^0 \Upsilon$ with no internal edges or contracted boundary tails:
\[s^\Upsilon|_{\oPM_{\Upsilon'}} =
F_{\Upsilon'}^*\left(\boxplus_{\Omega\in\Conn(\CB\Upsilon')} s^\Omega\right),\]
in the first case, while in the second case
\[s_{Q}^\Upsilon|_{\oPM_{\Upsilon'}} = \biguplus_{\{(Q_\Omega)_{\Omega\in \Conn(\CB\Upsilon')}||Q_\Omega|=r(\Omega),\bigsqcup_\Omega Q_\Omega\subseteq Q\}}
F_{\Upsilon'}^*\left(\boxplus_{\Omega\in \Conn(\CB\Upsilon')}s^\Omega_{Q_\Omega}\right),\]
where $\uplus$ is the operation defined in Notation \ref{def:union_of_sections}.
\item\label{it:2_for_extending_span}
In the first case, let $\Upsilon$ be the smoothing of $\Xi$ (c.f. Definition \ref{def:smoothingraph}).   We take the branches of $s^\Upsilon(\Sigma^{\Xi})$ to be $\{w_1,\ldots,w_a\},$ with equal weights, in the sense of Definition \ref{def:multisection}. For any other component $\Xi'\in\Conn(\CB\Lambda)$ with a smoothing $\Upsilon',$ we put $s^{\Upsilon'}(C^{\Xi'})=0.$  In the second case, each $\Xi_i$ is smooth. We take the branches of $s^{\Xi_1}_{Q}(C^{\Xi_1}),~|Q|=r(\Xi_1),$ to be $G(w_Q)$ with equal weights\footnote{This implies that the branches of $s(C)$ are precisely $G(W_F)$ with equal weights.},
and for any other component $\Xi'\in\Conn(\CB\Lambda)$ with a smoothing $\Upsilon'$, we put $s^{\Upsilon'}_Q(C^{\Xi'})=0.$
\item\label{it:3_for_extending_span}
For every $\Upsilon\in\mathcal{V}'(\Gamma)$, both  $s^\Upsilon$ and $s_{Q}^\Upsilon$ have compact support.
\end{enumerate}

The construction of $s^\Upsilon$ and $s_{Q}^\Upsilon$ is by induction on $\dim\oCM_\Upsilon^{1/r}$.   In the second case, we first define $s_Q^\Upsilon$ for $Q\subset F$ with $|Q|=r(\Upsilon)$ and then put, for $|Q|>r(\Upsilon)$,
\[s_Q^\Upsilon = \biguplus_{Q'\subset Q,~|Q'|=r(\Upsilon)}s_{Q'}^\Upsilon.\] 
This definition is seen to have the required properties, using \eqref{eq:union_of_sections}.
%Using \eqref{eq:union_of_sections}, we see that this definition is consistent with the required properties of the multisections.

The base case of the inductive construction of $s^\Upsilon$ and $s_{Q}^\Upsilon$ is the partially stable case, where we have to choose one vector.  The construction from the proof of the previous items works here as well.
Suppose, then, that we have constructed multisections with the above three properties for all $\Omega\in\mathcal{V}'(\Gamma)$ with $\dim\oCM_\Omega^{1/r}<n$.
We define $s^\Upsilon|_{\partial^0\oPM_v}$ and $s^{\Upsilon}_{Q}|_{\partial^0\oPM_\Upsilon}$ by item \eqref{it:1_for_extending_span}. By Observation \ref{obs:for_comp}, and using \eqref{eq:union_of_sections} in the second case, we see that this multisection is indeed well-defined and compactly-supported on $\partial^0\oPM_v$.  Let $\Upsilon'$ be a graph identical to $\Upsilon$ but with injective labelings.  Represent the quotient $\tilde{q}:\oPM_{\Upsilon'}/G\simeq\oPM_{\Upsilon},$ where $G$ is a finite group, as in the last paragraph of Proposition \ref{prop:pointwise_positivity}. Pull back $s^{\Upsilon}|_{\partial^0\oPM_\Upsilon}$ and $s^\Upsilon_Q|_{\partial^0\oPM_\Upsilon}$ to $G$-invariant multisections $s^{\Upsilon'}|_{\partial^0\oPM_{\Upsilon'}}$ and $s^{\Upsilon'}_Q|_{\partial^0\oPM_{\Upsilon'}}.$ Extend these to multisections $\tilde{s}^{\Upsilon'}$ and $\tilde{s}^{\Upsilon'}_Q$ with compact support over $\oPM_{\Upsilon'}$.  Define
\[{s}^{\Upsilon'}=\biguplus_{g\in G}g\cdot \tilde{s}^{\Upsilon'},\;\;\;\;{s}^{\Upsilon'}_Q=\biguplus_{g\in G}g\cdot \tilde{s}^{\Upsilon'}_Q.\]
These multisections descend to the quotient $\oPM_{\Upsilon'}/G$, and we pull back using $\tilde{q}^{-1}$ to define ${s}^{\Upsilon'}$ and ${s}^{\Upsilon}_Q$ on all $\oPM_\Upsilon$, extending their previous definition over boundary.
%Extend $s^\Upsilon_{i,J}$ to all $\oPM_\Upsilon$ as a multisection with compact support, and symmetrize it over $\Aut(\Upsilon)$ (See Definition A.10 in Appendix A of \cite{PST14} for a symmetrization of a multisection). In the first case (second case),
If $\Upsilon$ is the smoothing of $\Xi$ (respectively, $\Xi_1$) we perform the extensions and symmetrizations of $s^\Upsilon$ (respectively, $s^\Upsilon_{Q}$) so that item \eqref{it:2_for_extending_span} also holds. %: there is no difficulty there. %This is obvious in the first case, while in the second case it is a consequence of the simple observation that almost every choice (in the measure theoretic sense of a.e., with respect to Lebesgue measure on $\cW_{\Sigma^{\Xi_i}}$ for example) of the branches at the points of $\Sigma^{\Xi_i}$ will satisfy the general position requirement.
The induction follows.

The resulting multisections, defined on $\oPM_\Gamma$ by pulling back from the space $\oPM_{\text{for}_{\text{marking}}(\Gamma)}$, satisfy the required properties, so item \eqref{it:single_Witten-span} of the proposition is proven.

The last item of the proposition, item \eqref{it:single_L_i}, is proven completely analogously to Proposition~3.49 of \cite{PST14}. The only comment is that, during the recursive process, some initial conditions involve defining the sections on partially stable components, an issue that did not appear in \cite{PST14}.  On such components, we set the section to be zero. This has no effect since the section of Witten's bundle has a prescribed, non vanishing behavior on such components. %There is no difficulty in extending the multisection in a non vanishing way away from strata which contain such components.
%
%This completes the proof of Proposition \ref{prop:a_single_section_for_Witten_or_L_i}, modulo the two lemmas.
\end{proof}

We now return to the proofs of Lemmas~\ref{lem:a single section for Witten-closed step} and \ref{lemma:extension}.

\begin{proof}[Proof of Lemma \ref{lem:a single section for Witten-closed step}]
The claim is trivial when the dimension of $\CM_\Gamma$ is zero or when the rank of the Witten bundle is zero, so we assume that both are positive.
Write $N=\bigcup_{\Lambda\in\partial\Gamma}\TRAM_\Lambda.$%~~N_0=\bigcup_{\Lambda\in\partial\Gamma}\oCM_\Lambda^{1/r}.\]
 Then $N$ is a normal crossings divisor in $\TRAM_\Gamma$, and on $N$, a multisection $s$ is defined by the second item: the second item defines a multisection on
$\coprod_{\Lambda\in\partial\Gamma}\TRAM_\Lambda$, which descends to $N$ by Observation \ref{obs:comp_of_Ass}.

Suppose first that the anchor does not have twist $r-1$.  By the first item and Lemma \ref{obs:trans_for_Ass}, the multisection $\bar{s}|_{\oCM_\Lambda^{1/r}}$ is transverse to zero for any
$\Lambda\in\partial\Gamma$.  We can thus extend it to a transverse multisection $\bar{s}\in C_m^\infty(\oCM_\Gamma^{1/r},\cW).$ %If $\lambda=0$ or the anchor is not Ramond, require
%the extension to satisfy the third item as well. %(there is no difficulty there).
This gives a multisection $s\in C_m^\infty(N',\cW)$ where $N'=N\cup\oCM_\Gamma^{1/r}\hookrightarrow\TRAM_\Gamma,$ is just $\oCM_\Gamma^{1/r}$ if the anchor is not Ramond, and a
normal crossing divisor in $\TRAM_\Gamma$ otherwise. If the anchor is not twisted $-1$ then $s$ satisfies the requirements of the lemma, and we are done. %and if %first two items, also the third, unless the anchor is Ramond and
%$\lambda\neq 0.$ If
%the anchor is not twisted $-1$ we are done.
Otherwise recall the coherent section $u$ defined in Example \ref{ex:simple_coherent}. The multisection $s|_{N'}-u|_{N'}$ is a multisection of $\cW'\to N'.$ Extend it to a smooth multisection $s'\in C_m^\infty(\TRAM_\Gamma,\cW).$ %such that $s'_{\Sigma,\lambda}=\frac{1}{m}\sum(\xi_{v_i-u_{\Sigma,\lambda}}),$ in case $\lambda\neq 0.$
Then $s'+u$ extends $s$ to a coherent multisection of $\TRAM_\Gamma$ which satisfies all the requirements.

If the anchor has twist $r-1$, then the multisection $s$, currently defined only on $N$, evaluates positively at the anchor over any point of $N$.  Extend $s$ to a multisection $s_0$ in neighborhood $U$ of $N$, small enough so that $s$ still evaluates positively at the anchor over any point of $U$. For any point $p\in \oCM_v^{1/r}$, the first part of Proposition \ref{prop:pointwise_positivity} allows one to construct a multisection $s_p$ that evaluates positively at the anchor over any point in a neighborhood $U_p$ of $p$. Cover $\oCM_v^{1/r}\setminus U$ by a finite number of such neighborhoods $U_{p_1},\ldots U_{p_m}$ that  do not intersect $N$.  Let $s_{i}$ be the multisection that corresponds to $p_i$, and let $\rho_0,\ldots, \rho_m$ be a partition of unity subordinate to the cover $U,U_{p_1},\ldots, U_{p_m}$.  Then $\sum_{i=0}^m \rho_i s_i$ extends $s$ and satisfies the requirements of the lemma.
\end{proof}

%\footnote{for proof-read: look for email named "harchavat hatach hiuvi" (in Hebrew, from 11.3)}
\begin{proof}[Proof of Lemma \ref{lemma:extension}]
%construct neighbrhoods for guys without forgettable or internal nodes.

Write $W=\bigcup_{\Lambda\in\partial\Gamma\setminus\partial^+\Gamma}\oPM_\Lambda$.

\noindent \textbf{Step $1$:}
Let $\Xi$ be a graph containing some internal edges or illegal boundary half-edges of twist zero, and let $u\in\CM_\Xi^{1/r}\subset\partial^+\oCM_\Gamma^{1/r}$ be a moduli point with corresponding $r$-spin surface $C_u$.  The first step is to construct, for each such $\Xi$ and $u$, a $\Xi$-neighborhood $U_u$ in $\oCM_\Gamma^{1/r}$, a $\Xi$-family of intervals $\{I_{u,h}\}_{h\in\Pos(\Xi)}$, and a $(U_u,\{I_{u,h}\}_h)$-positive multisection $s_u$ that restricts to $\zeta$ on $W\cap U_u$. (The na\"ive approach to the lemma would be to use standard arguments to show that one can extend the multisection to a neighborhood of $W$. However, due to noncompactness, this does not work directly; it only shows an extension to a set of the form $U\cap\oPM_\Gamma$, where $U$ is an open set containing $u$ in its closure.  This is why we first construct $(U_u,I_{u,h})$ as above, and then we explicitly construct a $(U_u,I_{u,h})$-positive extension of $\zeta$ to all $U_u$.)

%Let $\hat\Xi$ be the graph obtained by smoothing the internal edges of $\Xi$ and the edges which have an illegal half edge of twist $0.$
To construct $U_u$, $I_{u,h}$, and $s_u$, let $\{e_1, \ldots, e_L\}$ be the set of edges of $\Xi$ that are either internal or such that their illegal half-edge has twist zero.  In the restriction of the universal curve to a small contractible $\Xi$-neighborhood $U'_u$ of $u$, one can number the nodes corresponding to $e_1,\ldots,e_L$ by $n_1,\ldots, n_L$.  Write $\Gamma^{(Q)}$ for the graph obtained by smoothing all edges not numbered $1,\ldots ,L$, as well as all edges of $Q,$ for $Q\subseteq [L].$  Using the assumptions of the lemma, one can write
$\zeta^{(Q)}=\zeta|_{\oPM_{\Gamma^{(Q)}}\cap U'_u}$
for any $Q\subsetneq [L]$.   Recalling notation \ref{nn:fGamma}, define $\Xi^{(Q)}\in\CB\Gamma^{(Q)}$ by $\Xi^{(Q)}=f_{\Gamma^{(Q)}}(\Xi).$
%For any $J\subseteq [L]$ let $\Xi^{(J)}$ be the graph obtained by smoothing the edges of $J,$ so that $\Xi = \Xi^{\emptyset},\hat\Xi=\Xi^{([L])}.$ Write $\Gamma^{(J)}$ for the graph obtained by smoothing all edges not numbered $1,\ldots ,L$ together with the edges of $J.$
%For any $J\subset [L]$  (strict inclusion!) write $\zeta^{(J)}=\zeta|_{\oPM_{\Gamma^{(J)}}\cap U'_u}.$

By the assumptions on $\zeta$, for each $Q\subsetneq [L]$ one has $\zeta^{(Q)}=F_{\Gamma^{(Q)}}^*\zeta^{\CB(Q)}$, where $\zeta^{\CB(Q)}$ is a restriction of a positive multisection of $\oPM_{\CB\Gamma^{(Q)}}.$ Thus, for any $Q\subsetneq [L]$, there is a $\Xi^{(Q)}$-neighborhood $U^{\CB(Q)}\subseteq\oCM^{1/r}_{\CB\Gamma^{(Q)}}$ of $F_{\Gamma^{(Q)}}(u)$ and a family of intervals $\{I^{\CB(Q)}_h\}_{h\in\Pos(\Xi^{(Q)})}$ for $U^{\CB(Q)}$ on which $\zeta^{\CB(Q)}$ evaluates positively. Pulling back to $\oCM_{\Gamma^{(Q)}}^{1/r}$, we can find a $\Xi$-neighborhood $U^{(Q)}$ of $u$ inside $F_{\Gamma^{Q}}^{-1}(U^{\CB(Q)})\subseteq \oCM_{\Gamma^{(Q)}}^{1/r},$ and a family of intervals \[\{{I}^{(Q)}_{h'}(u')=\NNN(\phi_{\Gamma^{(Q)}}^{-1}(I_{h'}^{\CB(Q)}(F_{\Gamma^{(Q)}}(u'))))\}_{h'\in\Pos(\Xi^{(Q)})},\] where $\phi_{\Gamma^{(Q)}}$ is the map from Remark \ref{rmk:F_in_surface_level} and $\NNN$ is the normalization map. For any $u'\in U^{(Q)}$ and any $h$, the set ${I}^{(Q)}_h(u')$ is either an interval or a union of intervals inside $\Sigma_{u'}$; see Figure \ref{fig:for_ext}.  The endpoints of these intervals vary smoothly 
%with respect to the smooth structure 
on the universal curve restricted to $U^{(Q)}$.  By replacing with a smaller $\Xi$-neighborhood $U'_u$ of $u$, we may assume that $U'_u\cap\oCM_{\Gamma^{(Q)}}\subset U^{(Q)}.$

The simple yet crucial observation is that if $h\in\Pos(\Xi)$ is mapped by $f_{\Gamma^{(Q)}}$ to $h^{(Q)}:=f_{\Gamma^{(Q)}} (h)\in \Pos(\Xi^{(Q)})$, then for any $\Sigma$ in $U^{(Q)}\cap\oCM_{\Xi^{(Q)}}^{1/r}$ the node $n_h=n_h(C)$ belongs to an interval contained in ${I}^{(Q)}_{h^{(Q)}}(C).$ In Figure \ref{fig:for_ext}, both cases, when $f_{\Gamma^{(Q)}}$ is bijective or not, are illustrated.

\begin{figure}[t]
\centering

\begin{subfigure}{.3\linewidth}
\begin{tikzpicture}[scale=0.85]

\tkzDefPoint(0,0){O}
\tkzDefPoint(1,0){A}\tkzDefPoint(0,1){B}\tkzDefPoint(0,-1){C}
\tkzDefPoint(2,1){A'}\tkzDefPoint(2,-1){A''}\tkzDefPoint(2,0){OA}
\tkzDefPoint(-1,2){B'}\tkzDefPoint(1,2){B''}\tkzDefPoint(0,2){OB}
\tkzDefPoint(-0.707,-0.707+2){BL}\tkzDefPoint(0.707,-0.707+2){BR}
\tkzDefPoint(-1,-2){C'}\tkzDefPoint(1,-2){C''}\tkzDefPoint(0,-2){OC}
\tkzDefPoint(-0.707,0.707-2){CL}\tkzDefPoint(0.707,0.707-2){CR}
\tkzDefPoint(0.707,0.707){D}\tkzDefPoint(-0.707,0.707){E}
\tkzDefPoint(0.707,-0.707){D'}\tkzDefPoint(-0.707,-0.707){E'}

\tkzDefPoint(-0.857,0.557){F}\tkzDefPoint(-1.457,1.457){F'}\tkzDefPoint(-1.957,0.457){F''}
\tkzCircumCenter(F,F',F'')\tkzGetPoint{OF}
\tkzDrawArc(OF,F'')(F')

\tkzDefPoint(-0.857,-0.557){G}\tkzDefPoint(-1.457,-1.457){G'}\tkzDefPoint(-1.957,-0.457){G''}
\tkzCircumCenter(G,G',G'')\tkzGetPoint{OG}
\tkzDrawArc(OG,G')(G'')

\draw (0,0) circle [line width = 0.5pt, radius=1];
\tkzDrawArc(OA,A')(A'')
\tkzDrawArc(OB,B')(B'')
\tkzDrawArc(OC,C'')(C')
\tkzDrawArc[line width = 3pt](O,D)(E);
\tkzDrawArc[line width = 3pt](O,E')(D');
\tkzDrawArc[line width = 3pt](OB,BL)(BR);
\tkzDrawArc[line width = 3pt](OC,CR)(CL);

\node at (0,0.75) {$n_{h'}$};
\node at (0,1.25) {$n_{h}$};
\node at (0,-0.75) {$n_{f}$};
\node at (0,-1.25) {$n_{f'}$};
\node at (-1.2,0.65){$n_3$};
\node at (-1.2,-0.65){$n_2$};
\node at (1.25,0){$n_1$};

\end{tikzpicture}
\end{subfigure}
\hspace{0.25cm}
\begin{subfigure}{0.6\linewidth}
\begin{tikzpicture}[scale=0.85]

\tkzDefPoint(0,0){O}
\tkzDefPoint(1,0){A}\tkzDefPoint(0,1){B}\tkzDefPoint(0,-1){C}
\tkzDefPoint(2,1){A'}\tkzDefPoint(2,-1){A''}\tkzDefPoint(2,0){OA}
\tkzDefPoint(-1,2){B'}\tkzDefPoint(1,2){B''}\tkzDefPoint(0,2){OB}
\tkzDefPoint(-0.707,-0.707+2){BL}\tkzDefPoint(0.707,-0.707+2){BR}
\tkzDefPoint(-1,-2){C'}\tkzDefPoint(1,-2){C''}\tkzDefPoint(0,-2){OC}
\tkzDefPoint(-0.707,0.707-2){CL}\tkzDefPoint(0.707,0.707-2){CR}
\tkzDefPoint(0.707,0.707){D}\tkzDefPoint(-0.707,0.707){E}
\tkzDefPoint(0.707,-0.707){D'}\tkzDefPoint(-0.707,-0.707){E'}

\tkzDefPoint(-0.857,0.557){F}\tkzDefPoint(-1.457,1.457){F'}\tkzDefPoint(-1.957,0.457){F''}
\tkzCircumCenter(F,F',F'')\tkzGetPoint{OF}
\tkzDrawArc(OF,F'')(F')

\tkzDefPoint(-0.857,-0.557){G}\tkzDefPoint(-1.457,-1.457){G'}\tkzDefPoint(-1.957,-0.457){G''}
\tkzCircumCenter(G,G',G'')\tkzGetPoint{OG}
\tkzDrawArc(OG,G')(G'')

\draw[gray, line width = 3pt] (0,0) circle [radius=1];
\tkzDrawArc(OA,A')(A'')
\tkzDrawArc(OB,B')(B'')
\tkzDrawArc(OC,C'')(C')
\tkzDrawArc[line width = 3pt](O,D)(E);
\tkzDrawArc[line width = 3pt](O,E')(D');
\tkzDrawArc[line width = 3pt](OB,BL)(BR);
\tkzDrawArc[line width = 3pt](OC,CR)(CL);

\node at (0,0.75) {$n_{h'}$};
\node at (0,1.25) {$n_{h}$};
\node at (0,-0.75) {$n_{f}$};
\node at (0,-1.25) {$n_{f'}$};
\node at (-1.2,0.65){$n_3$};
\node at (-1.2,-0.65){$n_2$};
\node at (1.25,0){$n_1$};

\tkzDefPoint(4,0){rO}
\tkzDefPoint(5,1){r1}\tkzDefPoint(3,1){r2}
\tkzCircumCenter(rO,r1,r2)\tkzGetPoint{uO}
\tkzDrawArc(uO,r2)(r1)
\tkzDefPoint(0.707+4,-0.707+1){r1'}\tkzDefPoint(-0.707+4,-0.707+1){r2'}
\tkzDrawArc[line width = 3pt](uO,r2')(r1')

\tkzDefPoint(5,-1){r3}\tkzDefPoint(3,-1){r4}
\tkzCircumCenter(rO,r3,r4)\tkzGetPoint{dO}
\tkzDrawArc(dO,r3)(r4)
\tkzDefPoint(0.707+4,0.707-1){r3'}\tkzDefPoint(-0.707+4,0.707-1){r4'}
\tkzDrawArc[line width = 3pt](dO,r3')(r4')

\node at (4,0.25) {$n_{h^{(Q)}}$};
\node at (4,-0.25) {$n_{h^{'(Q)}}$};

\end{tikzpicture}
\end{subfigure}

\caption{A nodal disk $C_u$ in which the half-nodes opposite $n_1$, $n_2$, and $n_3$ are forgotten when one passes to $\CB C_u$, and $h,h',f,f'\in\Pos$. On the left, we consider the case $1,2,3\in Q$, so that the intervals $I_h^{(Q)},I_f^{(Q)},I_{h'}^{(Q)}$, and $I_{f'}^{(Q)}$ (marked in bold) are pullbacks of intervals $I_h^{\CB(Q)},I_f^{\CB(Q)},I_{h'}^{\CB(Q)}$, and $I_{f'}^{\CB(Q)}$.  On the middle and right, we consider the case $1,2,3\notin Q$; the right side shows part of the nodal disk obtained by the base operation, in which $n_h$ and $n_f$ are mapped to the same $n_{h^{(Q)}}$ and $n_{h'}$ and $n_{f'}$ are mapped to the same $n_{h^{'(Q)}}$. The intervals $I_{h^{(Q)}}^{\CB(Q)}$ and $I_{h^{'(Q)}}^{\CB(Q)}$ (marked in bold on the right) are pulled back to $I_h^{(Q)}=I_f^{(Q)}$ and $I_{h'}^{(Q)}=I_{f'}^{(Q)}$, respectively, in the middle.}
\label{fig:for_ext}
\end{figure}

%COMMENTED BELOW: Longer version of caption
%
% \caption{\Emily{To be shortened.} The left and middle part of the figure show part of a nodal disk $C_u$, in which the half-nodes $n_1$, $n_2$, and $n_3$ are legal with twist $r-2$ (so the other half-nodes of these nodes are forgotten when one passes to $\CB C_u$), $h,h',f,f'\in\Pos$, and the disk component of $n_{h'}$ becomes unstable after forgetting the half-nodes. In the left part, we consider the case $1,2,3\in Q$, so that the intervals $I_h^{(Q)},I_f^{(Q)},I_{h'}^{(Q)}$, and $I_{f'}^{(Q)}$ (marked in bold) are pullbacks of intervals $I_h^{\CB(Q)},I_f^{\CB(Q)},I_{h'}^{\CB(Q)}$, and $I_{f'}^{\CB(Q)}$.  In the middle and right sides, we consider the case $1,2,3\notin Q$; the right side shows part of the nodal disk obtained by the base operation.  Now $n_h$ and $n_f$ are mapped to the same $n_{h^{(Q)}}$ on the right, and $n_{h'}$ and $n_{f'}$ are mapped to the same $n_{h^{'(Q)}}$. In addition, the intervals $I_{h^{(Q)}}^{\CB(Q)}$ and $I_{h^{'(Q)}}^{\CB(Q)}$ (marked in bold on the right) are pulled back to $I_h^{(Q)}=I_f^{(Q)}$ and $I_{h'}^{(Q)}=I_{f'}^{(Q)}$, respectively, in the middle, which are unions of boundary intervals containing intervals which surround $n_h,n_f,n_{h'},n_{f'}$.}

Thus, we can find a family of non-intersecting intervals $I_h(u)\subset\partial C_u$ containing $n_h(u)$ and no other special point, such that $I_h(u)\subseteq I_{h^{(Q)}}^{(Q)}(u)$ for all $h\in H^+(\Xi)$ and all $Q\subset L$.  We extend smoothly the family $\{I_h(u)\}_{h\in H^+(\Xi)}$ to a $\Xi$-family of intervals $\{I_{u,h}\}_{h\in\Pos(\Xi)}$ on $U'_u$.  By continuity of the endpoints of the intervals $I_{h}^{(Q)}$ with respect to the topology of the universal curve, there exists a contractible $\Xi$-neighborhood $U_u\subseteq U'_u$ of $u$ such that, for all $u' \in U_u\cap\oPM_{\Gamma^{(Q)}}\subseteq U_u\cap W$, we have $I_{u,h}(u')\subseteq I_{h^{(Q)}}^{(Q)}(u')$.

We now show that one can extend the given multisection $\zeta|_{W\cap U_u}$ to all $U_u\cap\oPM_\Gamma$.  By the definition of $I_{h,u}$ and $U_u$, the multisection $\zeta|_{W\cap U_u}$ evaluates positively at $\{I_{h,u}\}$.  We can extend $\zeta|_{W\cap U_u}$ to a multisection $s'$ on a neighborhood $V_u\subseteq U_u$ of $W\cap U_u$ such that $s'$ evaluates positively at each $I_{u,h}(u')$ for $u'\in V_u$ and $ h\in\Pos(\Xi)$. Let $V'_u$ be another neighborhood of $W\cap U_u$ whose closure is contained in $V_u.$
Let $\rho'_u$ be a smooth nonnegative function that equals $1$ on $V'_u$ and $0$ on $U_u\setminus V_u.$

Denote by $\hat\Xi$ the graph obtained by smoothing the internal edges of $\Xi$ and the edges that have an illegal half-edge of twist zero; then $\Pos(\Xi)=\Pos(\hat\Xi)$ canonically.  Let $\overline{W}$ be the closure of $W$ in $\M_{\Gamma}^{1/r}$, and let $U_{u,1},U_{u,2},\ldots\subset U_u\setminus \overline{W}$ be a locally finite cover of $U_u\setminus \overline{W}$ by contractible $\hat\Xi$-sets.

By Proposition \ref{prop:pointwise_positivity}, we can construct sections $s'_{u,i}\in C^\infty(\cW\to U_{u,i})$ that are positive with respect to $(U_{u,i},\{I_{u,h}\}_{h\in\Pos(\hat\Xi)}).$ We now take subsets $U'_{u,i}$ with $\overline{U'}_{u,i}\subset U_{u,i}$ that also cover $U_u\setminus \overline{W}$, and let $\rho_{u,i}$ be smooth nonnegative functions that equal $1$ on $U'_{u,i}$ and $0$ outside of $U_{u,i}$.  Then $s_u = \rho'_u s'+ \sum_{i\geq 1}\rho_{u,i}s'_{u,i}$ is defined on all $U_u$, is smooth and $(U_u,\{I_{u,h}\}_{h})$-positive, and restricts to $\zeta$ on $W\cap U_u$.  This completes the first step.

\vspace{0.25cm}

\noindent \textbf{Step $2$:}
Since $\bigcup_{\Lambda\in\partial\Gamma\setminus\partial^+\Gamma}\partial^+\oCM_\Lambda^{1/r}$ is compact, the previous step allows constructing tuples $(\Xi_j,V_{j},\{I_{j,h}\}_{h\in\Pos(\Xi_j)},s'_{j})$, with $j=1,\ldots, M$, where $V_{j}$ is a $\Xi_j$-set, $\bigcup_{\Lambda\in\partial\Gamma\setminus\partial^+\Gamma}\partial^+\oCM_\Lambda^{1/r}\subset \bigcup_{j=1}^M V_{j},$ and the smooth multisection $s'_{j}$ is $(V_{j},\{I_{j,h}\}_h)$-positive and agrees with $\zeta$ on $W\cap V_{j}.$ Using a partition of unity $\{\hat{\rho}_j\}_{j\in[M]}$ subordinate to $\{V_{j}\}_{j\in[M]}$, we construct a smooth multisection
$s_0 = \sum\hat{\rho}_j s'_{j}$ defined on $\hat{U}=\bigcup V_{j}$ that agrees with $\zeta$ on $\hat{U}\cap W.$

We claim that $s_0$ is positive. Indeed, consider $u\in \partial^+\oCM_\Gamma^{1/r}\cap \hat{U}$ with $u\in \CM_\Lambda^{1/r},$ where $\Lambda$ has at least one internal edge or illegal boundary half-edge of twist $0$.  Let $S_u$ be the set of indices $i$ with $u\in V_i$.  Then for any $i\in S_u$, the definition of a $\Xi_i$-set implies that $\Xi_i\in\partial^!\Lambda$.  Let $Z_u=\bigcap_{i\in S_u}V_{i},$ and for $u'\in Z_u$, write $I_h(u')=\bigcap_{i\in S_u}I_{i,h}(u').$ Note that for $h\in\Pos(\Lambda)$, the intervals $I_h(u)\neq\emptyset$, since each is the intersection of intervals that contain a specific node $n_h$.  Thus, these intervals are also nonempty in a neighborhood $Z'_u\subseteq Z_u$ of $u$. By the definition of $S_u$ and of $\Lambda$-sets, there is a neighborhood $Z^0_u$ of $u$ that does not meet $\text{Supp}(\rho_i)$ for any $i\notin S_u$.  Thus, on $Z'_u\cap Z^0_u$, each $s'_{i}$ that is defined evaluates positively at the intervals $I_h$, and hence so does $s_0=\sum\hat{\rho}_j s'_{j}$.

\vspace{0.25cm}

\noindent \textbf{Step $3$:}
We now construct, for any $\Lambda$ without internal edges or illegal boundary half-edges of twist zero and any point $u\in\CM_\Lambda^{1/r}\subseteq\partial^+\oCM_\Gamma^{1/r}$, a contractible $\Lambda$-neighborhood $U_u$, a $\Lambda$-family of intervals $\{I_{u,h}\}_{h\in\Pos(\Lambda)}$ for $U$, and a $(U_u,\{I_{u,h}\}_{h})$-positive section $s_u.$
This is straightforward: on the surface $\Sigma_u$, we draw intervals $I_{u,h}(u)$ that satisfy the requirements of a $\Lambda$-family of intervals at $u$.  We then extend these in a smooth way to $\Lambda$-family $I_{u,h}$ in a small enough contractible $\Lambda$-neighborhood $U_u$ of $u$ and use Proposition \ref{prop:pointwise_positivity} to construct sections $s_u\in C^\infty(U_u,I_u)$ that are $(U_u,I_u)$-positive.

By compactness of $\partial^+\oCM_\Gamma^{1/r}\setminus \hat{U}$ ,we can therefore find tuples $(\Lambda_i,U_i,I_i,s_i)$, for $i=1,\ldots ,N$, such that each $U_i$ is a contractible $\Lambda_i$-set for $\Lambda_i$ without internal edges or illegal half-edges of twist zero, each $I_i=\{I_{i,h}\}_{h\in \Pos(\Lambda_i)}$ is a $\Lambda_i$-family of intervals for $U_i$, each $s_i$ is $(U_i,I_i)$-positive, and
$\partial^+\oCM_\Gamma^{1/r}\setminus\hat{U}\subset\bigcup U_i$.

For $i=1,\ldots,N$, let $U'_i$ be an open set whose closure is contained in $U_i$, such that $\bigcup U'_i$ also covers $\partial^+\oCM_\Gamma^{1/r}\setminus \hat{U}$.  For each $i$, let $\rho_i$ be a smooth nonnegative function defined on $\hat{U}\cup\bigcup_{i=1}^n U_i$ that equals $1$ on $U'_i$ and equals $0$ outside $U_i$.  Let $\rho_0$ be a smooth nonnegative function such that $\hat U\setminus(\bigcup U'_i) \subset \text{Supp}(\rho) \subset \hat{U}$ and $\rho_0|_W=1$.  Then $\sigma'=\sum_{i=0}^N\rho_is_i$ is a multisection, defined on $W$ and on a neighborhood of $\oPM_\Gamma$, that extends $\zeta$.  In addition, $\sigma'$ is positive, as one can see by applying the same argument used for the positivity of $s_0$.

\vspace{0.25cm}

\noindent \textbf{Step $4$:}
The multisection $\sigma'$ is transverse on $W$ and positive near $\partial^+\oCM_\Gamma^{1/r}.$
Let $\widetilde{\Gamma}$ be a graph identical to $\Gamma$ but with injective labels, and let $\tilde{q}:\oCM_{\widetilde{\Gamma}}^{1/r}\to\oCM_\Gamma^{1/r}$ be the quotient map. Extend $\tilde{q}^*\sigma'$ transversally to all of $\oPM_{\widetilde{\Gamma}}$.  Symmetrize the extension of $\tilde{q}^*\sigma'$ with respect to the action of $\Aut(\Gamma)$ on $\cW\to\oPM_{\widetilde{\Gamma}}$, where symmetrization of a multisection is defined in \eqref{eq:average} in Appendix \ref{sec:ap_euler}. The result is a pullback $\tilde{q}^*\sigma$ of a multisection $\sigma$ of $\cW\to\oPM_{\Gamma}$.  Then $\sigma$ is the required global multisection, and item \eqref{it:extension_general} of the lemma is proved.

We now turn to items \eqref{it:extension_pos} and \eqref{it:extension_pos_moreover}. We need two observations.
\begin{obs}\label{obs:positive_section_k=0_smooth}
If $\Gamma$ is smooth and graded without boundary tails, then there exists $s_+\in C^\infty(\oPM_\Gamma,\cW)$ that vanishes on $\oPM_\Gamma\setminus\CM_\Gamma^{1/r}$ and, for any smooth $C\in\CM_\Gamma^{1/r}$, the section $(s_+)_C$ evaluates positively at $\partial\Sigma \subseteq C$.
\end{obs}

Indeed, for any $C\in\CM_\Gamma^{1/r}$, one can find a vector $v_C\in\cW_C,$ thought of as a global section of $|J|$ over $C$, that evaluates positively on $\d\Sigma \subseteq C$ by choosing all of the zeroes of this section to be internal, using Lemma \ref{lem:surj_of_eval}, and multiplying the chosen section by $-1$ if needed. %and is an especially simple instance of the proof method of Proposition \ref{prop:pointwise_positivity}; in particular, since there are no boundary markings, the fact that all zeroes are internal implies that the section has the same sign at each boundary point.  Thus, multiplying by $-1$ if required yields the desired $v_{C}$. 
We can extend $v_C$ to a section $s_{C}$ defined in a neighborhood $U_C$ of $C$ with the same properties. From here, a standard partition of unity argument yields a global section $s_+\in\Gamma(\CM_\Gamma^{1/r},\cW)$ with the required positivity.  We then obtain the desired section by replacing $s_+$ with $\rho s_+$, where $\rho$ is a function on $\oPM_\Gamma$ that is positive on $\CM_\Gamma^{1/r}$ and vanishes quickly enough when approaching $\oPM_\Gamma\setminus\CM_\Gamma^{1/r}$, and then extending $\rho s_+$ to $\oPM_\Gamma\setminus\CM_\Gamma^{1/r}$ by zero.  This verifies Observation \ref{obs:positive_section_k=0_smooth}.
\begin{obs}\label{obs:second} If $\CB\Gamma$ has a component with no boundary tails, then for any $\Xi\in\partial^0\Gamma$, there is a component of $\CB\Xi$ with no boundary tails.
\end{obs}

Returning to the lemma, in order to prove item \eqref{it:extension_pos}, we proceed as in the proof of the previous item, only that after the third step, if $\Gamma$ has no boundary tails, we extend $\zeta$ further to a neighborhood of $\oPM_\Gamma\setminus\CM_\Gamma^{1/r}$ in $\oPM_\Gamma$.  By Observation \ref{obs:second} and the strong positivity assumption, if we restrict to a small enough neighborhood $U_0$, then for any $C\in U_0$ the extended section (which we also denote by $\sigma'$ as in the proof of Lemma \ref{lemma:extension}) evaluates positively on a nonempty subset of $\partial\Sigma$.  Let $s_+$ be the section constructed in Observation \ref{obs:positive_section_k=0_smooth}.
Multiply $\sigma'$ by a smooth function $\rho:\oPM_\Gamma\to[0,1]$ that equals $0$ outside of $U_0$ and equals $1$ on $\partial\oCM_\Gamma^{1/r}\cap\oPM_\Gamma$, and extend the result to a global section $\sigma''$ that equals zero outside of $U_0$.  Then $\sigma=s_++\sigma''$ is the required extension.

The proof of prove item \eqref{it:extension_pos_moreover} is similar to the previous case, and we use the same notations. We extend $\zeta$ to $U_0$ as above, denoting also the result by $\zeta$. Again, by the assumptions and Observation \ref{obs:second}, for $U_0$ small enough, the requirement regarding the intersections between the positivity loci of the branches of $s_C$ and $\zeta_C$ holds, since the locus in $\partial\Sigma$ on which a given branch of a section is positive is a union of open intervals. Defining $\rho,\sigma'',s_+$, and $\sigma$ as above, we see that the resulting $\sigma$ satisfies the modified requirements.
%
%The proof of Lemma~\ref{lemma:extension} is thus complete.
\end{proof}

\subsection{Proofs of Lemma \ref{lem:existence} and Proposition \ref{prop:no_zeroes_for_Witten_without_bdry_markings}}\label{subsec:constructions_bc_cons}
%We now turn to the proofs of previously-stated results concerning the existence of certain canonical multisections.  We require the following theorem:
\begin{thm}\label{thm: hirsch}
Let $E\to M$ be an orbifold vector bundle over a smooth orbifold with corners, and let $V = \R_{+}^n.$ Fix smooth multisections $s_0,\ldots,s_n\in C_m^\infty(M,E)$, and let $F: V \to C_m^\infty(M,E)$ be the map
\[
(\lambda_{i})_{i\in [n]}\to F_\Lambda = s_0 + \sum \lambda_i s_i.
\]
Denote by $p_M:V\times M\to M$ the projection. If the multisection
\[
F^{ev}\in C^\infty_m(V\times M, p_M^*E), \qquad
F^{ev}\left(\lambda,x\right)= F_\lambda\left(x\right),
\]
is transverse to zero, then the set $\{v\in V | F_v \pitchfork 0\}$ is residual.
\end{thm}
This is the orbifold analogue of \cite[Theorem 3.52]{PST14}; the proof is identical.

\begin{proof}[Proof of Lemma \ref{lem:existence}]
Let $s$ be a special canonical multisection of $\cW$ with the properties of Proposition \ref{prop:a_single_section_for_Witten_or_L_i}, \eqref{it:single_Witten}.  Denote its components by $(s^v)_{v\in\mathcal{V}(\Gammar)}$.  Then $s$ is positive on a set $U_+$ of the form given in \eqref{eq:U_+}, hence is nonvanishing in a neighborhood $U'_+$ of $U_+$.  In addition, let
\begin{equation}
\label{eq:ijk}
w_{ijk} \in \CS_i, \;\;  1 \leq i \leq l, \;\; 1 \leq j \leq d_i, \;\; 1 \leq k \leq m_{ij},
\end{equation}
be a finite set of special canonical multisections of the $j$th copy of $\CL_i,$ which together span the fibers of $\CL_i$ at any point of the compact $\oPMr\setminus U'_+$. The existence of these multisections is guaranteed by Proposition \ref{prop:a_single_section_for_Witten_or_L_i},  \eqref{it:single_L_i}. Set $J=\{ijk\}_{i,j,k},\;\;V= (\R_{+})^J$, where $i$, $j$, and $k$ range over the values in \eqref{eq:ijk}.

For any $v\in \mathcal{V}(\Gammar)$ and any set $K$ as in the statement of the lemma, write
\[
J_{v,K} = \{abc \; | \; ab\in K, ~c\in [m_{ab}]\}\subseteq J.
\]
Apply Theorem~\ref{thm: hirsch} with
$
M = \CM_v^{1/r}, ~ E = \cW_v\oplus\bigoplus_{ab\in K}\CL_a\to\CM_v^{1/r},$
\[
V_{v,K} = \R_+^{J_{\Lambda,K}},~F = (F_{v,K})_{\lambda} = \bar{s^v}+\sum_{\alpha \in J_{\Lambda,K}} \lambda_{\alpha} w^v_\alpha, ~ \lambda =\{\lambda_{\alpha}\}_{\alpha \in J_{\Lambda,K}} \in V_{\Lambda,K},
\]
Observe that, in the notation of Theorem \ref{thm: hirsch}, $F_{v,K}^{\ev}\pitchfork0.$ Indeed, at any zero $(\lambda,u)\in V_{v,K}\times\CM_v^{1/r}$, the derivatives of $F_{v,K}^{\ev}$ in the $\CM_v^{1/r}$-directions span $\cW_u$ by the assumption on $s$, while those in the $V_{v,K}$-direction span $\bigoplus_{ab\in K}\CL_a$.

Set $
X_{v,K} = \{ \lambda \in V_{\Lambda,K}| (F_{\Lambda,K})_\lambda \pitchfork 0\}.
$
Then
$X_{v,K}$ is residual by Theorem~\ref{thm: hirsch}.  Moreover, if a$p_{v,K} : V \to V_{v,K}$ denotes the projection map, then $X = \bigcap_{v,K} p_{v,K}^{-1}(X_{v,K})$ is residual. Write $s_{ij}=s_{ij;\lambda} = \sum_{k\in[m_{ij}]}\lambda_{ijk}w_{ijk},$ and $\mathbf{s}=F_\lambda$
for some $\lambda\in X$.  Then for any abstract vertex $v\in\mathcal{V}(\Gammar)$ and any set $K$, we have
\[
\bar{s}^v\oplus\bigoplus_{ab\in K} s_{ab;\lambda}^v = (F_{v,K})_\lambda \pitchfork 0.
\]

If \eqref{eq:open_rank_with_psi} holds, then for any $\Lambda\in\partial^0\Gammar$ there is an abstract vertex $v$ with
\[\rk_\R (\cW_v)+2\sum_{i\in I(v)}d_i > \dim_\R(\CM_v^{1/r}).\]
By choosing $K = \bigcup_{i\in I(v)}{\{i\}\times[d_i]},$ transversality shows that $\bar{s}^v\oplus\bigoplus_{ab\in K} s_{ab;\lambda}^v$ does not vanish in $\CM_v^{1/r}$, so $\mathbf{s}$ does not vanish in $\CM_\Lambda^{1/r}.$
\end{proof}
%
%Proposition \ref{prop:no_zeroes_for_Witten_without_bdry_markings} is an immediate consequence of the following lemma.
%
\begin{lemma}\label{lem:exist_and_homotopy_inv_strongly_pos}
For any smooth graded open graph $\Gamma$, there exist global strongly positive multisections. Moreover, for any two such multisections $s_0$ and $s_1$, there is a homotopy $H\in C_m^\infty(\oPM_\Gamma\times[0,1],\cW)$ with $H(-,i)=s_i$ for $i=0,1$ such that every $s_t$ is strongly positive.
\end{lemma}
Lemma \ref{lem:exist_and_homotopy_inv_strongly_pos} implies Proposition \ref{prop:no_zeroes_for_Witten_without_bdry_markings}: when $k=0$, for every $\Lambda\in\partial^0\Gamma\cup\{\Gamma\},$ at least one open component of $\CB\Lambda$ has no boundary tails. The strong positivity, Definition \ref{def:strongly_pos}, assures that the multisection does not vanish on $\CM_\Lambda^{1/r}$.
%Thus, to finish the proof of Proposition \ref{prop:no_zeroes_for_Witten_without_bdry_markings}, we need only prove Lemma~\ref{lem:exist_and_homotopy_inv_strongly_pos}.
\begin{proof}[Proof of Lemma \ref{lem:exist_and_homotopy_inv_strongly_pos}]
The existence follows from Proposition \ref{prop:a_single_section_for_Witten_or_L_i}, Item \eqref{it:single_Witten_strong_pos}.
The ``moreover" part follows from the ``moreover" part of that item, applied to $s=s_0\uplus s_1$ ($\uplus$ is defined in Definition \ref{def:union_of_sections}). Denote the multisection defined in Item~\eqref{it:single_Witten_strong_pos} by $\hat{s},$ and let $\rho_\epsilon:[0,1]\to[0,1]$ satisfy $\text{Supp}(\rho_0)=[0,\frac{1}{3}],\text{Supp}(\rho_1)=[\frac{2}{3},1]$, and $\rho_\epsilon(\epsilon)=1.$ Then the following homotopy satisfies our requirements:
$H(-,t)=\rho_0(t)s_0+\rho_1(t)s_1+t(1-t)\hat{s}.$%This concludes the proof of Lemma~\ref{lem:exist_and_homotopy_inv_strongly_pos} and hence of Proposition \ref{prop:no_zeroes_for_Witten_without_bdry_markings}.
\end{proof}

\subsection{Proofs of Lemma \ref{lem:partial_homotopy} and Theorem \ref{thm:int_numbers_well_defined}}\label{subsec:homotopies}
%We start with the proof of Lemma \ref{lem:partial_homotopy}.
\begin{proof}[Proof of Lemma \ref{lem:partial_homotopy}]
Let $\mathbf{s}_1$ be the projection of $\mathbf{s}$ to $E_1$, and extend $\mathbf{s}$ and $\mathbf{r}$ to a neighborhood $Y$ of $V_+.$ By the proof of Observation \ref{obs:sum_of_canonical} and the definition of canonical multisections, there is an open subset $Y_+$ of $Y$ containing $V_+$,
such that $Y_+ = Y_0\cup Y_1$ for $Y_0$ and $Y_1$ which satisfy:
\begin{enumerate}
\item\label{it:Y0}
All branches of the $\cW$-components of both $\mathbf{s}$ and $\mathbf{r}$ are positive with respect to the same families of intervals in an open set containing $Y_0.$
\item\label{it:Y1}
For any $C\in Y_1$, there is a point $p\in\partial\Sigma \subseteq C$ such that all branches of $\mathbf{s}$ and $\mathbf{r}$ evaluate positively at $p.$
\end{enumerate}
($Y_1$ is a neighborhood of the \stronglypositive boundary strata.)  By shrinking $Y_+$ if needed, we assume that $\oPMr\setminus Y_+$ is a compact orbifold with corners.
Let $Y'_+$ be another open set satisfying the same properties, and $\overline{Y'_+}\subset Y_+.$

Take a finite set of multisections $w_j$ of $E_2\to\oPMr$, where $1 \leq j \leq m$ and $
\text{Supp}(w_j)\subseteq\oPMr\setminus Y_+'$, such that the multisections $w_j$ span every fiber of $E_2\to \partial\oPMr\setminus Y_+,$ that is for any choice of a local branch, for each $w_1,\ldots, w_m$, the branches span the fiber.  Assume, furthermore, that the multisection $w_j$ projects to zero in all but one direct summand of $E_2$ %(with respect to the decomposition in the statement of the lemma) 
and that $w_j$ is canonical if this direct summand is $\CL_i$, whereas $w_j$ is pulled back from the base (Definition \ref{def:canonical for Witten}) if this direct summand is $\cW$.  Such multisections exist by Proposition \ref{prop:a_single_section_for_Witten_or_L_i}, \eqref{it:single_Witten-span} and \eqref{it:single_L_i}, and the compactness of $\oPMr\setminus Y_+$.

Let $\pi : \oPMr\times [0,1] \to \oPMr$ be the projection, %onto the first factor. %Let $h\in C^\infty_m(\partial\oPMr,\pi^*E|_{\partial\oPMr})$ be the linear homotopy defined~by
%\[
%h(p,t) = (1-t)\mathbf{s}(p)+t\mathbf{r}(p),~\text{for }p \in \partial\oPMr, ~~t\in [0,1]
%\]
and set, for $\lambda \in \R_+^m$
\[
H_\lambda(p,t) = %h(p,t) 
(1-t)\mathbf{s}(p)+t\mathbf{r}(p)+ t(1-t)\sum \lambda_i w_i, ~p \in Y_+\cup\partial\oPMr, t\in [0,1] .
\]
By positivity, $H_\lambda$ is nowhere-vanishing on $Y_+'\times[0,1]$, and in particular transverse, for each $\lambda$. For $\lambda$ in a small enough neighborhood $N$ of $0 \in \R^m$, the multisection $H_\lambda(p,t)$ is also nonzero on $ Y_+\times[0,1]$.  Apply Theorem~\ref{thm: hirsch} to
\[
 M= \partial^0\oPMr\times \left(0,1\right),~E = \pi^*E|_{\partial^0\oPMr}, ~ V= N,~F=H.
\]
For $p\notin Y_+$, the derivatives of $H^{\ev}$ in directions tangent to $\partial\oCMr$ span the fiber $(E_1)_p$ whenever $\mathbf{s}_1(p)$ vanishes, by assumption \ref{it:id_and_trav} of our current lemma. Since the multisections $w_j$ span $(E_2)_p,$ the derivatives of $H^{\ev}$ in the $\R^m$-directions span $(E_2)_p$ for all $p \in \partial^0\oPMr\setminus Y_+.$ Thus, $H^{\ev}|_{\R_+^m\times\partial\oPMr}\pitchfork 0$, and by Theorem \ref{thm: hirsch}, the set $X = \{\lambda\in V|~H_\lambda \pitchfork 0\}$ is residual.

For any $\Gamma\in A,$ let $E_\Gamma \to \CM_{\CB\Gamma}$ be the bundle induced by $E$ on $\CM_{\CB\Gamma},$ so that
$F_\Gamma^* E_\Gamma = E.$
Write $
M_\Gamma = \CM_{\CB\Gamma} \times (0,1),~\text{and}~E_\Gamma = \pi_\Gamma^*E_\Gamma,
$
where $\pi_\Gamma:M_\Gamma \to\CM_{\CB\Gamma}$ is the projection.
%It follows from Observation~\ref{obs:extension_of_canonical_conds} and Remark~\ref{rmk:scc} that
By the definitions of the $\{w_j\}$ and of $A$, there exists $H^{\CB\Gamma}:\R_+^m\to C^\infty_m(E_\Gamma\to M_\Gamma)$ for any $\Gamma \in A$ such that
\[
H_\lambda|_{\CM_\Gamma \times (0,1)} = (F_\Gamma \times id_{(0,1)})^* H^{\CB\Gamma}_\lambda, ~ \lambda \in \R_+^m.
\]

Apply Theorem~\ref{thm: hirsch} with the same $V$, and
$
M = M_\Gamma, ~ E = E_\Gamma, ~ F = H^{\CB\Gamma}.
$ Since $\mathbf{s}_1 \pitchfork 0,$ then also $\mathbf{s}^{\CB\Gamma}_1\pitchfork 0$. Thus, the same argument that showed $H^{\ev}\pitchfork 0$ shows $(H^{\CB\Gamma})^{\ev}\pitchfork 0$.  It follows that $X_\Gamma = \{\lambda \in \R^m \; | \; H^{\CB\Gamma}_\lambda \pitchfork 0\}$ is residual.

Since $X$ and $X_\Gamma$ are residual, $X \cap \bigcap_{\Gamma \in A} X_\Gamma\neq\emptyset,$ and for any $\lambda\in X \cap \bigcap_{\Gamma \in A} X_\Gamma$ and any $\Gamma\in A$ such that $\CB\Gamma$ has no partially stable components, the homotopy $H^{\CB\Gamma}_\lambda$ does not vanish at $\CM_\Gamma\times[0,1]$ due to transversality and the dimension count of Observation \ref{obs:dim_of_base}. The homotopy $H_\lambda$ does not vanish on $Y\times[0,1]$ by of the choice of $N.$  Thus, $H_\lambda$ satisfies our requirements for any $V'_+\subset V_+\cap Y_+$ of the form \eqref{eq:V_+}, and the proof is complete.
\end{proof}
\begin{proof}[Proof of Theorem \ref{thm:int_numbers_well_defined}]
By Lemma \ref{lem:existence}, one can find a global canonical multisection $\mathbf{s}\in C_m^\infty(\oPMr,E)$ that does not vanish in $\partial\oPMr$ or on $U_+$ of the form \eqref{eq:U_+}.  Hence, $e\left(E;\mathbf{s}|_{\partial\oPMr\cup U_+}\right)$ can be defined.
If $\mathbf{r}$ is another canonical multisection whose domain contains $\partial\oPMr\cup U'_+,$ for $U'_+$ of the same form, then by Lemma \ref{lem:partial_homotopy} applied to $C=\partial^0\oPMr$, the trivial bundle $E_1$, and $V_+ =U_+\cap U'_+$, one can find a nowhere-vanishing homotopy between $\mathbf{s}$ and $\mathbf{r}$.  Thus, by Lemma \ref{lem:zero diff as homotopy}, we have
\[\int_{\oCMr}e\left(E;\mathbf{s}|_{\partial\oPMr\cup U_+}\right) = \int_{\oCMr}e\left(E;\mathbf{r}|_{\partial\oPMr\cup U'_+}\right),\]
which completes the proof.
\end{proof}

%%%%%%%%%%%%%%%%%%%%%%%%%%%%%%%%%%%%%%%%%%%%%%%%%%%%%%%%%%%%%%%%%%%%%%%

\appendix

\section{Multisections and relative Euler classes}\label{sec:ap_euler}

Our definitions for orbifolds with corners, operations between them, and orientations are those of \cite[Section 3]{Zernik}. Analogous definitions for manifolds with corners appear in \cite{Joyce}.
In short, an orbifold is given by a proper \'etale groupoid $M=(M_0,M_1,s,t,e,i,m)$ (Definition 22 in \cite{Zernik}), which is a category with objects $M_0$ and morphisms $M_1$, where both $M_0$ and $M_1$ are manifolds with corners.\footnote{More precisely, the category of orbifolds is a $2$-localization of the category of proper \'etale groupoids by a natural notion of refinement.  A similar comment holds for orbibundles.}  The source and target maps $s,t:G_1\to G_0$ take a morphism to its source and target, respectively.  The map $e:M_0\to M_1$ takes an object $x$ to its identity morphism $1_x$, the map $i:M_1\to M_1$ takes an element to its inverse, and $m$ is the composition of morphisms, denoted by $m(g,h)=gh$ whenever defined.  The maps $s,t,e,i,m$ are required to be e\'tale, and the map $s\times t:M_1\to M_0\times M_0$ is proper. The \emph{coarse space} or \emph{orbit space} is the quotient $|M|=M_0/M_1,$ with the quotient topology. The isotropy group of $x\in M_0$ is the group of elements $\gamma\in M_1$ with $s(\gamma)=t(\gamma)=x$.  It is well known that %an effective orbifold, meaning 
an orbifold such that the generic isotropy group is trivial, is completely determined from the local orbifold structure of a cover of its coarse space; see \cite[Proposition 5.29]{MoMr}.  The same holds for orbifolds with corners.

%Any point $p$ in the coarse space of an orbifold with corners has a local orbifold structure given by a quotient $U/G$ where $U$ is an open subset of $\R^n$ when $p$ is internal and of $\{(x_1,\ldots,x_n)|x_1,\ldots,x_k\geq 0\}\subset\R^n$ for a corner of depth $k,$ and $G$ is a finite group acting on $U$ by orientation preserving diffeomorphisms.

An orbibundle $E$ over an orbifold with corners $M$ is given by a vector bundle $\pi:E\to M_0$ and a fiberwise linear map
\[\mu: E\times_{M_0}M_1:=\{(e,\gamma)|\pi(e)=t(\gamma)\}\rightarrow E,\]
where we write $e \cdot \gamma$ for $\mu(e,\gamma)$.  It is required that 
\[
\pi(e\cdot\gamma)=s(\gamma),\quad e\cdot 1_{\pi(e)}=e,\quad\text{and}\quad (e\cdot\gamma)\cdot\delta = e\cdot (\gamma\delta)\] whenever the last equality makes sense; see \cite[Section 5]{MO1}. This data naturally defines an orbifold whose objects are $E,$ and the morphisms and other maps are defined using $\pi$ and $\mu$.  The isotropy group of $x$ acts on the fiber $E_x.$

%
%Similarly for an orbibundle $E\to M$, the neighborhood in the total space of any point $p$ which belongs to the coarse space of $M$ has a local structure $(U\times \R^m)/G,$ where $U,G$ are as before and $G$ preserves the projection.

Our discussion of multisections, Euler classes, and their relations follows \cite{Dusa}. Although there the underlying category is manifolds without boundary, while for us it is the category of manifolds with corners, the constructions and proofs are easily adapted to our case. We also refer the reader to the Appendix of \cite{PST14} for a short description of the analogous definitions and for manifolds with corners.
Inspired by \cite[Definition 4.13]{Dusa}, %let $\mathcal{Q}_{\geq 0}$ be the category whose elements are the non negative rational and there are only identity morphisms.
we define:
\begin{definition}\label{def:multisection}
A \emph{multisection} of an orbibundle $E\to M$ over an orbifold with corners is a function
$\LLL: E\to \mathbb{Q}_{\geq 0}%\mathcal{Q}_{\geq 0}
$ that satisfies:\begin{enumerate}
\item for all $e=(x,v)\in E,$ and $\gamma\in X_1$ with $x=t(\gamma),$  $\LLL(x,v) = \LLL((x,v)\cdot \gamma)$;
\item for each $x\in M_0$ there is an open neighborhood $U$, a nonempty finite set of smooth local sections $s_i:U\to E,~i=1,\ldots, N$ called \emph{local branches}, and numbers $\mu_1,\ldots, \mu_N\in\mathbb{Q}_{>0}$ called \emph{weights}, such that $\sum_{v\in E_y}\LLL(y,v) = 1$ for all $y\in U$ and $\LLL(y,v) = \sum_{i\; |\; s_i(y)=v}\mu_i$~for~all~${v\in E_y}$.
\end{enumerate}
The triple $(U,\{s_i\},\mu_i)$ is called a \emph{local structure} for the multisection. We denote by $C_m^\infty(E)$ the collection of multisections of $E.$
\end{definition}

When $N=1$ for all $x$, the multisection is just a section in the category of orbibundles over orbifolds with corners.  Any section descends to a map from the coarse space of $M$ to that of $E$. We work with multisections because in the orbifold category, sections may be constrained to vanish at some points. Additional vanishing constraints come from canonicality, see \cite[Remark 3.5]{PST14}.

We write $Z(s)=\{p \; |\; \LLL(p,0)\neq 0\}$ for the zero locus of a multisection $s$.
A multisection $s$ is \emph{transverse to zero} if every local branch is transverse to the zero section; in this case, we write $s\pitchfork 0.$

%\subsection{Operation on multisections}
Linear operations multisections are defined in \cite[Page 38]{Dusa} via \[(\LLL_1+\LLL_2)(x,v)=\sum_{v_1+v_2=v}\LLL_1(x,v_1)\LLL_2(x,v_2),~~\lambda\LLL(x,v)=\LLL(x,\lambda v).\]
These operations define a vector space structure on the set of multisections. When a finite group $G$ acts on $E\to M$, an induced action on multisections is defined via $(g\cdot\LLL)(x,v)=\LLL(g^{-1}(x,v)).$ The \emph{symmetrization} of $\LLL$ with respect to $G$ is defined in \cite[Definition A.10]{PST14} as the $G$-invariant multisection
\begin{equation}\label{eq:average}
\LLL^G(x,v) = \frac{1}{|G|}\sum_{g\in G}g\cdot\LLL(x,v).
\end{equation}

\begin{nn}\label{def:union_of_sections}
If $\LLL_1,\ldots, \LLL_m$ are multisections of an orbifold vector bundle $E$ on an orbifold $M$, we write 
\[
\uplus_{i=1}^m \LLL_i(x,v) := \frac{1}{m}\sum\nolimits_{i=1}^m\LLL_i(x,v).
\]
\end{nn}

The \emph{support} of a multisection is the support of $\LLL$, which is locally the union of the images of the local branches $s_i$.  We often denote a multisection $\LLL$ by its support $s$.  A multisection descends to a function on the coarse space of $E$, considering $E$ as an orbifold with corners, so we sometimes use the term ``support" to refer, by slight abuse of terminology, to the support on the coarse bundle.
If $s_i$ is the support of $\LLL_i$, then $\uplus_{i=1}^m s_i(x)$ denotes $\uplus_{i=1}^m \LLL_i(x,v)$.  In particular, the support of $\uplus_{i=1}^m s_i(x),$ as a set, is the union of the supports of the $s_i$.  Under this convention, we have, for any multisection $s$,
\begin{equation}\label{eq:union_of_sections}
\uplus_{i=1}^m s=s.
\end{equation}

%\subsection{Homology Euler class}
%We now very roughly sketch the definition of the homology Euler class and its connection to the zero locus of multisection. We do not write the full construction as it is completely analogous to the construction in \cite{Dusa}, which itself is quite long. Another source for the non relative case is \cite[Section 2]{TehFuk}, while \cite[Appendix A]{PST14} treats the relative case, but for manifolds with boundary rather than orbifolds with boundary.
\begin{rmk}\label{rmk:consistency_etc}
The boundary $\partial X$ of a manifold with corners $X$ is itself a manifold with corners, equipped with a %(not necessarily injective)
map $i_X:\partial X\to X$.  We extend this map to $i_X:X\sqcup\partial X\to X$ by the identity.  If $U \subseteq X$ and $V \subseteq \d X$, then we say that a section $s$ of a bundle over $U\sqcup V$ is \emph{consistent} if $s(x)=s(y)$ whenever $i_X(x)=i_X(y),$ where the equality is obtained using the natural identifications of the fibers $E_x$ and $E_y$.  Consistent multisections are similarly defined.\footnote{More precisely, part of the information of a vector bundle $E\to X$ in the category of vector bundles over manifolds with corners is a vector bundle $\partial E\to\partial X$ together with an identification $\widetilde{i}_X:(\partial E)_x\to i_X^*E_x$.  The map $\widetilde{i}_X$ extends by the identity map to $\widetilde{i}_X:E\sqcup\partial E\to i_X^*E.$ We use these natural identifications to define the equality of $s(x)$ and $s(y)$.}

Throughout this paper, we omit the maps $i_X$ from the notation and identify $x$ with $i_X(x)$.  By a section or a multisection, we always mean a consistent section or multisection. In particular, canonical multisections are automatically consistent, as well as all other multisections constructed in this paper.
\end{rmk}
Let $E\to M$ be an orbibundle over an orbifold with corners such that $\rk E=\dim M$.  Assume that $E$ and $M$ are oriented, and, while $|M|$ need not be compact, assume that it has finitely many connected components.

Suppose $U\subseteq |M|$ is an open subspace such that $|M|\setminus U$ is a compact orbifold with corners.
% with compact complement and $V=\partial (|M|\setminus U).$
Let $s$ be (the support of) a multisection of $|E|\to U\sqcup \partial |M|$, which is, by definition, induced from a multisection of $E$ on the preimage of $U\sqcup \partial |M|$ in $M_0$. Suppose that $s$ vanishes nowhere on $U\sqcup \partial |M|.$ %Extend $s$ to a global transverse multisection $\tilde s$ of $E\to M.$
Let $\tilde s$ be an extension of $s$ to a multisection of $E\to M$ with isolated zeroes.
For any $|p|\in |M|$ in the zero locus of $s,$ let $p\in M_0$ be a preimage of $|p|,$ and let $(U,\{s_i\}_{i=1}^N,\{\mu_i\}_{i=1}^N)$ be a local structure for $s$ at $p.$ By shrinking $U$ if necessary, assume that the only zero of any local branch $s_i$ in $U$ may have is at $p$, and that $E|_U$ is trivialized as $U\times \R^n$; then, give the fiber $\R^n$ the orientation induced from $E$.

Let $\deg_{p}(s_i)$ be the degree of vanishing of $s_i$ at $p$: It is zero if $s_i(p)\neq 0$; otherwise let $S$ be the boundary of a small ball $B\subset U$ containing $p$, with the orientation induced from the orientation of $B\subset M$.  Let $S'$ be the $L^2-$unit sphere in $\R^n$ with the orientation induced from the unit ball of $\R^n$.  We define $\deg_p(s_i)$ as the degree of the map $f:S\to S'$ given by $x\to \frac{s_i(x)}{|s_i(x)|_{L^2}}.$
\begin{nn}\label{nn:weighted_signed}
Under the above assumptions, we define the \emph{weight} of $|p|$~as
\[
\eps_{|p| }=\frac{1}{|G|}\sum\nolimits_{i=1}^N \mu_i\deg_p(s_i), 
\] where $G$ is the isotropy group of $p.$ %$|p|.$
We define, for a global multisection $\tilde s$, the \emph{weighted cardinality} of $Z(\tilde s)$ as
\[\#Z(\tilde s) = \sum\nolimits_{|p|\in |M|}\eps_{|p|}.\]
\end{nn}
One can show that $\text{deg}_{p}(s_i)$ and $\eps_{|p|}$ depend only on the orientations of $E,M.$ In fact, it suffices to choose a relative orientation for $E\to M,$ rather than orienting both $E$ and $M.$
%When $\tilde s$ is transverse and all zeroes of $\tilde s$ have trivial isotropy
%We use the following theorem.
\begin{thm}\label{prop:euler_as_zero_locus}
Let $M$, $E$, and $s$ be as above, and let $\tilde s$ be a global extension with isolated zeroes. Then $\#Z(\tilde s)$ and the homology class $\sum \eps_{|p|}[|p|]\in H_0(|M|;\mathbb{Q})$ depend only on $s$ and not on $\tilde s.$ Changing the orientations of $E$ and $M$ but preserving the relative orientation does not change the class or the count. This homology class is therefore Poincar\'e dual to a relative cohomology class with compact support $H_c^n(|M|,\partial |M|;\mathbb{Q})$, which is called the \emph{relative Euler class}, and by definition satisfies,
\[\int_{|M|}e(E,s)=\#Z(\tilde s).\]
When $\partial|M|=\emptyset$, the relative Euler class is the Euler class of the orbibundle.
\end{thm}
%Some simple examples of calculations can be found in \cite{TehFuk}, Examples 2.6.5, 2.6.6.
The independence of choices follows from a standard cobordism argument, which is easily modified to the orbifold case. The fact that the weighted zero count of a transverse multisection is dual to the Euler class in the case of no boundary is a special case of the results of \cite{Dusa}. It extends without difficulty to our setting. A final useful observation is the following:
\begin{obs}\label{obs:smaller_U}
Let $E,M,U$, and $s$ be as above.  Let $U'\subseteq U$ be an open set such that $|M|\setminus U'$ is a compact orbifold with corners, and $s' = s|_{\partial |M|\cup U'}$.  Then a transverse extension of $s$ is also a transverse extension of $s'.$ Thus,
\[\int_{|M|} e(E ; s') =\int_{|M|} e(E ; s) \in \mathbb{Q}%H_*(M;\mathbb{Q})
.\]
\end{obs}

\bibliographystyle{abbrv}
\bibliography{OpenBiblio}

\end{document}